%% file: ADmSOS.tex
\documentclass[sn-mathphys-num]{sn-jnl}
\usepackage{geometry}
\usepackage[utf8]{inputenc}
\usepackage{amssymb}

\usepackage{bm}
\usepackage[normalem]{ulem}
\usepackage[justification=centering]{caption}
\usepackage{subcaption}
\usepackage{optidef}

\usepackage{stanli}
\usepackage{adjustbox}
\usepackage{graphicx}
\usepackage{multirow}
\usepackage{multicol}
\usepackage{hhline}
\usepackage{adjustbox}
\usepackage{amssymb}
\usepackage{amsthm}
\usepackage{thmtools}
\usepackage{nicematrix}
\usepackage{booktabs}
\usepackage[mathscr]{euscript}

\newcommand{\stkout}[1]{\ifmmode\text{\sout{\ensuremath{#1}}}\else\sout{#1}\fi}

\newtheorem{proposition}{Proposition}
\newtheorem{lemma}{Lemma}
\newtheorem{remark}{Remark}
\newtheorem{example}{Example}

\newtheorem{theorem}{Theorem}
\newtheorem{corollary}{Corollary}

\newtheorem{assumption}{Assumption}

 \usepackage{xcolor}
\usepackage{hyperref}
  \hypersetup{
  colorlinks,
  citecolor=blue,
  linkcolor=blue,
  urlcolor=blue}

\newcommand{\I}{I}
\newcommand{\Ik}{\I_k}
\newcommand{\hI}{\hat{I}}
\newcommand{\hIk}{\hat{I}_k}
\newcommand{\Iell}{\I_\ell}
\newcommand{\Ikell}{\I_{k,\ell}}
\newcommand{\hIkell}{\hat{\I}_{k,\ell}}
\newcommand{\hSigmak}{\hat{\Sigma}[\x,\I_k]}
\newcommand{\hRk}{\hat{\Rbb}[\x,\I_k]}
\newcommand{\G}{\bm{G}}
\newcommand{\A}{\bm{A}}
\newcommand{\B}{\bm{B}}
\newcommand{\bb}{\bm{b}}
\newcommand{\C}{\bm{C}}
\newcommand{\D}{\bm{D}}
\newcommand{\M}{\bm{M}}

\newcommand{\x}{\bm{x}}
\newcommand{\y}{\bm{y}}
\newcommand{\Rbb}{\mathbb{R}}
\newcommand{\N}{\mathbb{N}}
\newcommand{\Sbb}{\mathbb{S}}

\newcommand*\squared[1]{\tikz[baseline=(char.base)]{
		\node[shape=rectangle,draw,inner sep=0pt, minimum size=4mm] (char) {#1};}}

\newcommand{\fullciteay}[1]{\citeauthor{#1} (\citeyear{#1})}

\title{Moment-SOS hierarchies for arrow-type polynomial matrix inequalities with applications to structural optimization}

\author*[1]{\fnm{Marouan} \sur{Handa}}\email{handa@utia.cas.cz}

\author[1,2]{\fnm{Marek} \sur{Tyburec}}\email{marek.tyburec@cvut.cz}

\author[3]{\fnm{Giovanni} \sur{Fantuzzi}}\email{giovanni.fantuzzi@fau.de}

\author[4]{\fnm{Victor} \sur{Magron}}\email{vmagron@laas.fr}

\author[1,5]{\fnm{Michal} \sur{Ko\v{c}vara}}\email{m.kocvara@bham.ac.uk}

\affil*[1]{ \orgname{Institute of Information Theory and Automation}, \orgaddress{\street{Pod Vodárenskou věží 4}, \city{Prague}, \postcode{18208}, \country{Czech Republic}}}

\affil[2]{ \orgname{Faculty of Civil Engineering, Czech Technical University}, \orgaddress{\street{Thákurova 7}, \city{Prague}, \postcode{16629}, \country{Czech Republic}}}

\affil[3]{\orgname{Department of Mathematics, FAU Erlangen-Nürnberg}, \orgaddress{ \city{Erlangen}, \postcode{91058},\country{Germany}}}

\affil[4]{\orgname{LAAS-CNRS}, \orgaddress{\street{7 avenue du colonel Roche}, \city{Toulouse}, \postcode{31400}, \country{France}}}

\affil[5]{\orgname{School of Mathematics, University of Birmingham}, \orgaddress{\street{B15 2TT}, \city{Birmingham}, \postcode{610101},\country{UK}}}

\begin{document}

\abstract{ 
The Arrow Decomposition (AD) technique, initially introduced in \fullciteay{kovcvara2021decomposition}, demonstrated superior scalability over the classical chordal decomposition in the context of Linear Matrix Inequalities (LMIs) if the matrix in question satisfied suitable assumptions. The primary objective of this paper is to extend the AD method to address Polynomial Optimization Problems (POPs) involving large-scale Polynomial Matrix Inequalities (PMIs), with the solution framework relying on moment-sum of square (mSOS) hierarchies. As a first step, we revisit the LMI case and weaken the conditions necessary for the key AD theorem presented in \fullciteay{kovcvara2021decomposition}. This modification allows the method to be applied to a broader range of problems. Next, we propose a practical procedure that reduces the number of additional variables, drawing on physical interpretations often found in structural optimization applications. For the PMI case, we explore two distinct approaches to combine the AD technique with mSOS hierarchies. One approach involves applying AD to the original POP before implementing the mSOS relaxation. The other approach applies AD directly to the mSOS relaxations of the POP. We establish convergence guarantees for both approaches and prove that theoretical properties extend to the polynomial case. Finally, we illustrate the significant computational advantages offered by the application of AD, particularly in the context of structural optimization problems.}

\noindent 



\maketitle
\section{Introduction}

 Global minimization of multivariate polynomials subject to polynomial inequalities (POP) is a fundamental mathematical optimization problem. While generally NP-hard, POPs can be solved using the moment-sum of squares (mSOS) hierarchy. Relying on representation of positive polynomials on basic semialgebraic sets \cite{putinar1993positive}, the mSOS hierarchy generates increasingly tighter linear semidefinite programming (SDP) relaxations of increasing size, and hence a non-decreasing sequence of lower bounds for POP. Under a mild assumption called the Archimedean assumption \cite{lasserre2001global}, these relaxations converge asymptotically to the global minimum of POP. In addition, the convergence often occurs generically in a finite number of steps \cite{nie2014optimality}. 
 
However, the relaxations can result in large size of SDP matrices and thus the mSOS does not scale well for problems with large number of variables and/or for problems that require high relaxation degrees. To improve the scalability one can exploit sparsity in the monomial coefficients of involved polynomials, which leads to correlative sparsity \cite{lasserre2006convergent}, term sparsity \cite{wang2021tssos,wang2020chordal}, their combination \cite{wang2022cs}, and ideal sparsity \cite{korda2024exploiting}. 
For more general exploitation of symmetry invariance under finite groups, see \cite{riener2013exploiting}. For more details on correlative/term sparsity and their multiple applications, we refer  the interested reader to the recent monograph \cite{magron2023sparse}. 
Convergence rates for hierarchies exploiting correlative sparsity have been obtained in \cite{korda2024convergence}.

 
While POPs with scalar constraints are well studied, many applications require polynomial matrix inequalities (PMIs), i.e., matrices whose entries are multivariate polynomials (see \cite{henrion2005solving,tyburec2021global,bondar2022recovering}).  Assuming that these matrices are dense, analogous convergent hierarchies to the scalar case were developed in \cite{henrion2006convergent} based on the representation of positive definite polynomial matrices \cite{scherer2006matrix}. 
As in the scalar case, scalability can be improved by exploiting the sparsity occurring in the polynomial nature of the PMI. This includes correlative sparsity and term sparsity as shown in the very recent works \cite{miller2024sparse, handaterm}. Contrary to the scalar case, the PMI settings also allows one to exploit structural sparsity of the matrix entries through, e.g., chordal sparsity  \cite{zheng2023sum,zheng2018decomposition}.

As an alternative to chordal structural sparsity, the concept of arrow decomposition (AD) method has recently been developed to speed up the solution of problems subject to linear matrix inequalities (LMIs) with matrices having an arrow structure \cite{kovcvara2021decomposition}. The arrow structure naturally appears in linear elliptic self-adjoint variational problems discretized by finite elements. In these problems, the equilibrium equation can be reformulated either as an LMI or as a PMI involving a positive semidefinite stiffness matrix, which takes an arrow form due to the Schur complement \cite{tyburec2021global}. Rather than requiring chordal sparsity of these LMIs, the AD method requires positive semidefiniteness of the top-left block and bottom-right block of the LMI, making it particularly suitable for such problems by requiring fewer additional variables than the classical chordal decomposition method. Thus, it maintains a better scalability of the decomposed problems with a linear growth in complexity (CPU time vs number of variables), unlike the chordal decomposition wwhich typically scales worse than cubic. In addition, in \cite{kovcvara2021decomposition}, a connection between AD method applied to topology optimization problems and the domain decomposition method was provided. This connection allows one, on the one hand, to physically interpret the additional variables. On the other hand, techniques from domain decomposition can be directly used to formalize the matrix decomposition of the LMI.

\paragraph{Contribution and organization of the manuscript.} So far, the AD method has been restricted to linear matrix inequality (LMI) problems, constrained to cases where both the top-left and bottom-right blocks of the main matrix are positive definite. Furthermore, it has not been applicable to degenerate problems—that is, cases in which at least one of the decomposed matrices is rank-deficient. This paper makes three main contributions. First, we extend arrow decomposition to PMI problems. Second, we provide more general results for the linear case that allow one treatment of degenerate problems. Third, we demonstrate significant computational benefits through applications in structural optimization.

The paper is structured as follows. We begin in Section \ref{sec:background} by introducing all necessary definitions and notation and recalling the moment sum-of-squares method for the PMI case. Section \ref{sec: ADLMI} is dedicated to the AD method in the linear case: We present the AD method and extend the results from \cite{kovcvara2021decomposition} by providing weaker conditions for the applicability of the AD method. This allows for a proper treatment of zero-stiffness design elements in topology optimization problems. Furthermore, we demonstrate that AD typically leads to degenerate problems without an interior point, i.e., violating Slater condition. To fix this, we propose a projection-based method to postprocess resulting SDPs. Remarkably, this post-processing also reduces the number of additional variables and the size of matrix inequalities. 

In Section \ref{sec:ADMSOS}, we present our main results on extending AD theory to the PMI case, combining it with the mSOS method. We first prove that, analogously to the linear case, AD can be applied directly to POP problems ("prior to mSOS”). However, because this approach increases the number of variables in the polynomial basis, it typically does not improve scalability. Its presentation is nonetheless insightful and serves as a bridge to the subsequent approach. To address scalability, we show that AD can also be applied at the relaxation level ("posterior to mSOS”) and establish a corresponding convergence result. In addition, we show that the decomposed relaxations conserve the post-processing results from the linear case. We finally show that AD applied posterior to mSOS can be seen as special instance of AD applied prior to mSOS, allowing to perform AD prior to mSOS without including additional variables in the polynomial basis while maintaining the same theoretical convergence. 

Next, we apply the developed theory to topology optimization problems of frame structures. In particular, in Section \ref{sec:frames}, we present compliance and weight optimization problems in frame structure design and provide, through examples, a physical interpretation of the additional variables in the context of structural engineering as well as of the post-processing procedure. In Section \ref{sec:numerics}, we provide numerical illustrations demonstrating the benefits of AD: reduced solution times and improved numerical accuracy of the SDP relaxation. Moreover, the detailed examples in Sections \ref{sec:frames} and \ref{sec:numerics} aim to make the techniques presented in this paper more accessible to readers from the structural engineering community. Finally, we summarize our contributions and suggest directions for future research in Section \ref{sec:conclusion}.


\section{Notations and background }
\label{sec:background}
In this section, we first introduce all the notations needed for this paper. We further provide the definition of polynomial optimization problems constrained by polynomial matrix inequalities (PMIs), followed by a brief overview of Lasserre's moment sum-of-squares (mSOS) hierarchy. 
\subsection{Notation and definitions}
\label{sec:notation}
All the vectors considered in this paper are row vectors. For a given integer $n$, we denote by $[n]$ the set $\lbrace 1,\cdots, n \rbrace$. 
We denote by $\vert \bm{B} \vert $ the dimension of a given real square matrix $\bm{B}$, i.e. if $\bm{B} \in \mathbb{R}^{n \times n}$ then $\vert \bm{B} \vert =n$. We use the same notation for vector lengths and set cardinalities. 
We denote by $(\bm{B})_{ij}$ the $(i,j)$th component of $\bm{B}$. 
We denote by $\bm{0}_{\mathbb{R}^{n \times m}}$ (resp. $\bm{I}_{\mathbb{R}^{n}}$) the $n \times m$ zero matrix (resp. the identity matrix of size $n$ ). When the dimensions are clear, we simply use $\bm{0}$ (resp. $\bm{I}$). Let $\mathbb{S}^n$ be the space of $n \times n$ symmetric matrices and let $\mathcal{I} \subset [n]$. For $\bm{A} \in \mathbb{S}^n$, we denote by $\bm{A}_{\mathcal{I}}$ the sub-matrix of $\bm{A}$ whose rows and columns are indexed by the elements in $\mathcal{I}$. We denote by $\bm{A}^\dagger$ the Moore-Penrose inverse of $\bm{A}$ \cite[Lemma 14.1]{gallier2011geometric}. Recall that by definition we have $\bm{A}\bm{A}^\dagger\bm{A}=\bm{A}$ and $\bm{A}^\dagger\bm{A}\bm{A}^\dagger=\bm{A}^\dagger$. 
Given $\bm{A}, \bm{B}  \in \mathbb{S}^n$, the notation $\bm{A}\succeq \bm{B}$ (resp.~$\bm{A} \succ \bm{B}$) stands for $\bm{A} - \bm{B}$ being positive semidefinite (resp.~definite). 
Given a vector $\bm{y} \in \Rbb^m$ and matrices $\A_0, \A_1, \dots, \A_m$ a relation of the form $\A_0 + \sum_{i=1}^m y_i \A_i \succeq 0$ is called a \textit{Linear Matrix Inequality} (LMI).

Let $\bm{A}=[a_{ij}]  \in \mathbb{R}^{n \times m}$ and $ \bm{B}=[b_{rs}] \in \mathbb{R}^{p \times q}$. The Kronecker product of $\bm{A}$ and $\bm{B}$, denoted as $\bm{A} \otimes \bm{B}$, is a matrix in $\mathbb{R}^{np \times mq}$, defined element-wise by: $(\bm{A} \otimes \bm{B})_{uv}=a_{ij}b_{rs}$, such that $u=(i-1)p+r$ and $v=(j-1)q+s$, for $i \in [n], j \in [m], r \in [p]$, and $ s \in [q]$. We recall here some key properties of the Kronecker product which will be used in this paper: $\bm{A} \otimes (\bm{B}+\bm{C}) =\bm{A} \otimes \bm{B}+\bm{A}\otimes \bm{C}$, 
$(\bm{A} \otimes \bm{B})\otimes \bm{C}=\bm{A} \otimes (\bm{B}\otimes \bm{C})$, 
$(\bm{A} \otimes \bm{B}) (\bm{C} \otimes \bm{D})=(\bm{AC}) \otimes (\bm{BD})$, $(\bm{A} \otimes \bm{B})^\dagger=\bm{A}^\dagger \otimes \bm{B}^\dagger$, $(\bm{A} \otimes \bm{B})^T=\bm{A}^T \otimes \bm{B}^T $. Additionally, for $\bm{A} \in \mathbb{R}^{n \times m}$ and $\bm{B} \in \mathbb{R}^{p \times q}$, the Kronecker product can be expressed as: $ \bm{A} \otimes \bm{B}=(\bm{I}_n \otimes \bm{B})(\bm{A} \otimes \bm{I}_q)=(\bm{A} \otimes \bm{I}_p)(\bm{I}_m \otimes \bm{B})$. We denote by $\text{Vec}$ the column-wise vectorization operator that transforms a matrix $\bm{A} \in \mathbb{R}^{n \times m}$ into a vector $\text{Vec}(\bm{A}) \in \mathbb{R}^{nm \times 1} $ by stacking the columns $\bm{A}$ on top of each other. We recall the following property of vectorization operator with Kronecker product: $\text{Vec}(\bm{ABC})=(\bm{C}^T\otimes \bm{A})\text{Vec}(\bm{B})$. 

Further,  let $\x \in \Rbb^{n_x}$. A real matrix polynomial $\bm{G}$ in the variable $\bm{x}$ is a matrix function of the form $\bm{G}(\bm{x})= \sum_{\bm{\alpha} \in \mathbb{N}^{n_x}} \bm{G}_{\alpha}\bm{x^\alpha}$, with each $\bm{G}_{\alpha} \in \Rbb^{n \times m}$.
Moreover, we denote by $\Rbb[\x]:= \lbrace \sum_{\bm{\alpha} \in \mathbb{N}^{n_x}} g_{\alpha}\bm{x^\alpha}:  g_{\alpha} \in \Rbb \rbrace $ the set of scalar real valued polynomials, by $\Rbb^{n\times m}[\x]:= \lbrace \sum_{\bm{\alpha} \in \mathbb{N}^{n_x}}\bm{G}_{\alpha}\bm{x^\alpha}:  \bm{G}_{\alpha} \in \Rbb^{n\times m} \rbrace $ the set of polynomial matrices of size $n \times m$, by $\Sbb^{n}[\x]:= \lbrace \sum_{\bm{\alpha} \in \mathbb{N}^{n_x}}\bm{G}_{\alpha}\bm{x^\alpha}:  \bm{G}_{\alpha} \in \Sbb^{n} \rbrace$ the set of symmetric polynomial matrices of size $n$, by $\mathbb{N}^{n_x}_d:=\lbrace \bm{\alpha} \in \mathbb{N}^{n_x}: \sum_i\alpha_i \leq d \rbrace$ the set of $n_x$-tuples of natural numbers for which the sum is bounded by a degree $d$, by $\Rbb[\x]_d:= \lbrace \sum_{\bm{\alpha} \in \mathbb{N}^{n_x}_d} g_{\alpha}\bm{x^\alpha}:  g_{\alpha} \in \Rbb \rbrace $ the set of scalar polynomials of degree at most $d$, by $\Sbb^{n}[\x]_d:= \lbrace \sum_{\bm{\alpha} \in \mathbb{N}^{n_x}_d}\bm{G}_{\alpha}\bm{x^\alpha}:  \bm{G}_{\alpha} \in \Sbb^{n}  \rbrace$, the set of symmetric polynomial matrices of size $n$ and of degree at most $d$. A relation of the form $\bm{G}(\bm{x}) \succeq 0$ where $\bm{G} \in \mathbb{S}^{n}[\bm{x}]$ is called a \textit{Polynomial Matrix Inequality} (PMI).

We say that a sequence $\bm{y}=(y_{\bm{\alpha}})_{\bm{\alpha} \in \mathbb{N}^{n_x}}$ has a representing measure if there exists a finite Borel measure $\mu$ supported on $\mathcal{K} \subset \mathbb{R}^{n_x}$ such that $y_{\bm{\alpha}}=\int_{\mathcal{K}}\bm{x^\alpha}d \mu$ for all $\bm{\alpha} \in \mathbb{N}^{n_x}$. In this case we say that $\bm{y}$ is a sequence of moments. For $\bm{x} \in \mathbb{R}^{n_x}$, we denote by $\bm{b}_d(\bm{x})$ the standard monomial basis of the vector space $\mathbb{R}[\bm{x}]_d$, i.e.
$$ \bm{b}_d(\bm{x})=(\bm{x}^{\bm{\alpha}})_{\bm{\alpha} \in \mathbb{N}_{d}^{n_x}}=\begin{pmatrix}  1& x_1 & x_2 & \cdots & x_1^2 & x_1x_2 & \cdots & x_1^d & \cdots & x_n^d \end{pmatrix}.$$
Let $\bm{y}=(y_{\bm{\alpha}})_{\bm{\alpha} \in \mathbb{N}^{n_x}}$, $p(\bm{x})=\sum_{\bm{\alpha} \in \mathbb{N}^{n_x}}p_{\bm{\alpha}}\bm{x^\alpha} \in \mathbb{R}[\bm{x}]$ and $ \bm{G}(\bm{x})=\sum_{\bm{\alpha} \in \mathbb{N}^{n_x}}\bm{G}_{\bm{\alpha}}\bm{x^\alpha} \in \mathbb{S}[\bm{x}]$. 
Let us further define the linear operator $\mathscr{L}_{\bm{y}}(p):=\sum_{\bm{\alpha} \in \mathbb{N}^{n_x}}p_{\bm{\alpha}}y_{\bm{\alpha}}$. In general, for any $\bm{P} \in \mathbb{R}^{n \times m}[\bm{x}]$, $ \mathscr{L}_{\bm{y}}(\bm{P}(\x)) $ is applied entry-wise on the matrix $\bm{P}(\x)$. Moreover, we define
\begin{itemize}
    \item the $d-$th order \textit{pseudo-moment matrix} associated with $\bm{y}$ by
 $$ \bm{M}_d(\bm{y})=\mathscr{L}_{\bm{y}} \left( \bm{b}_d(\bm{x})^T\bm{b}_d(\bm{x})  \right)=\left(y_{\bm{\alpha} +\bm{\beta}} \right) _{\bm{\alpha}, \bm{\beta}};$$
    \item the $d-$th order \textit{localizing matrix} \footnote{Notice that $\bm{M}_d(\bm{G}\bm{y})$ is a block matrix.} associated with $\bm{y}$ and $\bm{G}$ by $$\bm{M}_d(\bm{G}\bm{y})=\mathscr{L}_{\bm{y}} \left( \bm{G}(\bm{x}) \otimes \bm{b}_{d}(\bm{x})^T\bm{b}_{d}(\bm{x}) \right)=\left(\sum_{\bm{\gamma} \in \mathbb{N}^n_s} y_{\bm{\alpha} +\bm{\beta}+\bm{\gamma}}  \bm{G}_{\bm{\gamma}} \right) _{\bm{\alpha}, \bm{\beta} }.$$
\end{itemize}

\subsection{The moment sum-of-squares hierarchy }
\label{sec:mSOS}


We consider the following polynomial optimization problem (POP)
\begin{equation}
\tag{POP} p^*=\underset{\substack{\x \in \mathcal{K}} }{\min}~ p(\x),      
  \label{POP-PMI}  
\end{equation}
where $p \in \mathbb{R}[\bm{x}]_d$, $\bm{G} \in \mathbb{S}^{s}[\bm{x}]_{d_{G}}$ and $\mathcal{K}:=\lbrace \bm{x} \in \mathbb{R}^{n_x}: \bm{G(x)} \succeq 0\rbrace$ is a compact set. Notice that this setting does not limit itself to a single constraint: a collection of constraints $\bm{G}_j(\bm{x}) \succeq 0, j \in [m]$, can be concatenated into a single constraint $\bm{G}(\bm{x}) \succeq 0$ by forming a block diagonal matrix $\bm{G}(\bm{x}) = \text{diag}(\bm{G}_1(\bm{x}),\ldots,\bm{G}_m(\bm{x}))$. On the other hand, for $s=1$, $\bm{G(x)} \geq 0$ constitutes a polynomial scalar inequality. Finally, here we consider the matrix $\bm{G}$ in a general context, without assuming any specific arrow structure. The combination of mSOS with arrow type matrices will be discussed in Section \ref{sec:ADMSOS}.

Let $\mathcal{M}(\mathcal{K})$ be the space of finite signed Borel measures on $\mathcal{K}$, and let $\mathcal{M}(\mathcal{K})_+ \subset \mathcal{M}(\mathcal{K})$ denote the convex cone of nonnegative finite Borel measures on $\mathcal{K}$. If $p^*$ is a global minimum to \eqref{POP-PMI}, then
\begin{equation}
\begin{aligned}
  p^*=&\underset{\substack{\mu \in \mathcal{M}(\mathcal{K})_+} }{\inf} \int _\mathcal{K}p d \mu,&  \\
 \text{s.t. }&\mu(\mathcal{K})=1.
\end{aligned}
\label{generalised-moment-pblm}
\end{equation}
One can write \eqref{generalised-moment-pblm} in terms of the sequence $\bm{y}$ as 
\begin{equation}
\begin{aligned}
  p^*=\ &\underset{\substack{\y}}{\inf} \ \mathscr{L}_{\bm{y}}(p)&  \\
 \text{s.t. }&\bm{y}_0=1,& \\
 &\bm{y} \text{ has a representing measure } \mu \text{ on } \mathcal{K}.
\end{aligned}
\label{generalised-moment-pblm-Vectors}
\end{equation}

In practice, one can only solve finite dimensional truncations of \eqref{generalised-moment-pblm-Vectors}. To show convergence of a sequence of truncations, let us first recall that a matrix $\bm{G} \in \mathbb{S}^s[\bm{x}] $ is sum-of-squares (SOS) if there exists $\bm{H} \in \mathbb{R}^{s \times \ell}[\bm{x}]$ such that $ \bm{G}(\bm{x})= \bm{H}(\bm{x}) \bm{H}(\bm{x})^T.$ Further, let us assume the following algebraic compactness condition:

\begin{assumption}[Archimedean condition]\label{assumption-archimedean}
There exist an SOS polynomial $p_0 : \mathbb{R}^{n_x}  \rightarrow \mathbb{R}$ and a matrix SOS polynomial $\bm{R}: \mathbb{R}^{n_x}  \rightarrow \mathbb{S}^N$ such that the set $\big\lbrace \bm{x} \in \mathbb{R}^n : p_0(\bm{x})+  \left\langle \bm{R}(\bm{x}), \bm{G}(\bm{x}) \right\rangle \geq 0 \big\rbrace$ is compact.
\end{assumption}

For this setting, the following convergence result then holds:
\begin{theorem}[Dual facet of Putinar's Positivstellensatz \texorpdfstring{\citep[Theorem~2.44]{lasserre2015introduction}}{}]
 If Assumption \ref{assumption-archimedean} holds, then the sequence $\bm{y}=(y_{\alpha})_{\bm{\alpha \in \mathbb{N}^{n_x}}}$ has a representing measure $\mu$ supported on $\mathcal{K}$ if and only if 
 \begin{equation}
     \bm{M}_d(\bm{y}) \succeq 0 \text{ and } \bm{M}_d(\bm{G}\bm{y}) \succeq 0, \forall d \in \mathbb{N}.
     \label{moment-localizing-matrix}
 \end{equation}
 \label{putinar}
\end{theorem}
Hence, we obtain the following equivalent infinite-dimensional linear program 
\begin{equation}
\begin{aligned}
  p^*= \ &\underset{\substack{\y}}{\inf} \ \mathscr{L}_{\bm{y}}(p)&  \\
 \text{s.t. }&\bm{y}_0=1,& \\
 &\bm{M}_d(\bm{y}) \succeq 0, \forall d \in \mathbb{N}, & \\
 &\bm{M}_d(\bm{G}\bm{y}) \succeq 0, \forall d \in \mathbb{N}.
\end{aligned}
\label{mSOS-infinite}
\end{equation}
Let $d_p=\text{deg}(p)$, $d_G=\text{deg}(\bm{G})$, $ r_G=\lceil \frac{d_G}{2}  \rceil$ and $r \geq r_{\text{min}}:=\max (r_G, \lceil \frac{d_p}{2}  \rceil )$. The hierarchy of its finite-dimensional truncations then reads as
\begin{equation}
\tag{mSOS} \begin{aligned}
  p_r=&\underset{\substack{\y \in \mathbb{R}^{\vert \bm{b_{2r}}(\bm{x}) \vert}}}{\inf} \mathscr{L}_{\bm{y}}(p)&  \\
 \text{s.t. }&\bm{y}_0=1,& \\
 &\bm{M}_r(y) \succeq 0, & \\
 &\bm{M}_{r-r_{G}}(\bm{G}y) \succeq 0
\end{aligned}
\label{mSOS-finite}
\end{equation}
and is controlled by the relaxation order $r \in \mathbb{N}$. The problem \eqref{mSOS-finite} is called Lasserre moment sum-of-squares hierarchy. By solving \eqref{mSOS-finite} for a sequence of increasing relaxation orders, we obtain a non-decreasing sequence of lower bounds to \eqref{POP-PMI}. 

\begin{theorem}[\texorpdfstring{\citep[Theorem~2.2]{henrion2006convergent}}{}]
 Under Assumption \ref{assumption-archimedean}, we have $p_r \leq p_{r+1} \leq \cdots \leq p^*$ and $\lim_{r \rightarrow \infty}p_r =p^*$ 
 \label{mSOS-convergence}
\end{theorem}
As stated in \cite{henrion2006convergent}, Assumption \ref{assumption-archimedean} is not very restrictive: it suffices that one of the constraints of problem \eqref{POP-PMI} is of the form $\rho^2-\rVert \bm{x} \rVert^2 \geq 0$ with a suitable $\rho \in \mathbb{R}$. If this is not the case, it can be added as a redundant constraint by choosing $\rho$ so that it does not affect the solution of the problem \eqref{POP-PMI}. 
Finite convergence occurs when the rank of the pseudo-moment matrix stabilizes \cite{curto1996solution}, i.e.,
\begin{equation}
 \text{Rank}(\bm{M}_r(\bm{y}^*))= \text{Rank}(\bm{M}_{r-r_{G}}(\bm{y}^*)),
\label{flatness}
\end{equation}
where $\bm{y}^*$ is a solution to \eqref{mSOS-finite} at a fixed $r$. 
Based on this flatness condition, one can extract from \eqref{mSOS-finite} $\tau=\text{Rank}(\bm{M}_r(\bm{y}^*))$ distinct global minimizers of the problem \eqref{POP-PMI} using a linear algebra routine detailed in \cite{henrion2005detecting}. Depending on the nature of the problem, one can also have other types of conditions for finite convergence, see for instance \cite{henrion2023occupation}, or \cite{tyburec2021global,tyburec2022global,tyburec2024globalweightoptimizationframe} for conditions developed particularly for topology optimization problems.

\section{The arrow decomposition method}
\label{sec: ADLMI}

We start this section by recalling the notion of a so-called arrow-type matrix
\begin{equation}
    \mathbf{G} = \begin{bmatrix}\A & \B \\ \B^T& \bm{\Gamma}\end{bmatrix},
\end{equation}
in which $\A \in \mathbb{S}^n$, $\bm{\Gamma} \in \mathbb{S}^m$, and $\B \in \mathbb{R}^{n\times m}$ and where 
\begin{equation}
\A=\sum_{k=1}^p\A_k \text{ and } \A_k \in \Sbb^n, \forall k \in [p].
    \label{definition-arrow-A}
\end{equation}
Further, let $\mathcal{I}:=[n]$ denote the row/column index set of $\A$ and let $(\mathcal{I}_k)_{k \in [p]}$ be a partition of the set $\mathcal{I}$, where $\mathcal{I}_k$ corresponds to the index set of the non-zero rows and columns of $\bm{A}_k$. For example, if 
$$\bm{A} = \bm{A}_1 + \bm{A}_2 =\begin{bmatrix}  2 & 1 & 0 \\ 1 & 1 & 0 \\ 0 & 0 & 0\end{bmatrix}+\begin{bmatrix}  0 & 0 & 0 \\ 0 & 1 & 1 \\ 0 & 1 & 2\end{bmatrix},$$
we have $\mathcal{I}=\lbrace 1, 2, 3 \rbrace$, $\mathcal{I}_1=\lbrace 1, 2 \rbrace$, and $\mathcal{I}_2=\lbrace 2, 3 \rbrace$.

Now, for each subset $\mathcal{I}_k$, we define a matrix $\bm{\mathscr{E}}_{\mathcal{I}_k} \in \mathbb{R}^{ \vert \mathcal{I}_k \vert \times n}$ as $$(\bm{\mathscr{E}}_{\mathcal{I}_k})_{ij}=\left\lbrace\begin{aligned}
  &1, \text{ if  } \mathcal{I}_k(i)=j, \\ 
  &0, \text{ otherwise},
  \end{aligned}\right.
 $$ 
 where the notation $\mathcal{I}_k(i)$ stands for the $i$-th element of $\mathcal{I}_k$. This definition allows a compact representation of a sparse matrix. For example, for the matrix $\bm{A}_1$ defined above, we have $ \bm{\mathscr{E}}_{\mathcal{I}_1}=\begin{bmatrix}
     1 & 0 & 0 \\ 0 & 1 & 0 
 \end{bmatrix}$ and $\bm{A}_1=\bm{\mathscr{E}}_{\mathcal{I}_1}^T \begin{bmatrix}
     2 & 1 \\ 1 & 1
 \end{bmatrix} \bm{\mathscr{E}}_{\mathcal{I}_1}$.
 
 In what follows, we denote the intersection of the sets $\mathcal{I}_k$ and $\mathcal{I}_\ell$ by $\mathcal{I}_{k,\ell}=\mathcal{I}_k \cap \mathcal{I}_{\ell}$ with $k, \ell \in [p]$ and $k < \ell$. 

Using the index sets $\mathcal{I}_k$ related to the matrices $\mathbf{A}_k$, let us now decompose a rectangular matrix $\B \in \Rbb^{n \times m}$ as
\begin{equation}
\B=\sum_{k=1}^p\B_k \text{ and } \B_k \in \Rbb^{n\times m}, \forall k \in [p],
    \label{definition-arrow-B}
\end{equation} such that $(\bm{B}_k)_{ij}=0 \text{ for } i \notin \mathcal{I}_k$. 

Then, the so-called \textit{arrow type} matrix $ \G \in \Sbb^{n+m} $ reads as
\begin{equation}
\G=  \begin{bmatrix}\A & \B \\ \B^T& \bm{\Gamma}\end{bmatrix} = 
\sum_{i=1}^p\G_k+ \begin{bmatrix}
    \bm{0} & \bm{0} \\
    \bm{0} & \bm{\Gamma}
\end{bmatrix}
\text{ where }
\G_k=\begin{bmatrix}
    \A_k & \B_k \\
    \B_k^T & \bm{0}
\end{bmatrix}, k \in [p]
    \label{definition-arrow-matrix}
\end{equation}
with $\bm{\Gamma} \in \mathbb{S}^m$. 
By definition, we have 
\begin{equation}
    (\bm{G}_k)_{ij}=0 \text{ for } (i,j) \notin \mathcal{I}_k \cup \lbrace n+1,\cdots, n+m\rbrace, k \in [p].
    \label{definition-arrow-matrix-sparsity}
\end{equation}

In the following, we aim to replace the condition $\bm{G} \succeq 0$ with a set of LMIs $\bm{G}_k \succeq 0$ where each $\bm{G}_k$ is potentially of reduced size (see Section \ref{computational-consideration-lmi}). In general, the matrix $\G$, even when endowed with the arrow structure described above, can only be decomposed using classical chordal sparsity graph techniques (see \cite{agler1988positive,andersen2015chordal}). To enable the arrow decomposition technique, we introduce the following assumption on the matrices $\bm{A}_k$ and $\bm{\Gamma}$:
\begin{assumption}[Positive semidefiniteness\label{Assumption-4}]
$\A_k \succeq 0, \forall k \in [p]$ and $\bm{\Gamma} \succeq 0$.
\end{assumption}

Before stating the result allowing for an arrow decomposition of the matrix $\G$, let us first recall the generalized Schur complement for positive semidefinite (PSD) matrices. 
\begin{lemma}[\cite{gallier2011geometric}, Theorem 16.1]
   Let $\bm{A} \in \mathbb{S}^n$, $\bm{\Gamma} \in \mathbb{S}^m$, and $\bm{B}\in \mathbb{R}^{n\times m}$. Then, the following conditions are equivalent: \begin{itemize}
       \item $\begin{bmatrix} \bm{A} & \bm{B} \\ \bm{B}^T & \bm{\Gamma}\end{bmatrix} \succeq 0$,
       \item $\bm{A} \succeq \bm{0}$, $\bm{\Gamma}-\bm{B}^T\bm{A}^\dagger\bm{B} \succeq0, \bm{A}\bm{A}^\dagger\bm{B}=\bm{B} $.
   \end{itemize}
   \label{shur-complement}
\end{lemma}

The following theorem ensures that we can replace an LMI involving arrow type matrices by smaller LMIs at the cost of introducing additional variables. 
\begin{theorem}
\label{Kocvara-AD-PSD}
Let $\G$ be a matrix defined as in \eqref{definition-arrow-matrix}. Suppose that Assumption~\ref{Assumption-4} holds. Then the following statements are equivalent:
\begin{itemize}
\item $\G \succeq 0$
\item   there exist matrices $\D_{k,\ell} \in \Rbb^{n\times m}$ such that $(\D_{k,\ell})_{ij}=0$ whenever $(i,j) \notin \mathcal{I}_{k,\ell} \times  [m] , k<\ell $, and symmetric matrices $(\bm{C}_k)_{k \in [p]}$, such that $ \G=\sum_{k=1}^p \widetilde{\bm{G}}_k(\bm{D}_k,\bm{C}_k)$, and 
\begin{equation}
\begin{aligned}
\widetilde{\G}_k(\bm{D}_k,\bm{C}_k)=\G_k+ \begin{bmatrix}
        \bm{0}& \bm{D}_k \\
        \bm{D}_k^T & \C_k
    \end{bmatrix}  \succeq 0, k \in [p],
\end{aligned}
\label{AD-equation-LMI}
\end{equation}
where \begin{equation}
    \bm{D}_k=-\sum_{\ell:\ell<k} 
        \D_{\ell,k}+\sum_{\ell:\ell>k}  \D_{k,\ell},
    \label{additional-variables-dependencies}
\end{equation}
\begin{equation}
    \sum_{k=1}^p\bm{C}_k=\bm{\Gamma},
    \label{equation-compliances-advariables-balance}
\end{equation}

\end{itemize}
\end{theorem}
\begin{remark}
To keep notation simple, in the rest of this paper we will omit the dependency of the matrices $\widetilde{\bm{G}}_k$ on the additional variables $\bm{D}_k$ and $\bm{C}_k$, and will write $\widetilde{\bm{G}}_k:= \widetilde{\bm{G}}_k(\bm{D}_k,\bm{C}_k)$.

\end{remark}
Theorem \ref{Kocvara-AD-PSD} generalizes \cite[Theorem 3]{kovcvara2021decomposition} where the assumptions $\bm{A} \succ 0$ and $\bm{\Gamma} \succ 0$ where needed. These requirements, however, can be restrictive in certain applications, such as structural optimization (see Section \ref{sec:frames}). Although the proof of Theorem \ref{Kocvara-AD-PSD} closely follows that of \cite[Theorem 3]{kovcvara2021decomposition}, we provide it here explicitly for the reader's convenience.


\begin{proof}
   Suppose that $\bm{G} \succeq 0$. We will construct the matrices $\bm{D}_{k,\ell}$ and $\bm{C}_k$ that satisfy the conditions stated in the theorem. To this goal, let us first define $\bm{X}=\bm{A}^\dagger\bm{B}$. According to Lemma \ref{shur-complement}, we have $\bm{AA^\dagger B}=\bm{B} $ so that \begin{equation}
\displaystyle \sum_{k=1}^p\bm{A}_k\bm{X}=\sum_{k=1}^p\bm{B}_k.
       \label{AD-proof-eq1}
   \end{equation}
Thus, the entries of $\bm{A}_k \bm{X}$ corresponding to row indices in $\mathcal{I}_k$ that do not appear in any intersection with other index sets are uniquely determined. Specifically, for all $j \in [p]$ and $i \in \mathcal{I}_k \setminus \left( \bigcup_{\ell:\ell >k}(\mathcal{I}_{k,\ell}) \cup \bigcup_{\ell:\ell <k}(\mathcal{I}_{\ell,k})\right)$, we have \begin{equation}
    (\bm{A}_k\bm{X})_{ij}=(\bm{B}_k)_{ij}.
    \label{AD-proof-equality-1}
\end{equation}
 For each $k \in [p-1]$, we define successively $\bm{D}_{k,\ell}$ such that $ \ell > k$, by selecting a solution of the equation \begin{equation}
    \bm{A}_k\bm{X}-\bm{B}_k=-\sum_{\substack{\ell:\ell<k \\ \mathcal{I}_{\ell,k}\neq \varnothing}}{\bm{D}_{\ell,k}}+\sum_{\substack{\ell:\ell>k \\ \mathcal{I}_{k,\ell}\neq \varnothing}}{\bm{D}_{k,\ell}}.
    \label{AD-proof-equation-advariables}
\end{equation}
 The solutions of \eqref{AD-proof-equation-advariables} are obviously not unique, however any selected solution $\bm{D}_{k,\ell}$ is consistent with the last equation $$\bm{A}_p\bm{X}-\bm{B}_p=-\sum_{\substack{\ell:\ell<p \\ \mathcal{I}_{\ell,p}\neq \varnothing}}{\bm{D}_{\ell,p}},$$ because of Equality \eqref{AD-proof-eq1}. Moreover, by \eqref{definition-arrow-matrix-sparsity} and \eqref{AD-proof-equality-1}, each $(\bm{D}_{k,\ell})_{ij}=0$ when $(i,j) \notin \mathcal{I}_{k,\ell} \times [m]$ as desired.


Next, let us define $\bm{C}_k=\bm{S}+\bm{X}^T\bm{A}_k\bm{X}, k \in [p],$ where $\bm{S}=\frac{1}{p}(\bm{\Gamma}-\bm{X}^T\bm{B})$. We have $ \sum_{k=1}^p\bm{C}_k=p\bm{S}+\bm{X}^T(\sum_{k=1}^p\bm{A}_k)\bm{X}=\bm{\Gamma}-\bm{X}^T\bm{B}+\bm{X}^T\bm{A}\bm{X}=\bm{\Gamma}$. 

Moreover, we have  
\begin{align*}
  \widetilde{\bm{G}}_k&=\bm{G}_k-\sum_{\substack{\ell:\ell<k \\ \mathcal{I}_{\ell,k}\neq \varnothing}}\begin{bmatrix}\bm{0} & \bm{D}_{\ell,k} \\ \bm{D}_{\ell,k}^T & \bm{0}\end{bmatrix}+\sum_{\substack{\ell:\ell>k \\ \mathcal{I}_{k,\ell}\neq \varnothing}}\begin{bmatrix}\bm{0} & \bm{D}_{k,\ell} \\ \bm{D}_{k,\ell}^T & \bm{0}\end{bmatrix}+\begin{bmatrix} \bm{0} & \bm{0} \\ \bm{0} & \bm{C}_k\end{bmatrix}  \\
  &=\begin{bmatrix}
      \bm{A}_k & \bm{A}_k\bm{X} \\ \bm{X}^T\bm{A}_k^T & \bm{X}^T\bm{A}_k\bm{X}+\bm{S}      
  \end{bmatrix}.
\end{align*}
Since $\bm{A}_k \succeq 0$ by assumption, we conclude that $  \widetilde{\bm{G}}_k \succeq 0$ by using Lemma \ref{shur-complement}.
%

The reversed assertion holds by summing over $k$ the PSD matrices $\widetilde{\bm{G}}_k$.
\label{proof-AD}
\end{proof}

 Practically, the conditions \eqref{additional-variables-dependencies} and \eqref{equation-compliances-advariables-balance} will not be defined explicitly as in the proof of Theorem \ref{Kocvara-AD-PSD}. Instead, the matrices $\bm{D}_{k,\ell}$ and $\bm{C}_k$ are determined numerically via semidefinite programming, such that the constraints \eqref{additional-variables-dependencies} and \eqref{equation-compliances-advariables-balance} are embedded directly into the formulation \eqref{AD-equation-LMI} as
\begin{equation}
    \begin{aligned}
\G_k+ \begin{bmatrix}
        \bm{0}&  -\sum_{\ell:\ell<k} 
        \D_{\ell,k}+\sum_{\ell:\ell>k}  \D_{k,\ell} \\
        -\sum_{\ell:\ell<k} 
        \D_{\ell,k}^T+\sum_{\ell:\ell>k}^T  \D_{k,\ell} & \C_k
    \end{bmatrix} &\succeq 0, \forall k \in [p-1], \\
    \G_p+ \begin{bmatrix}
        \bm{0}& -\sum_{\ell:\ell<k} 
        \D_{\ell,p} \\
        -\sum_{\ell:\ell<k} 
        \D_{\ell,p}^T & \bm{\Gamma}-\sum_{k=1}^p\C_k
    \end{bmatrix} &\succeq 0.
\end{aligned}
\label{compact-formulation-AD-SDP}
\end{equation}
Let us notice that the choice of the decomposed matrices $\bm{B}_k$ is not unique. There is complete freedom in selecting them, as the positive semidefiniteness of $\widetilde{\bm{G}}_k$ can always be ensured by appropriately choosing the matrices $\bm{D}_{k,\ell}$.
Similarly, the parametrization of the blocks $\C_k$ is not unique, as any parametrization satisfying $\sum_{k=1}^p \C_k = \bm{\Gamma}$ is valid. 

For clarity and consistency, we adopt the compact formulation \eqref{compact-formulation-AD-SDP} in the sequel.


\begin{example}
Let $\bm{G}=\begin{bmatrix}
    \bm{A} &\bm{B}\\
    \bm{B}^T & \gamma
\end{bmatrix}$, where $$\bm{A}=\begin{bmatrix}
    2a_1+a_2 & a_1 & a_1+a_2 & a_2 & 0  \\
    a_1 & 2a_1 & a_1 & 0 & 0\\
a_1+a_2 & a_1 & a_1+a_2 & a_2 & 0 \\
a_2 & 0 & a_2 & a_2+a_3 & a_3 \\
0 & 0 & 0 & a_3 & 2a_3 
\end{bmatrix},$$
 $\bm{B}=\begin{bmatrix}
     b_1 & b_2& b_3 & b_4 & b_5 
\end{bmatrix}^T$, with $a_1,a_2,a_3, \gamma \in \mathbb{R}_{>0}$ and $b_1,b_2,b_3,b_4,b_5 \in \mathbb{R}$.
Using the index subsets $\mathcal{I}_1=\lbrace 1,2,3 \rbrace$, $\mathcal{I}_2=\lbrace 1,3,4 \rbrace$ and $\mathcal{I}_3=\lbrace 4,5 \rbrace$, we can extract from $\bm{A}$ the corresponding PSD matrices:
$$\bm{A}_1=a_1 \bm{\mathscr{E}}_{\mathcal{I}_1}^T\begin{bmatrix}
    2 & 1 & 1 \\ 1 & 2 & 1 \\ 1 & 1 & 1 
\end{bmatrix}\bm{\mathscr{E}}_{\mathcal{I}_1},
\bm{A}_2=a_2\bm{\mathscr{E}}_{\mathcal{I}_2}^T\begin{bmatrix}
    1 & 1 & 1  \\ 1  & 1 & 1 \\ 1  &  1 & 1  
\end{bmatrix}\bm{\mathscr{E}}_{\mathcal{I}_2},\bm{A}_3=a_3\bm{\mathscr{E}}_{\mathcal{I}_3}^T\begin{bmatrix}
     1 & 1 \\  1 & 2
\end{bmatrix}\bm{\mathscr{E}}_{\mathcal{I}_3}.$$
The intersections of the index subsets are $\mathcal{I}_{1,2}=\lbrace 1, 3\rbrace$, $\mathcal{I}_{1,3}=\varnothing$ and $\mathcal{I}_{2,3}=\lbrace 4\rbrace$. According to the index subsets, we choose $\bm{B}_1 = \bm{\mathscr{E}}_{\mathcal{I}_1}^T\begin{bmatrix}b_1 & b_2 & b_3 \end{bmatrix}^T$, $\bm{B}_2 = \bm{\mathscr{E}}_{\mathcal{I}_2}^T\begin{bmatrix}0 & 0 & b_4 \end{bmatrix}^T$, and $\bm{B}_3 = \begin{bmatrix} 0 & b_5\end{bmatrix}^T$. 
In contrast, 
$\D_1 = \D_{1,2} = \bm{\mathscr{E}}_{\mathcal{I}_1}^T\begin{bmatrix}d_{1,2}^{(1)} & 0 & d_{1,2}^{(2)}\end{bmatrix}^T$,
$\D_2  =-\D_{1,2}+\D_{2,3} =\bm{\mathscr{E}}_{\mathcal{I}_2}^T \begin{bmatrix}-d_{1,2}^{(1)}  & -d_{1,2}^{(2)} & d_{2,3}^{(1)} \end{bmatrix}^T$ 
and $\D_3 = -\D_{2,3} = \bm{\mathscr{E}}_{\mathcal{I}_2}^T\begin{bmatrix} -d_{2,3}^{(1)} & 0\end{bmatrix}^T$ are unique, with the notation $d_{k,\ell}^{(i)}$ expressing that it is an $i$-th unknown component of the $\bm{D}_{k,\ell}$ matrix. 
Using Theorem \ref{Kocvara-AD-PSD}, we can find $ d_{1,2}^{(1)}, d_{1,2}^{(2)} $, $ d_{2,3}^{(1)}$, $c_1$ and $c_2$ such that the condition $\bm{G} \succeq 0$ is equivalent to
$$ \widehat{\bm{G}}_1= \begin{bmatrix}
   \bm{A}_1  &  \bm{B}_1+ \bm{D}_1 \\  \bm{B}_1^T+\bm{D}_1^T & c_1
\end{bmatrix} \succeq 0,
  $$
$$ \widehat{\bm{G}}_2= \begin{bmatrix}
    \bm{A}_2  &   \bm{B}_2+\bm{D}_2 \\  \bm{B}_2^T+\bm{D}_2^T & c_2
\end{bmatrix} \succeq 0,
  $$
and 
$$ \widehat{\bm{G}}_3= \begin{bmatrix}
   \bm{A}_3  &   \bm{B}_3+\bm{D}_3 \\  \bm{B}_3^T+\bm{D}_3^T & \gamma-c_1-c_2
\end{bmatrix} \succeq 0.
  $$
\label{first-example}
\end{example}
\subsection{Computational considerations}
\label{computational-consideration-lmi}
We start this section by noting that, although not required for the validity of the theoretical results, certain structural properties of the index sets $\mathcal{I}_k$ and of the submatrices $\bm{A}_k$, $\bm{B}_k$ tend to improve numerical performance. In particular, the following properties help reduce the size of the decomposed matrices $\widehat{\bm{G}}_k$ and limit the number of additional variables introduced during decomposition:
\begin{itemize}
\item \textbf{Non-nesting subsets.}
$\mathcal{I}_{k} \cup \mathcal{I}_{\ell}  \neq \mathcal{I}_{k}$, for all $1 \leq k \neq \ell \leq p$.

\emph{Comment:} If this property fails, one matrix $\bm{A}_1$ is contained as a submatrix in another $\bm{A}_2$, leading to $\vert \bm{G} \vert = \vert \bm{G}_2 \vert$ and the decomposition will have no computational advantage. 
\item \textbf{Sparse intersection.} For each $k \in [p]$, there are at most $p_k$ indices $\ell_i$ such that $\mathcal{I}_{k,\ell_i}\neq \varnothing$, $i \in [p_k]$ where $1 \leq p_k \ll p.$

\emph{Comment:} This limits the number of additional variables introduced when using the arrow decomposition technique. Unlike chordal decomposition methods which follow the chordal sparsity pattern and give little control over the number of extra variables \cite{fukuda2001exploiting}, this property allows a more efficient decomposition. Unless the matrices  $\A_k$ are fully dense, it is generally possible to choose a decomposition satisfying this property.
\end{itemize}

Moreover, for each $k \in [p]$, at least one of the following properties should hold:
\begin{itemize}
\item \textbf{Sparse submatrices.}  $\A_k$ and $\B_k$ are sparse.

\emph{Comment:} This property clearly shrinks the size of $\widehat{\G}_k$ by removing the zero rows and columns, as seen in Example \ref{first-example}.
 \item \textbf{Low rank submatrices.} $\A_k$ is low rank.
 
 \emph{Comment:} This property allows in one hand, to reduce the size $\widehat{\G}_k$ by adequate projection procedure, on the other hand, it allows to reduce the number of additional variables needed to parametrize $\B_k+\D_k$ while maintaining positive semidefinitness of $\widehat{\G}_k$. It will be explained in more details in Proposition \ref{proposition-rank-defficiency}.
\end{itemize}
In practice, when decomposing the main matrix $\G$, it is preferable to select a decomposition that satisfies as many of these properties as possible, thereby balancing matrix size reduction with computational efficiency.



Furthermore, as Example \ref{first-example} revealed, it is convenient to maintain a component-wise representation of \eqref{additional-variables-dependencies}. To this end, we define $\bm{D}_k = \bm{\Pi}_k \bm{D}$, where $\bm{D}$ is a dense matrix formed by concatenating all entries of $\bm{D}_k$ for every $k \in [p]$, and $\bm{\Pi}_k$ is a binary (zero-one) matrix of appropriate dimensions that selects the relevant entries of $\bm{D}$ corresponding to each $k \in [p]$. The existence of such a construction is established in Lemma~\ref{compacte-form-add-variables} in Appendix~\ref{appendix}. This reformulation facilitates handling the component-wise relationship between $\bm{D}_k$ and $\bm{D}_{k,\ell}$ and also simplifies implementation. 
 
\begin{example}
    Consider the matrices $\bm{D}_k$, $k \in \{1,2,3\}$, in Example \ref{first-example}. Let $\bm{D}=\begin{bmatrix}
        d_{1,2}^{(1)} & d_{1,2}^{(2)} & d_{2,3}^{(1)}
    \end{bmatrix}^T$. We can write $\bm{D}_1=\bm{\mathscr{E}}_{\mathcal{I}_1}^T\bm{\Pi}_1\bm{D}$, $\bm{D}_2=\bm{\mathscr{E}}_{\mathcal{I}_2}^T\bm{\Pi}_2\bm{D}$ and $\bm{D}_3=\bm{\mathscr{E}}_{\mathcal{I}_3}^T\bm{\Pi}_3\bm{D}$, in which $$\bm{\Pi}_1=\begin{bmatrix}
        1 & 0 & 0\\ 0 & 0 & 0 \\ 0 & 1 & 0 
    \end{bmatrix},\quad
    \bm{\Pi}_2=\begin{bmatrix}
        -1 & 0 & 0\\ 0 & -1 & 0 \\ 0 & 0 & 1
    \end{bmatrix}, \quad
    \bm{\Pi}_3=\begin{bmatrix}
     0 & 0 & -1 \\ 0 & 0 & 0
    \end{bmatrix}.$$
\end{example}

Now, when determining the additional matrix variables $\D_{k}$ and $\C_k$ via semidefinite programming, it often occurs that one or more matrices $\bm{A}_k$ are rank-deficient. This implies that the constraint $\widetilde{\bm{G}}_k \succeq 0$ may lack an interior point, potentially leading to numerical instability in interior-point methods. The following result offers a further refinement of the approach in \cite{kovcvara2021decomposition}: it systematically removes potential rank deficiencies in the matrices $\bm{A}_k$, yielding smaller LMIs that are equivalent to \eqref{AD-equation-LMI}. Additionally, it provides conditions under which the resulting LMIs admit strict feasibility.

\begin{proposition}
Suppose that Assumption \ref{Assumption-4} holds. For each $k \in [p]$, let $\bm{P}_k \in \mathbb{R}^{n_k \times s_k}, s_k \leq n_k$, be a matrix with full column rank such that $\bm{P}_k^T\bm{P}_k=\bm{I}_{s_k}$, $\text{Span}(\bm{A}_k)\subseteq \text{Span}(\bm{P}_k)$. The condition \eqref{AD-equation-LMI} is equivalent to the conditions
 \begin{equation}
 \begin{aligned}
    & \begin{bmatrix}
        \bm{P}_k^T\bm{A}_k\bm{P}_k & \bm{P}_k^T(\bm{B}_k+\bm{D}_k) \\ (\bm{B}_k+\bm{D}_k)^T\bm{P}_k & \bm{C}_k
    \end{bmatrix} \succeq 0, \forall k \in [p-1], \\
     &\begin{bmatrix}
        \bm{P}_p^T\bm{A}_p\bm{P}_p & \bm{P}_p^T(\bm{B}_p+\bm{D}_p) \\ (\bm{B}_p+\bm{D}_p)^T\bm{P}_k & \bm{\Gamma}-\sum_{k=1}^{p-1}\bm{C}_k
    \end{bmatrix} \succeq 0,\\
    &\text{Span}(\bm{B}_k+\D_k) \subseteq \text{Span}(\bm{P}_k).
     \label{projected-lmi}
      \end{aligned}
 \end{equation}

    Moreover, if $\text{Span}(\bm{A}_k)= \text{Span}(\bm{P}_k)$ and $\bm{\Gamma}-\bm{B}^T\bm{A}^\dagger\bm{B} \succ 0$, then the above matrix inequalities in \eqref{projected-lmi} are strict.
    

\label{proposition-rank-defficiency}
\end{proposition}

\begin{proof}
Let $k \in [p]$ and define $\bm{D}_k$ as in \eqref{AD-proof-equation-advariables}. According to Theorem \ref{Kocvara-AD-PSD}, it suffices to show that $$\begin{bmatrix}
        \bm{A}_k & (\bm{B}_k+\bm{D}_k) \\ (\bm{B}_k+\bm{D}_k)^T & \bm{C}_k
    \end{bmatrix} \succeq 0 \Leftrightarrow \begin{bmatrix}
        \bm{P}_k^T\bm{A}_k\bm{P}_k & \bm{P}_k^T(\bm{B}_k+\bm{D}_k) \\ (\bm{B}_k+\bm{D}_k)^T\bm{P}_k & \bm{C}_k
    \end{bmatrix} \succeq 0.$$ 
    
    \begin{itemize}
        \item[$ \Rightarrow $] 
We have $$ \begin{bmatrix}
        \bm{P}_k^T\bm{A}_k\bm{P}_k & \bm{P}_k^T(\bm{B}_k+\bm{D}_k) \\ (\bm{B}_k+\bm{D}_k)^T\bm{P}_k & \bm{C}_k
    \end{bmatrix}=\begin{bmatrix}
        \bm{P}_k & \bm{0}\\ \bm{0} & \bm{I}
    \end{bmatrix}^T \begin{bmatrix}
        \bm{A}_k & (\bm{B}_k+\bm{D}_k) \\ (\bm{B}_k+\bm{D}_k)^T & \bm{C}_k
    \end{bmatrix}\begin{bmatrix}
        \bm{P}_k & \bm{0}\\ \bm{0} & \bm{I}
    \end{bmatrix} \succeq 0.$$    
    Moreover, we have $\text{Span}(\bm{B}_k+\D_k) \subseteq \text{Span}(\bm{A}_k)$ by Lemma \ref{shur-complement}, and thus $\text{Span}(\bm{B}_k+\D_k) \subseteq \text{Span}(\bm{P}_k)$ using the assumption on $\bm{P}_k$. 
     \item[$ \Leftarrow $] Let $\bm{u} \in \mathbb{R}^{n_k}$ and $\bm{v} \in \mathbb{R}^{m}$. 
     \begin{itemize}
         \item If $\bm{u} \in \text{Span}(\bm{P}_k)$, there exists $\bm{w}$ such that $\bm{u}=\bm{P}_k\bm{w}$. Then, we have
        \begin{align*}
\begin{bmatrix}
         \bm{u} \\ \bm{v}
     \end{bmatrix}^T \begin{bmatrix}
        \bm{A}_k & (\bm{B}_k+\bm{D}_k) \\ (\bm{B}_k+\bm{D}_k)^T & \bm{C}_k
    \end{bmatrix}  &\begin{bmatrix}
         \bm{u} \\ \bm{v}
     \end{bmatrix}=\\ &\begin{bmatrix}
         \bm{w} \\ \bm{v}
     \end{bmatrix}^T \begin{bmatrix}
        \bm{P}_k^T\bm{A}_k\bm{P}_k & \bm{P}_k^T(\bm{B}_k+\bm{D}_k) \\ (\bm{B}_k+\bm{D}_k)^T\bm{P}_k & \bm{C}_k
    \end{bmatrix}  \begin{bmatrix}
         \bm{w} \\ \bm{v}
     \end{bmatrix} \geq 0.     
     \end{align*}
     \item if $\bm{u} \notin \text{Span}(\bm{P}_k)$ then $\bm{u} \notin \text{Span}(\bm{A}_k)$, so that $\bm{u}^T\A_k \bm{u}=0$. Since $\text{Span}(\bm{B}_k+\D_k)\subseteq \text{Span}(\bm{P}_k)$, we have also $\bm{u}^T(\B_k+\D_k)\bm{v}=0$. Thus, we have 
  \begin{align*}
\begin{bmatrix}
         \bm{u} \\ \bm{v}
     \end{bmatrix}^T \begin{bmatrix}
        \bm{A}_k & (\bm{B}_k+\bm{D}_k) \\ (\bm{B}_k+\bm{D}_k)^T & \bm{C}_k
    \end{bmatrix}  &\begin{bmatrix}
         \bm{u} \\ \bm{v}
     \end{bmatrix}=    \bm{v}^T\C_k\bm{v} \geq 0, \text{ by Assumption \ref{Assumption-4}}.
     \end{align*}
  \end{itemize}
    \end{itemize}
      Now, if $\text{Span}(\bm{A}_k)=\text{Span}(\bm{P}_k)$, we have $\bm{P}_k^T\bm{A}_k\bm{P}_k \succ 0$. Similarly to the proof of Theorem \ref{Kocvara-AD-PSD}, we can define the matrices $\bm{D}_{k,\ell} $ so that it solves Equation \eqref{AD-proof-equation-advariables}, and $\bm{C}_k=\frac{1}{p}\bm{S}+\bm{X}^T\bm{A}_k\bm{X}$ with $\bm{X}=\bm{A}^\dagger \bm{B}$ and $\bm{S}=\bm{\Gamma}-\bm{B}^T\bm{A}^\dagger\bm{B}$. Using the assumption $\bm{S} \succ 0$ and Lemma \ref{shur-complement}, we conclude that 
      \begin{align*}
&\begin{bmatrix}
        \bm{A}_k & \bm{B}_k+\bm{D}_k \\ (\bm{B}_k+\bm{D}_k)^T & \bm{C}_k
    \end{bmatrix} \succeq 0  \Leftrightarrow \\ &\begin{bmatrix}
        \bm{P}_k^T\bm{A}_k\bm{P}_k & \bm{P}_k^T(\bm{B}_k+\bm{D}_k) \\ (\bm{B}_k+\bm{D}_k)^T\bm{P}_k & \bm{C}_k
    \end{bmatrix}=\begin{bmatrix}
        \bm{P}_k^T\bm{A}_k\bm{P}_k & \bm{P}_k^T\bm{A}_k\bm{X} \\ \bm{X}^T\bm{A}_k^T\bm{P}_k &  \bm{X}^T\bm{A}_k\bm{X}+\bm{S} 
    \end{bmatrix} \succ 0.       
      \end{align*}

\end{proof}

Proposition \ref{proposition-rank-defficiency} is particularly useful when one of the matrices $\bm{A}_k$ is rank-deficient for some $k \in [p]$. In such cases, by the spectral theory for positive semidefinite matrices, we can find a matrix $\bm{P}_k \in \mathbb{R}^{s_k \times n_k}$ with $ \bm{P}_k^T\bm{P}_k=\bm{I}_{s_k}$, $\text{Span}(\bm{A}_k) \subseteq \text{Span}(\bm{P}_k)$ and $s_k < n_k$, to which we impose the condition $\text{Span}(\bm{B}_k+\D_k)\subseteq \text{Span}(\bm{P}_k)$. This allows for a dimensionality reduction of one or more of the matrices involved the LMIs \eqref{projected-lmi}. 

Moreover, since the condition $\text{Span}(\bm{B}_k+\D_k)\subseteq \text{Span}(\bm{P}_k)$ is equivalent to $\bm{B}_k + \bm{\Pi}_k \bm{D} \in \text{Im}(\bm{P}_k)$, it suffices to impose $\bm{B}_k + \bm{\Pi}_k \bm{D} \in \text{Im}(\bm{A}_k)$. This constraint can be incorporated implicitly by reparameterizing $\bm{B}_k+\D_k$: choose a matrix $\bm{W}_k$ such that $\text{Span}(\bm{W}_k) = \text{Null}(\bm{A}_k)$, and solve for $\bm{D}$ the system $$\left\lbrace (\bm{B}_k + \bm{\Pi}_k \bm{D})^T \bm{W}_k = \bm{0}, k \in [p] \right\rbrace.$$ This approach allows elimination of some (or all) entries of $\bm{D}$ thereby improving numerical stability related to the interior point method while handling the LMI $\widetilde{\bm{G}}_k \succeq 0$; see, e.g., \texorpdfstring{\citep[Chapter~2]{de2006aspects}}{}. 

\begin{remark}
If the matrix $\bm{\Gamma}$ is rank-deficient, the same procedure can be applied by finding a matrix $\bm{P}_{\bm{\Gamma}} \in \mathbb{R}^{m \times s}$ with $s < m$, such that $\bm{P}_{\bm{\Gamma}}^T \bm{P}_{\bm{\Gamma}} = \bm{I}_s$, $\text{Span}(\bm{\Gamma}) \subseteq \text{Span}(\bm{P}_{\bm{\Gamma}})$ and $\text{Span}(\bm{B}_k+\D_k)\subseteq \text{Span}(\bm{P}_k)$. One can then work with the transformed matrices $\bm{P}_{\bm{\Gamma}}^T \bm{C}_k \bm{P}_{\bm{\Gamma}}$ for all $k \in [p]$.
\end{remark}

\begin{example}
Consider again the matrices from Example \ref{first-example}. We have $\text{rank}(\bm{A}_1)=3$,  $\text{rank}(\bm{A}_2)=1$, and $\text{rank}(\bm{A}_3)=2$. Therefore, while the null space bases of $\bm{A}_1 $, $\bm{A}_2$ and $\bm{A}_3$ respectively evaluate as
$$
\bm{W}_1 = \begin{bmatrix}
0 & 0\\
0 & 0\\
0 & 0\\
1 & 0\\
0 & 1
\end{bmatrix},\quad
\bm{W}_2 = \begin{bmatrix}
0 & -1 & -1 & 0\\
1 & 0 & 0 & 0\\
0 & 1 & 0 & 0\\
0 & 0 & 1 & 0\\
0 & 0 & 0 & 1
\end{bmatrix},\quad
\bm{W}_3 = \begin{bmatrix}
1 & 0 & 0\\
0 & 1 & 0\\
0 & 0 & 1\\
0 & 0 & 0\\
0 & 0 & 0
\end{bmatrix},
$$
range space bases read as
$
\bm{P}_1 = \bm{W}_3,
\bm{P}_2 =\bm{\mathscr{E}}_{\mathcal{I}_2}^T \begin{bmatrix} 1 & 1 & 1
\end{bmatrix}^T \text{ and }
\bm{P}_3 = \bm{W}_1.
$

%
Solving the system $\left\lbrace(\bm{B}_k + \bm{\Pi}_k \bm{D})^\top \bm{W}_k = \bm{0}, k=1,2,3 \right\rbrace$, we can reparametrize $\bm{D}$ using a single variable $d_R^{(1)}$ as
$$\bm{D} = \begin{bmatrix}
        d_{1,2}^{(1)} \\ d_{1,2}^{(2)} \\ d_{2,3}^{(1)}
    \end{bmatrix} = \begin{bmatrix}-b_4/3\\ -b_4/3 \\ -2b_4/3 \end{bmatrix} + \begin{bmatrix}-1\\-1\\1\end{bmatrix} d_{R}^{(1)}.$$
Finally, projection using the range space bases $\bm{P}_k$ leads to the final matrix inequalities of decreased size,
$$
\begin{bmatrix}
2 a_1 & a_1 & a_1 & b_1-b_4/3 - d_R^{(1)}\\
a_1 & 2 a_1 & a_1 & b_2\\
a_1 & a_1 & a_1 & b_3 - b_4/3- d_R^{(1)}\\
b_1-b_4/3-d_R^{(1)} & b_2 & b_3-b_4/3 - d_R^{(1)} & c_1
\end{bmatrix}\succeq 0,
$$
$$
\begin{bmatrix}
9 a_2 & b_4+3d_R^{(1)}\\
b_4+3d_R^{(1)} & c_2
\end{bmatrix}\succeq 0,
\begin{bmatrix}
a_3 & a_3 & 2b_4/3 -d_R^{(1)}\\
a_3 & 2a_3 & b_5\\
2b_4/3-d_R^{(1)} & b_5 & \gamma-c_1-c_2
\end{bmatrix}\succeq 0.
$$

\end{example}


\if{
\section{Exploiting arrow sparsity for polynomial optimization }
\label{sec:ASPOP}
As in Section~\ref{sec: ADLMI}, given $n,m \in \N$, let us assume that $I = \{1,\dots,n\}$ is partitioned into $p\geq  2$ overlapping sets $I = \I_1 \cup \dots \cup \I_p$. 
For convenience, let us define $\I_m = \{n+1,\dots,n+m\}$, $\Ikell := \Ik \cap \Iell$, $\hIk := \I_k \cup \I_m$ for each $k=1,\dots,p$ and $\hIkell = \Ikell \cup \I_m$. 
Let $n_k := |\Ik|$ and $\hat{n}_k := |\hIk|$. 

Let us now consider a vector of variables $\x = (x_1,\dots,x_n,x_{n+1},\dots,x_{n+m})$, a polynomial 
$f \in \Rbb[\x]$ and a finite subset of polynomials $g=(g_1,\dots,g_r) \subset \Rbb[\x]$ for $r \in \N$. 
\begin{assumption}\label{hyp:ASPOP}
One has $f=\sum_{k=1}^p (f_k + \hat{f_k})$, with $f_k \in \Rbb[\x,\Ik]$, $\hat{f_k} \in \Rbb[\x,\hIk]$, for each $k=1,\dots,p$ and for each $j=1,\dots,r$ there exists $k(j)$ such that $g_j \in \Rbb[x,\hI_{k(j)}]$. 
For each $k = 1,\dots,p$, there exists $j(k)$ such that $g_{j(k)}=\hat{n}_k - \sum_{i \in \hIk} x_i^2$.
\end{assumption}
Under Assumption \ref{hyp:ASPOP} one can partition $J=\{1,\dots,r\}$ into $p$ overlapping sets $J=J_1\cup \dots \cup J_p$ with $J_k := \{j \in J: g_j \in \Rbb[\x,\hIk] \}$. 

Similarly to the hierarchy exploiting correlative sparsity (CS) for polynomial optimization, we now derive a hierarchy exploiting arrow sparsity (AS). 
For each $k=1,\dots,p$, let us define the subset $\hRk \subset \Rbb[\x]$ as follows:
\begin{align}
\label{eq:ASSOS}
\hRk:= \left\{ 
 s_k - \sum_{\ell < k} d_{\ell,k} + \sum_{\ell > k} d_{k,\ell} + c_k : 
 s_k \in \Rbb[x,\hIk], 
 d_{k,\ell} \in  \Rbb[x,\hIkell], 
 c_k \in  \Rbb[x,\I_m]
\right\}.
\end{align}
We also define the subset $\hSigmak = \hRk \cap \Sigma[\x] $ of the SOS cone. 
Let $K = \{\x \in \Rbb^n : g_1(\x) \geq 0, \dots, g_r(\x) \geq 0 \}$. Under Assumption \ref{hyp:ASPOP} one can define the set $K_k := \{\x \in \Rbb^{n_k} :  g_j(\x) \geq 0, j \in J_k \}$. 
We are now able to state our main result. 
\begin{theorem}
\label{th:ASPOP}
Let $f, g$ be as in Assumption \ref{hyp:ASPOP} and suppose that $f$ is positive  on $K = \{\x \in \Rbb^n : g_1(\x) \geq 0, \dots, g_r(\x) \geq 0 \}$. Then $f= h_1 + \dots + h_p$ for some $h_k >0$ on $K_k$. 
\end{theorem}
}\fi

\section{Arrow decomposition and the moment sum-of-squares hierarchy }
\label{sec:ADMSOS}

This section is devoted to exploiting arrow-type structure in polynomial optimization problems. Thus, in what follows, we consider the optimization problem defined in \eqref{POP-PMI} such that $\bm{G} \in \mathbb{S}^{n+m}[\bm{x}]$ has an arrow shape, i.e.,
\begin{equation}
\bm{G}(\bm{x})= \sum_{k=1}^p \bm{G}_k(\bm{x})+ \begin{bmatrix}
    \bm{0} & \bm{0} \\
    \bm{0} & \bm{\Gamma}(\bm{x})
\end{bmatrix},
\label{arrow-structure-PMI}
\end{equation}
where for all $k \in [p]$, $\bm{G}_k(\bm{x})=\begin{bmatrix}
    \A_k(\bm{x}) & \B_k(\bm{x}) \\
    \B_k(\bm{x})^T & \bm{0}
\end{bmatrix} $ with $\bm{A}_k(\bm{x}) \in \mathbb{S}^n[\bm{x}]$, $\bm{B}_k(\bm{x})\in \mathbb{R}^{n \times m}[\bm{x}]$, and $\bm{\Gamma}(\bm{x}) \in \mathbb{S}^m[\bm{x}]$. Similarly to Section \ref{sec: ADLMI}, we make the following assumption on the matrices $\bm{A}_k(\bm{x})$ and $\bm{\Gamma}(\bm{x})$: 
\begin{assumption}[Positive semidefiniteness\label{Assumption-5}]
For all $\bm{x} \in \mathcal{K}$ and $ k \in [p]$, $\A_k(\bm{x}) \succeq 0$ and $\bm{\Gamma}(\bm{x}) \succeq 0$.
\end{assumption}

In this section, we present two different approaches for combining the mSOS and AD method:
\begin{itemize}
\item AD prior to the mSOS hierarchy (Section~\ref{subsec:prior}): AD is applied directly to the original polynomial optimization problem. Although this approach has little practical interest, we include it as it provides an insightful link between the main POP and the second approach.
\item AD posterior to the mSOS hierarchy (Section~\ref{subsec:posterior}): AD is applied to the relaxations within the mSOS hierarchy. This is the main focus of the paper.
\end{itemize}

\subsection{AD prior to the mSOS hierarchy}\label{subsec:prior}
First, notice that it follows from the proof of Theorem \ref{Kocvara-AD-PSD} that if, for all $\bm{x}$, the matrix $\bm{G}(\bm{x})$ satisfies the assumptions for which Theorem \ref{Kocvara-AD-PSD} holds, then the additional matrices $\bm{D}_k $ and $\bm{C}_k$ can be expressed as rational functions in $\bm{x}$. Indeed, from Equation \eqref{AD-proof-equation-advariables}, we would need to find a matrix $\bm{D}$ such that $\bm{D}_k=\bm{A}_k(\bm{x})\bm{A}(\bm{x})^\dagger\bm{B}(\bm{x})-\bm{B}_k(\bm{x})$ for all $k \in [p]$. However, this can not be exploited easily, because we would need to symbolically evaluate $\bm{A}(\bm{x})^\dagger\bm{B}(\bm{x})$. Instead, we introduce linear variables and require that their values match the fractional expressions at the optimum solution $\bm{x}^*$ of Problem \eqref{POP-PMI}. More specifically, set $\bm{C}=\lbrace \bm{C}_k\rbrace_{k \in [p-1]}$ and consider the following problem 
\begin{equation}
\tag{POP-AD}\begin{aligned}
  p_{\text{AD}}=\underset{\substack{\x, \bm{D}, \bm{C}} }{\inf}& p(\x),&  \\
 \text{s.t. }& \bm{G}_k(\bm{x})+ \begin{pmatrix}
    \bm{0} & \bm{\Pi}_k\bm{D} \\
    \bm{D}^T\bm{\Pi}_k^T & \bm{C}_k
    \end{pmatrix} \succeq 0,\ \ k \in [p-1],\\
    & \bm{G}_p(\bm{x})+ \begin{pmatrix}
    \bm{0} & \bm{\Pi}_p\bm{D} \\
    \bm{D}^T\bm{\Pi}_p^T & \bm{\Gamma}(\bm{x})-\sum_{k=1}^{p-1}\bm{C}_k
    \end{pmatrix} \succeq 0,
\end{aligned}
\label{POP_AD1}
\end{equation}
\begin{proposition}
Under Assumptions \ref{assumption-archimedean} and \ref{Assumption-5}, we have $p_{\text{AD}}=p^*$.
   \label{prop-addvariables-linear}
\end{proposition}

\begin{proof}
If $(\bm{x}^*,\bm{D}^*, \bm{C}^*)$ is a solution to \eqref{POP_AD1} then according to Theorem \ref{Kocvara-AD-PSD}, $\bm{x}^*$ is feasible to \eqref{POP-PMI} so that $p_{\text{AD}} \geq p^*$. Conversely, if $\bm{x}^*$ is a solution to \eqref{POP-PMI}, then by Theorem \ref{Kocvara-AD-PSD} we can find $\bm{D}^*, \bm{C}^*$ such that Equation \eqref{AD-equation-LMI} holds and
$ \widetilde{\bm{G}}_k(\bm{x}^*) \succeq 0$, for all $ k \in [p]$. Thus, $p_{\text{AD}} \leq p^*$.
\end{proof}

Now, we apply AD directly to the problem \eqref{POP-PMI}, then apply the mSOS naively on the resulting problem while taking into account the additional variables when constructing the pseudo-moment sequence. More precisely, we consider the relaxations
\begin{equation}
\tag{ADmSOS$_{1}$}\begin{aligned}
  p_{\text{AD},r}^{\text{prior}}=&\underset{\substack{\widetilde{\bm{y}}=(\bm{y}, \bm{y}_{\bm{D}}, \bm{y}_{\bm{C}}})}{\min} \mathscr{L}_{\bm{y}}\left(p \right),  \\
 \text{s.t. }& y_{\bm{0}}=1, \\
 & \bm{M}_r(\tilde{\bm{y}}) \succeq 0, \\
 &\bm{M}_{r-r_G}(\bm{G}_{k}\tilde{\bm{y}})+ \bm{M}_{r-r_G}\left( \begin{bmatrix}
   \bm{0} & \bm{\Pi}_k\bm{D} \\
  \bm{D}^T\bm{\Pi}_k^T & \bm{C}_k
\end{bmatrix}\tilde{\bm{y}} \right) \succeq 0,\ \  k \in [p-1], \\
 &\bm{M}_{r-r_G}(\bm{G}_{p}\tilde{\bm{y}})+ \bm{M}_{r-r_G}\left( \begin{bmatrix}
   \bm{0} & \bm{\Pi}_p\bm{D} \\
  \bm{D}^T\bm{\Pi}_p^T & \bm{\Gamma}(\bm{x})-\sum_{k=1}^{p-1}\bm{C}_k(\x)
\end{bmatrix}\tilde{\bm{y}} \right) \succeq 0,
\end{aligned}
\label{AD-prior-mSOS}
\end{equation}
where 
%
$\tilde{\bm{y}}$ is the pseudo-moment sequence associated with the variables $(\bm{x},\bm{D},\bm{C})$, and $r \geq r_{\text{min}}$ with $r_{\text{min}}$ defined in Section \ref{sec:mSOS}.
Under Assumptions \ref{assumption-archimedean} and \ref{Assumption-5}, and according to Proposition \ref{prop-addvariables-linear} and Theorem \ref{mSOS-convergence}, we can show that $p_{\text{AD},r}^{\text{prior}}=p_r$ for any relaxation order $r \geq r_{\text{min}}$. However, the sizes of the moment and localizing matrices in \eqref{AD-prior-mSOS} will typically be larger than in \eqref{mSOS-finite}. 
In order to illustrate this, let $\tilde{\bm{x}}=(\bm{x},\bm{D},\bm{C})$ and $\bm{b}_r(\tilde{\bm{x}})$ the monomial basis associated to $\tilde{\bm{x}}$ at the relaxation order $r$. 
We denote by $n_{\bm{D}}$ and $n_{\bm{C}}$ the number of distinct the component-wise additional variables $\bm{D}$ and $\bm{C}$ respectively. 
Recall that $\left\lvert\bm{b}_r(\bm{x})\right\rvert= \binom{n+r}{n}$. 
We have $\frac{\vert \bm{b}_{r+1}(\bm{x}) \vert}{\vert \bm{b}_{r}(\bm{x}) \vert}=1+\frac{n_x}{r+1}$ and $\frac{\vert \bm{b}_{r+1}(\tilde{\bm{x}}) \vert}{\vert \bm{b}_{r}(\tilde{\bm{x}}) \vert}=1+\frac{n_x+n_{\bm{C}}+n_{\bm{D}}}{r+1}.$ 
This shows that, for a relatively large $n_{\bm{C}}+n_{\bm{D}}$, the size of the monomial basis $\bm{b}_{r}(\tilde{\bm{x}})$ used for building the moment and localizing matrices grows much faster than the basis $\bm{b}_{r}(\bm{x})$ as $r$ increases.

Moreover, let $\bm{y}$ be a pseudo-moment vector defined for \eqref{mSOS-finite}, and $\widetilde{\bm{G}}_k(\bm{x})=\bm{G}_k(\bm{x})+\begin{bmatrix}
   \bm{0} & \bm{\Pi}_k\bm{D} \\
  \bm{D}^T\bm{\Pi}_k^T & \bm{C}_k
\end{bmatrix}$. The following result shows the conditions under which the matrix $\bm{M}_{r-r_{\bm{G}}}(\widetilde{\bm{G}}_k\tilde{\bm{y}})$ could have smaller size than $\bm{M}_{r-r_{\bm{G}}}(\bm{G}\bm{y})$.
\begin{proposition}
Let $p \in [p]$. We have $\vert  \bm{M}_{r-r_{\bm{G}}}(\widetilde{\bm{G}}_k\tilde{\bm{y}}) \vert \leq \vert  \bm{M}_{r-r_{\bm{G}}}(\bm{G}\bm{y}) \vert $ if and only if 
    \begin{equation}
\prod_{i=1}^{n_D+n_C}\left(1+\frac{r-r_{\bm{G}}}{n_x+i}\right) \leq \frac{\vert \bm{G} \vert}{\vert\widetilde{\bm{G}}_k\vert}.
        \label{inequation-proposition-complexity-AD-prior}
    \end{equation}
\end{proposition}
\begin{proof}
On the one hand, we have
  $\vert  \bm{M}_{r-r_{\bm{G}}}(\widetilde{\bm{G}}_k\tilde{\bm{y}}) \vert \leq \vert  \bm{M}_{r-r_{\bm{G}}}(\bm{G}\bm{y}) \vert   $
is equivalent to 
$$  \vert \widetilde{\bm{G}}_k\vert  \binom{n_x+n_{\bm{D}}+n_{\bm{C}}+r-r_{\bm{G}}}{n_x+n_{\bm{D}}+n_{\bm{C}}}  \leq \vert \bm{G}\vert \binom{n_x+r-r_{\bm{G}}}{n_x}. 
$$ 
 On the other hand, we have $\binom{n_x+n_{\bm{D}}+n_{\bm{C}}+r-r_{\bm{G}}}{n_x+n_{\bm{D}}+n_{\bm{C}}}=\binom{n_x+r-r_{\bm{G}}}{n_x} \prod_{i=1}^{n_D+n_C}\frac{n_x+r-r_{\bm{G}}+i}{n_x+i}.$ 
 We obtain  $\prod_{i=1}^{n_D+n_C}\frac{n_x+r-r_{\bm{G}}+i}{n_x+i}=\prod_{i=1}^{n_D+n_C}\left(1+\frac{r-r_{\bm{G}}}{n_x+i}\right) \leq \frac{\vert \bm{G}\vert}{\vert \widetilde{\bm{G}}_k\vert}.$  
\end{proof}
The inequality \eqref{inequation-proposition-complexity-AD-prior} holds if $r-r_{\bm{G}}=0$ or ($r-r_{\bm{G}}>0$ and $n_{\bm{D}}=n_{\bm{C}}=0$) or ($n_{\bm{D}}+n_{\bm{C}}>0$ and $0<r-r_{\bm{G}} \ll n_x$). In particular, if we have significantly large number of variables $ n_{\bm{D}}+n_{\bm{C}}$, then the previous inequality will more likely not hold. These conditions render solving numerically \eqref{AD-prior-mSOS} 
not very interesting, unless the decomposition reveals additional sparse structure in the POP such as correlative sparsity \cite{magron2023sparse}. However, this aspect lies beyond the scope of the present work.

%

\begin{example}
   Let $\bm{x} \in \mathbb{R}^3_{+}$ and 
   $\bm{G}(\bm{x})=\begin{bmatrix}
       2x_1^2 & x_1 & 0 & b_1 \\
       x_1 & x_1^2+2x_2^2 & x_2 & b_2 \\
       0 & x_2 & x_2^2 & b_3 \\
       b_1 & b_2 & b_3 & x_3
   \end{bmatrix}.$
   It can be decomposed as 
   $$\widetilde{\bm{G}}_1(\bm{x})=\begin{bmatrix}
       2x_1^2 & x_1 & b_1 \\ x_1 & x_1^2 & b_2+d_{1,2}^{(1)} \\ b_1 & b_2+d_{1,2}^{(1)} & c_1
   \end{bmatrix},\;\widetilde{\bm{G}}_2(\bm{x})=\begin{bmatrix}
       2x_2^2 & x_2 & -d_{1,2}^{(1)} \\ x_2 & x_2^2 & b_3 \\ -d_{1,2}^{(1)} & b_3 & x_3-c_1
   \end{bmatrix}.$$
   For $r=2$, we have $\vert \bm{M}_{r-r_{\bm{G}}}(\bm{G}\bm{y})\vert=16$ and $\vert \bm{M}_{r-r_{\bm{G}}}(\widetilde{\bm{G}}_1\bm{\tilde{y}})\vert=\vert \bm{M}_{r-r_{\bm{G}}}(\widetilde{\bm{G}}_2\bm{\tilde{y}})\vert=15$. Then for $r=3$, we get $\vert \bm{M}_{r-r_{\bm{G}}}(\bm{G}\bm{y})\vert=40$ and $\vert \bm{M}_{r-r_{\bm{G}}}(\widetilde{\bm{G}}_1\bm{\tilde{y}})\vert=\vert \bm{M}_{r-r_{\bm{G}}}(\widetilde{\bm{G}}_2\bm{\tilde{y}})\vert=45$.
\end{example}

\subsection{AD posterior to the mSOS hierarchy}\label{subsec:posterior}

The approach here is to apply the AD on the relaxations \eqref{mSOS-infinite} and derive its tractable truncations. Specifically, we provide new necessary and sufficient conditions for a sequence $\bm{y}=(y_{\bm{\alpha}})_{\bm{\alpha}\in \mathbb{N}^{n_x}}$ to have a representing measure on $\mathcal{K}=\lbrace  \bm{x} \in \mathbb{R}^{n_x}: \G(\x) \succeq 0\rbrace$ that exploit the arrow structure of $\G$. These new conditions have two advantages: 
\begin{itemize}
    \item it avoids including the component-wise auxiliary variables introduced by the decomposition (i.e., those in $\D$ and $\C$) in the monomial basis, as was required in the previous section, and thus leading to smaller localizing matrix inequalities; 
    \item when truncated, they give tighter lower bounds to \eqref{POP-PMI} compared to the standard approach \eqref{mSOS-finite}.
\end{itemize}

We consider again \eqref{POP-PMI} such that the matrix $\bm{G(x)} \in \Sbb^{n+m}[\x]$ has an arrow structure as in \eqref{arrow-structure-PMI}. Let $d \in \mathbb{N}$ be an arbitrary degree, $L_d=\vert \bb_d(\x)\vert $ and $\y \in \Rbb^{L_{2d}}$ a vector. By Equation \eqref{arrow-structure-PMI}, the matrix $\M_{d}(\G \y) $ can be written as
\begin{equation}
\M_{d}(\G \y)=\sum_{k=1}^p\bm{M}_{d}(\bm{G}_{k} \bm{y}) + \begin{bmatrix}
    \bm{0} & \bm{0} \\ \bm{0} & \widehat{\bm{\Gamma}}_d(\bm{y})
\end{bmatrix},
    \label{definition-Arrow-structure-localizing-matrix}
\end{equation}
with 
$$\bm{M}_{d}(\bm{G}_{k} \bm{y})=\begin{bmatrix}
    \widehat{\bm{A}}_{d,k}(\bm{y}) & \widehat{\bm{B}}_{d,k}(\bm{y}) \\
    \widehat{\bm{B}}_{d,k}(\bm{y})^T & \bm{0}
\end{bmatrix},$$
 and $\widehat{\bm{A}}_{d,k}(\bm{y}) =\mathscr{L}_{\y}\left(\A_k(\x) \otimes \bb_{d}(\x)^T\bb_{d}(\x)\right)$, $\widehat{\bm{B}}_{d,k}(\bm{y}) =\mathscr{L}_{\y}\left(\B_k(\x) \otimes \bb_{d}(\x)^T\bb_{d}(\x)\right)$, 
    and $\widehat{\bm{\Gamma}}_d(\bm{y}) =\mathscr{L}_{\y}\left(\bm{\Gamma}(\x) \otimes \bb_{d}(\x)^T\bb_{d}(\x)\right). $

We denote by $(\widehat{\mathcal{I}}_k)_{k \in [p]}$ the index set of the matrices $ \widehat{\bm{A}}_{d,k}(\bm{y})$, that can be defined explicitly by  \begin{equation}
\widehat{\mathcal{I}}_k=\left\lbrace (i_k-1) L_d +i_{b}: (i_k,i_b) \in \mathcal{I}_k \times [ L_d ] \right\rbrace,
    \label{definition-index-set-AD-posterior}
\end{equation} 
where the index $i_b$ enumerates the element in the monomial basis $\bm{b}_d(\bm{x})$. 


Notice that the block sparsity of the matrix $\M_d(\G\y)$ is the same as the sparsity of the matrix $\G(\x)$. 

Furthermore, we consider the matrices $\widehat{\bm{\Pi}}_k$ defined in \eqref{matrix-Pi} for the sets $\widehat{\mathcal{I}}_k$. We establish in Lemma \ref{lemma-correspondance-sigmas} ( Appendix \ref{appendix-lemma}) a direct link between the matrix $\bm{\Pi}_k$ associated to $\mathcal{I}_k$, and the matrix $\widehat{\bm{\Pi}}_k$ associated to $\widehat{\mathcal{I}}_k$ via the equation $\widehat{\bm{\Pi}}_k =\bm{\Pi}_k \otimes   \bm{I}_{L_{d}}$. 

The following result is an arrow decomposed version of Theorem \ref{putinar}.
\begin{theorem}
 Suppose that Assumptions \ref{assumption-archimedean} and \ref{Assumption-5} hold and let $\bm{y}=(y_{\bm{\alpha}})_{\bm{\alpha}\in \mathbb{N}^{n_x}}$ be a pseudo-moment vector. The following statements are equivalents: 
 \begin{enumerate}
     \item $\bm{y}$ has a representing measure supported on $\mathcal{K}$
     \item For all $d \in \mathbb{N}$, $\bm{M}_d(\bm{y}) \succeq 0$ and there exists matrices $ \widehat{\bm{D}}_d \in \mathbb{R}^{L_dn_{\mathcal{I}} \times L_d m}$ and $ \widehat{\bm{C}}_{d,k} \in \mathbb{S}^{mL_d}$ for all $k \in [p]$, such that 
     \begin{equation}
     \begin{aligned}     
       \M_{d}(\G_{k} \y)+ \begin{bmatrix}
   \bm{0} & (\bm{\Pi}_k \otimes \bm{I}_{L_{d}})\widehat{\bm{D}}_d \\
 \widehat{\bm{D}}_d^T (\bm{\Pi}_k \otimes \bm{I}_{L_{d}})^T & \widehat{\bm{C}}_{d,k}
\end{bmatrix} &\succeq 0, k \in [p-1], \\ 
 \M_{d}(\G_{p} \y)+ \begin{bmatrix}
   \bm{0} & (\bm{\Pi}_p \otimes \bm{I}_{L_{d}})\widehat{\bm{D}}_d \\
 \widehat{\bm{D}}_d^T (\bm{\Pi}_p \otimes \bm{I}_{L_{d}})^T & \widehat{\bm{\Gamma}}_d(\bm{y})-\sum_{k=1}^{p-1}\widehat{\bm{C}}_{d,k}
\end{bmatrix} &\succeq 0.
  \end{aligned}
       \label{condition-rank-def-moment}
   \end{equation}  
   
 \end{enumerate}
 \label{AD-putinar}
\end{theorem}

\begin{proof}
Suppose that the first statement holds. According to Theorem \ref{putinar}, we have $\bm{M}_d(\bm{y}) \succeq 0$ and $\bm{M}_d(\bm{Gy}) \succeq 0$ for all $d\in \mathbb{N}$. We claim that $\widehat{\bm{A}}_{d,k}(\bm{y}) \succeq 0, \forall k \in [p]$ and $\widehat{\bm{\Gamma}}_d(\bm{y}) \succeq 0$. To prove this claim, fix $k \in [p]$, $\bm{v} \in \mathbb{R}^{nL_d}$, and define the polynomial vector $\bm{w}(\x)=(\bm{I}_{n} \otimes \bb_d(\x)) \bm{v}$. We have
$$ \bm{v}^T \widehat{\bm{A}}_{d,k}(\bm{y})  \bm{v} = \mathscr{L}_{\y}\left( \langle \bm{A}_k(\bm{x}) \otimes \bm{I}_{L_d}, \bm{w(x)}^T\bm{w(x)}\rangle \right).$$
 By assumption, the matrix $\A_k(\x)$ is positive semidefinite on $\mathcal{K}$, and thus we have
$\langle \bm{A}_k(\bm{x}) \otimes \bm{I}_{L_d}, \bm{w(x)}^T\bm{w(x)}\rangle 
= \bm{w(x)}^T (\bm{A}_k(\bm{x}) \otimes \bm{I}_{L_d}) \bm{w(x)}
\geq 0$. Now, since $\bm{y}$ has representing measure supported on $\mathcal{K}$, we have $ \mathscr{L}_{\y}\left( \langle \bm{A}_k(\bm{x}) \otimes \bm{I}_{L_d}, \bm{w(x)}^T\bm{w(x)}\rangle \right)=\int_{\mathcal{K}} \langle \bm{A}_k(\bm{x}) \otimes \bm{I}_{L_d}, \bm{w(x)}^T\bm{w(x)}\rangle d \mu \geq 0$ (see section \ref{sec:notation}), so that $ \widehat{\bm{A}}_{d,k}(\bm{y})  \succeq 0$. 
Similarly, we prove that $\widehat{\bm{\Gamma}}_d(\bm{y}) \succeq 0$. Now, according to Proposition \ref{putinar}, we have $\bm{M}_d(\bm{Gy}) \succeq 0$, we can then conclude using Theorem \ref{Kocvara-AD-PSD}.

Conversely, suppose that the second statement holds. By Theorem \ref{Kocvara-AD-PSD}, we have $\bm{M}_d(\bm{Gy}) \succeq 0$ and $\bm{M}_d(\bm{y}) \succeq 0 $, for all $d \in \mathbb{N}$, so $\y$ has a representing measure by Theorem \ref{putinar}.
\end{proof}

\begin{remark}
 To ease the notation in the sequel, we omit the dependency of the matrices $\widehat{\bm{A}}_k$, $\widehat{\bm{B}}_k$, $\widehat{\bm{D}}$, $\widehat{\bm{C}}_k$ and $\widehat{\bm{\Gamma}}$, on the degree $d$. Moreover, we denote in the sequel $\widehat{\C}= (\widehat{\C}_k)_{k \in [p-1]}$.
\end{remark}

Theorem \ref{AD-putinar} allows us to state an arrow decomposition version of the infinite relaxations in \eqref{mSOS-infinite}:
\begin{corollary}
  Suppose that Assumptions \ref{assumption-archimedean} and \ref{Assumption-5} hold. We have 
\begin{equation}
\begin{aligned}
  p^*= \ &\underset{\substack{\y, \widehat{\bm{D}}, \widehat{\bm{C}}}}{\min} \ \mathscr{L}_{\bm{y}}(p)&  \\
 \text{s.t. }&\bm{y}_0=1,& \\
 &\bm{M}_d(\bm{y}) \succeq 0, \forall d \in \mathbb{N}, & \\
 &\M_{d}(\G_{k} \y)+ \begin{bmatrix}
   \bm{0} & (\bm{\Pi}_k \otimes \bm{I}_d) \widehat{\bm{D}} \\
 \widehat{\bm{D}}^T(\bm{\Pi}_k \otimes \bm{I}_d)^T & \widehat{\bm{C}}_k 
\end{bmatrix} \succeq 0, k \in [p-1], \forall d \in \mathbb{N},\\
 &\M_{d}(\G_{p} \y)+ \begin{bmatrix}
   \bm{0} & (\bm{\Pi}_p \otimes \bm{I}_d) \widehat{\bm{D}} \\
 \widehat{\bm{D}}^T(\bm{\Pi}_p \otimes \bm{I}_d)^T & \widehat{\bm{\Gamma}}(\bm{y})-\sum_{k=1}^{p-1}\widehat{\bm{C}}_k 
\end{bmatrix} \succeq 0, \forall d \in \mathbb{N}.
\end{aligned}
\label{ADmSOS-post-infinite}
\end{equation}
\label{corr-ADmSOS-post-infinite}
\end{corollary}

\begin{proof}
   Straightforward, we can show the equivalence between \eqref{ADmSOS-post-infinite} and \eqref{mSOS-infinite} by using Theorem \ref{AD-putinar} and Lemma \ref {lemma-correspondance-sigmas}.
\end{proof}

We now consider finite-dimensional relaxations of \eqref{ADmSOS-post-infinite}. Let $r\geq r_{\min} $ where $r_{\min} $ is defined in Section \ref{sec:mSOS}.  The resulting arrow-decomposed hierarchy of finite-dimensional semidefinite programs associated with \eqref{ADmSOS-post-infinite} is given by:

        \begin{equation}
\tag{ADmSOS$_{2}$}\begin{aligned}
  p_{\text{AD},r}^{\text{post}}=&\underset{\substack{\y, \widehat{\bm{D}}, \widehat{\bm{C}}}}{\min} \mathscr{L}_{\bm{y}}\left(p \right),  \\
 \text{s.t. }& y_{\bm{0}}=1, \\
 & \M_r(\y) \succeq 0, \\
 &\M_{r-r_G}(\G_{k} \y)+ \begin{bmatrix}
   \bm{0} & (\bm{\Pi}_k \otimes \bm{I}_{L_{r-r_{G}}})\widehat{\bm{D}} \\
 \widehat{\bm{D}}^T(\bm{\Pi}_k \otimes \bm{I}_{L_{r-r_G}})^T & \widehat{\bm{C}}_k 
\end{bmatrix} \succeq 0, k \in [p-1], \\
 &\M_{r-r_G}(\G_{p} \y)+ \begin{bmatrix}
   \bm{0} & (\bm{\Pi}_p \otimes \bm{I}_{L_{r-r_G}})\widehat{\bm{D}} \\
 \widehat{\bm{D}}^T(\bm{\Pi}_p \otimes \bm{I}_{L_{r-r_G}})^T & \widehat{\bm{\Gamma}}(\bm{y})-\sum_{k=1}^{p-1}\widehat{\bm{C}}_k
\end{bmatrix} \succeq 0.
\end{aligned}
\label{AD-posterior-lassereSOS}
\end{equation}

       %
        %
 

Next, we show that the relaxations \eqref{AD-posterior-lassereSOS} are at least as tight as \eqref{mSOS-finite}, and therefore converge as $r \rightarrow \infty$.
\begin{theorem}
Suppose that Assumptions \ref{assumption-archimedean} and \ref{Assumption-5} hold. Then for any relaxation order $r \geq r_{\text{min}}$, we have $ p_{r} \leq p_{\text{AD},r}^{\text{post}} \leq p^*$. Moreover, we have $\lim_{r \rightarrow \infty } p_{\text{AD},r}^{\text{post}}=\lim_{r \rightarrow \infty } p_{r}=p^*$.
    \label{Arrow-decomposition-PMI-posterior-mSOS} 
\end{theorem}

\begin{proof}
If $(\bm{y}^*,\widehat{\bm{D}}^*, \widehat{\bm{C}}^*)$ is a solution to \eqref{AD-posterior-lassereSOS}, then $\bm{y}^*$ is feasible for \eqref{mSOS-finite}, so $f_{\text{AD},r}^{\text{post}} \geq f_{r}$. 

Now, if $\bm{x}^*$ is a solution to \eqref{POP-PMI}, then by Theorem \ref{putinar}, one can find a moment vector $\bm{y}^*$ solution to Problem \ref{mSOS-infinite}. By Theorem \ref{AD-putinar}, we can find $(\widehat{\bm{D}}^*, \widehat{\bm{C}}^*)$ such that $(\bm{y}^*,\widehat{\bm{D}}^*, \widehat{\bm{C}}^*)$ is a solution to \eqref{AD-posterior-lassereSOS} so that $ f_{\text{AD},r}^{\text{post}} \leq f^*$. Since $\lim_{r \rightarrow \infty} f_r=f^*$ by Theorem \ref{mSOS-convergence}, we conclude that $\lim_{r \rightarrow \infty } f_{\text{AD},r}^{\text{post}}=\lim_{r \rightarrow \infty } f_{r}=f^*$.
\end{proof}

   Let us notice that the approach presented in this section preserves the applicability of the rank flatness condition and the extraction of minimizers presented in Section \ref{sec:mSOS}. Indeed, if $\y^*$ is a solution of \eqref{AD-posterior-lassereSOS} at a given relaxation order $r$,  we have $\M_r(\y^*) \succeq 0$ and $$\M_{r-r_{G}}(\G\y^*)=\sum_{k=1}^p\left(\M_{r-r_{G}}(\G_k\y^*) +\begin{bmatrix}
   \bm{0} & (\bm{\Pi}_k \otimes \bm{I}_{L_{r-r_{G}}})\widehat{\bm{D}} \\
 \widehat{\bm{D}}^T(\bm{\Pi}_k \otimes \bm{I}_{L_{r-r_G}})^T & \widehat{\bm{C}}_k 
\end{bmatrix} \right)\succeq 0,$$
allowing us to get the full information on $\y^*$ to test the condition \eqref{flatness}.

    \subsection{Computational considerations}
    \label{computation-consideration-msos}
    Here, as in Section \ref{computational-consideration-lmi}, we aim to deal with potential rank deficiencies in the matrices $\widehat{\bm{A}}_k(\bm{y})$. As the relaxation degree increases, the size of these matrices grows significantly, making it computationally expensive to test for rank deficiency or to compute their range and null spaces. This increased complexity poses practical challenges for directly implementing Proposition \ref{proposition-rank-defficiency}.
    
    In this section, we consider a special class of polynomial matrices $\bm{A}_k(\bm{x})$ whose range and null spaces are constant (i.e., independent of $\bm{x}$). For such matrices, we can derive a weaker version of Proposition \ref{proposition-rank-defficiency} applicable to the associated matrices $ \widehat{\bm{A}}_k(\bm{y})$, by leveraging the known range and null spaces of $\bm{A}_k(\x)$.
    
The first two results establish a link between the null space (respectively range space) of a given matrix $\bm{A}(\x)$ and its associated localizing matrix $ \widehat{\A}(\y) $.  
      \begin{lemma}
     Let $\bm{A} \in \mathbb{S}^n[\bm{x}]$, $\bm{b}(\bm{x})$ be a polynomial vector of length $L$, and  $\bm{y}=(y_{\bm{\alpha}})_{\bm{\alpha} }$ a given vector of length $L$. We have 
     $$\text{Null}(\bm{A(x)} \otimes \bm{I}_L) \subseteq \text{Null}(\mathscr{L}_y(\bm{A(x)}\otimes \bm{b(x)}^T\bm{b(x)})).$$
     \label{rank-defficiency-moment}
    \end{lemma}
    
    \begin{proof}
        Let $\bm{w} \in \text{Null}(\bm{A(x)})$. We have $\bm{w} \otimes \bm{I}_L \in \text{Null}(\bm{A(x)}\otimes \bm{I}_L)$ and 
        $$ \mathscr{L}_y(\bm{A}(\x)\otimes \bm{b}(\x)^T\bm{b}(\x))(\bm{w} \otimes \bm{I}_L)=\mathscr{L}_y(\bm{A}(\x)\bm{w}\otimes \bm{b}(\x)^T\bm{b}(\x))=0.$$
    \end{proof}
        
     \begin{lemma}
      Let $\bm{A} \in \mathbb{S}^n[\bm{x}]$ and $\bm{P} \in \mathbb{R}^{n \times s}$, $s \leq n$ with full column rank and such that $\bm{P}^T\bm{P}=\bm{I}_{s}$. \footnote{This result can be proved in a more general context, without assuming that $\bm{P}$ has full column rank or that  $\bm{P}^T\bm{P}=\bm{I}_{s}$. However, since the main result in Proposition \ref{proposition-projection-msos} always relies on these additional assumptions, a general proof is beyond the scope of this paper.} Suppose that $\text{Span} (\bm{A}(\bm{x})) \subseteq \text{Span}(\bm{P})$ for all $\bm{x} \in \mathbb{R}^{n_x}$. Let $\bm{b}(\bm{x})$ be a polynomial vector of length $L$, and  $\bm{y}=(y_{\bm{\alpha}})_{\bm{\alpha} }$ a given vector of length $L$. We have $$\text{Span}\left( \mathscr{L}_{\bm{y}}(\bm{A(x)} \otimes \bm{b}(\bm{x})^T\bm{b}(\bm{x})) \right) \subseteq \text{Span}\left( \bm{P} \otimes \bm{I}_L \right).$$ 
      \label{lemma-inclusion-msos}
     \end{lemma}
    
    \begin{proof}
    Let $\bm{x} \in \mathbb{R}^{n_x}$ and $\bm{v} \in \text{Span}\left( \mathscr{L}_{\bm{y}}(\bm{A(x)} \otimes \bm{b}(\bm{x})^T\bm{b}(\bm{x})) \right) $. There exists a vector $\bm{u} \in \mathbb{R}^{nL}$ such that $ \bm{v}=\mathscr{L}_{\bm{y}}(\bm{A(x)} \otimes \bm{b}(\bm{x})^T\bm{b}(\bm{x}))\bm{u}$. We denote by $\bm{U}$ and $\bm{V}$ the matrices such that $\bm{u}=\text{Vec}(\bm{U})$ and $\bm{v}=\text{Vec}(\bm{V})$, where recall $\text{Vec}$ is the vectorization operator. Using Kronecker product and vectorization operator properties, we have 
$\bm{A}(\x) \otimes \bm{b}(\bm{x})^T\bm{b}(\bm{x})\bm{u}=\text{Vec}(\bm{b}(\bm{x})^T\bm{b}(\bm{x}) \bm{U} \bm{A}(\bm{x}))$ and thus $$\bm{V}=\mathscr{L}_{\bm{y}}(\bm{b}(\bm{x})^T\bm{b}(\bm{x}) \bm{U} \bm{A}(\x) ).$$

Now, since $ \text{Span}(\bm{A}(\bm{x})) \subseteq  \text{Span}(\bm{P})$, 
there exists a matrix $\bm{W} \in \mathbb{R}^{L \times s}[\bm{x}]$ such that $ \bm{b}(\bm{x})^T\bm{b}(\bm{x}) \bm{U} \bm{A}(\bm{x})=\bm{W}(\x)\bm{P}^T$. We have then $ \bm{V}=\mathscr{L}_{\bm{y}}(\bm{W}(\x) \bm{P}^T)=\mathscr{L}_{\bm{y}}(\bm{I}_L\bm{W}(\x) \bm{P}^T)$, and therefore $$ \bm{v}=(\bm{P} \otimes \bm{I}_L)\mathscr{L}_{\bm{y}}(\bm{w(x)})$$ with $ \bm{w(x)}=\text{Vec}(\bm{W}(\x))$, proving that $\bm{v} \in \text{Span}\left( \bm{P} \otimes \bm{I}_L \right)$.
    \end{proof}

 We now state the main result of this section.    
   \begin{proposition}
    Suppose that Assumption \ref{Assumption-5} holds. For each $k \in [p]$, let $\bm{P}_k \in \mathbb{R}^{n_k \times s_k}$, $s_k \leq n_k$, be a matrix with full column rank satisfying $\bm{P}_k^T\bm{P}_k=\bm{I}_{s_k}$ and such that for all $\x \in \mathbb{R}^{n_x}$, we have $ \text{Span}(\bm{A}_k(\bm{x})) \subseteq \text{Span}(\bm{P}_k)$. Moreover, suppose that the vector $\bm{y}=(y_{\bm{\alpha}})_{\bm{\alpha}}$ has a representing measure supported on $\mathcal{K}$. Then the condition \eqref{condition-rank-def-moment}
    is equivalent to the conditions \begin{equation}
    \begin{aligned}
       & \begin{bmatrix}
        \widehat{\bm{P}}_k^T\widehat{\bm{A}}_k(\bm{y})\widehat{\bm{P}}_k & \widehat{\bm{P}}_k^T(\widehat{\bm{B}}_k(\bm{y})+(\bm{\Pi}_k \otimes \bm{I}_{L_d})\widehat{\bm{D}}) \\ (\widehat{\bm{B}}_k(\bm{y})+(\bm{\Pi}_k \otimes \bm{I}_{L_d})\widehat{\bm{D}})^T\widehat{\bm{P}}_k & \widehat{\bm{C}}_k
    \end{bmatrix} \succeq 0, \forall k \in [p-1], \\
    &  \begin{bmatrix}\widehat{\bm{P}}_p^T\widehat{\bm{A}}_p(\bm{y})\widehat{\bm{P}}_p & \widehat{\bm{P}}_k^T(\widehat{\bm{B}}_k(\bm{y})+(\bm{\Pi}_k \otimes \bm{I}_{L_d})\widehat{\bm{D}}) \\ (\widehat{\bm{B}}_p(\bm{y})+(\bm{\Pi}_p \otimes \bm{I}_{L_d})\widehat{\bm{D}})^T\widehat{\bm{P}}_p & \widehat{\bm{\Gamma}}(\y)-\sum_{k=1}^{p-1}\widehat{\bm{C}}_k
    \end{bmatrix} \succeq 0, \\
     &\text{Span}(\widehat{\bm{B}}_k(\y)+ (\bm{\Pi}_k\otimes \bm{I}_{L_d})\widehat{\D})\subseteq \text{Span}(\widehat{\bm{P}}_k), \forall k \in [p].
       \end{aligned}
    \label{LMI-projected-msos}
    \end{equation} 
    with $\widehat{\bm{P}}_k=\bm{P}_k \otimes \bm{I}_{L_d}$.
    \label{proposition-projection-msos}
    \end{proposition}
    
    \begin{proof}
        Since $\bm{y}$ has a representing measure on $\mathcal{K}$, we have that $\widehat{\bm{A}}_k(\bm{y}) \succeq 0 $ for all $k \in [p]$ ( see Proof of Theorem \ref{AD-putinar}). The rest of the proof is due to Proposition \ref{proposition-rank-defficiency} and Lemma \ref{lemma-inclusion-msos}.
    \end{proof}


    Proposition~\ref{proposition-projection-msos} is particularly useful when $\A_k(\x)$ is rank-deficient for some $k \in [p]$. In such cases, Lemma~\ref{rank-defficiency-moment} ensures that the associated localizing matrix is also rank-deficient. Consequently, one can find $\bm{P}_k$ satisfying the assumptions of Proposition~\ref{proposition-projection-msos} with $s_k<n_k$, thereby reducing the size of $\widehat{\bm{G}}_k$ while imposing the condition $\text{Span}(\widehat{\bm{B}}_k(\y)+ (\bm{\Pi}_k\otimes \bm{I})\widehat{\D})\subseteq \text{Span}(\widehat{\bm{P}}_k)$. This condition is equivalent to $\widehat{\bm{B}}_k(\bm{y})+(\bm{\Pi}_k \otimes \bm{I}_{L_d})\widehat{\bm{D}} \in \text{Im} (\widehat{\bm{P}}_k(\bm{y})) $ and it suffices to replace it by $\widehat{\bm{B}}_k(\bm{y})+(\bm{\Pi}_k \otimes \bm{I}_{L_d})\widehat{\bm{D}} \in \text{Im} (\widehat{\bm{A}}_k(\bm{y})) $. In applications where $\A_k(\x)$ has a constant null space, i.e. there exists a subspace $\mathcal{N}_k \subset \mathbb{R}^n$ such that $ \mathcal{N}_k ~\bot ~\text{Span}(\bm{P}_k)$ and $\text{Null}(\bm{A}(\x)) =\mathcal{N}_k$ for all $\x \in \mathcal{K}$ (see, e.g., Section \ref{sec:po_variable_reduction}), the condition $\widehat{\bm{B}}_k(\bm{y})+(\bm{\Pi}_k \otimes \bm{I}_{L_d})\widehat{\bm{D}} \in \text{Im} (\widehat{\bm{A}}_k(\bm{y})) $ can be enforced implicitly by solving for $\widehat{\D}$
      the system $$ \left\lbrace\left(\widehat{\bm{B}}_k(\bm{y})+(\bm{\Pi}_k \otimes \bm{I}_{L_d})\widehat{\bm{D}}\right)(\bm{W}_k \otimes \bm{I}_{L_d})=0, k \in [p] \right\rbrace,$$
     where $\bm{W}_k$ spans $\mathcal{N}_k $. This procedure also allows eliminating some (or all) of the entries of $\widehat{\bm{D}}$.

\subsection{A connection between AD posterior and AD prior the mSOS}
\label{appendix:connection-prior-posterior}

In this section, we emphasize that the problem \eqref{AD-posterior-lassereSOS} can also be obtained from the problem \eqref{AD-prior-mSOS} by considering the additional variables $\bm{D}$ and $\bm{C}$ 
as external variables to the hierarchy. More precisely, we consider the problem: 
\begin{equation*}
\begin{aligned}
 &\underset{\substack{\bm{y}, \bm{D}, \bm{C}}}{\min} \mathscr{L}_{\bm{y}}\left(f \right),  \\
 \text{s.t. }& y_{\bm{0}}=1, \\
 & \bm{M}_r(\bm{y}) \succeq 0, \\
 &\bm{M}_{r-r_G}(\bm{G}_{k}\bm{y})+ \bm{M}_{r-r_G}\left( \begin{bmatrix}
   \bm{0} & \bm{\Pi}_k\bm{D} \\
  \bm{D}^T\bm{\Pi}_k^T & \bm{C}_k
\end{bmatrix}\bm{y} \right) \succeq 0,\ \ k \in [p-1], \\
&\bm{M}_{r-r_G}(\bm{G}_{p}\bm{y})+ \bm{M}_{r-r_G}\left( \begin{bmatrix}
   \bm{0} & \bm{\Pi}_p\bm{D} \\
  \bm{D}^T\bm{\Pi}_p^T & \widehat{\bm{\Gamma}}(\bm{y})-\sum_{k=1}^{p-1}\bm{C}_k
\end{bmatrix}\bm{y} \right) \succeq 0.
\end{aligned}
\label{AD-prior-posterior-mSOS}
\end{equation*}
This problem is non-linear because of the terms 
$$ \begin{aligned}
  \bm{M}_{r-r_G}&\left( \begin{bmatrix}
   \bm{0} & \bm{\Pi}_k\bm{D} \\
  \bm{D}^T\bm{\Pi}_k^T & \bm{C}_k
\end{bmatrix}\bm{y} \right)= \\&\begin{bmatrix}
   \bm{0} & \bm{\Pi}_k\bm{D}\otimes \mathscr{L}_{\bm{y}}(\bm{b}_{r-r_G}(\bm{x})^T\bm{b}_{r-r_G}(\bm{x})) \\
  \mathscr{L}_{\bm{y}}(\bm{b}_{r-r_G}(\bm{x})^T\bm{b}_{r-r_G}(\bm{x}))\otimes \bm{D}^T\bm{\Pi}_k^T & \bm{C}_k\otimes \mathscr{L}_{\bm{y}}(\bm{b}_{r-r_G}(\bm{x})^T\bm{b}_{r-r_G}(\bm{x}))
\end{bmatrix}. \end{aligned}$$ However, we can see the matrix variables $\widehat{\bm{D}}$ and  $\widehat{\bm{C}}$ in Problem \eqref{AD-posterior-lassereSOS}  as a component-wise linearization of the previous non-linear terms, i.e., if we replace each entry\\ $\left[\bm{\Pi}_k\bm{D}\otimes \mathscr{L}_{\bm{y}}(\bm{b}_{r-r_G}(\bm{x})^T\bm{b}_{r-r_G}(\bm{x}))\right]_{ij} $ with $(\widehat{\bm{D}})_{ij}$, and each entry $\left[\bm{C}_k\otimes \mathscr{L}_{\bm{y}}(\bm{b}_{r-r_G}(\bm{x})^T\bm{b}_{r-r_G}(\bm{x}))\right]_{ij} $ with $(\widehat{\bm{C}}_k)_{ij}$, we recover exactly the problem \eqref{AD-posterior-lassereSOS}. 

Consequently, this shows that this "relaxation of relaxation" hierarchy is monotonically convergent with the relaxation order $r$, thanks to Theorem \ref{Arrow-decomposition-PMI-posterior-mSOS}. 
Thus, there is no need to add norm constraints to the additional variables in \eqref{AD-prior-mSOS} to maintain convergence.

    {\color{red}

}
\section{Application: topology optimization of frame structures}
\label{sec:frames}

This section is devoted to illustrating the arrow decomposition method when applied to topology optimization problems. In topology optimization problems, we aim to design structures that provide the optimal mechanical response to external inputs such as forces. Consider a discretized design domain comprising $n_\mathrm{n} \in \mathbb{N}$ nodes (representing joints) connected by $n_\mathrm{e} \in \mathbb{N}$ potential Euler-Bernoulli finite elements. External forces $\bm{f} \in \mathbb{R}^{n_\mathrm{dof}}$ are applied to specific nodes of the structure, with $n_\mathrm{dof}$ denoting the number of degrees of freedom. The optimization task involves finding the continuous cross-sectional areas $\bm{x} \in \mathbb{R}^{n_\mathrm{e}}_{\ge0}$ of individual finite elements that yield the most efficient design. This efficiency is evaluated with respect to two functions: the compliance $\gamma(\bm{x})$, which characterizes the inverse of the structure's stiffness, and the weight $w(\bm{x})$, which quantifies the total amount of material used.

For simplicity, here we consider a single-load scenario, i.e., there is only one set of loads applied to the structure. We refer the reader to \cite{tyburec2021global,tyburec2022global} for an extension to the more general case of multiple loading scenarios. At each element $e$, the local displacement vector $\bm{u}_e$ corresponds to three degrees of freedom: translations in the $x$- and $y$-directions, and a rotation. The global displacement vector of the structure, denoted by $\bm{u}$, is obtained by concatenating the local displacements $\bm{u}_e$ across all elements $e$.  Therefore, here, we define the compliance function as the potential energy of the loads $\bm{f}$ as

\begin{equation}
    \gamma := \bm{f}^\mathrm{T} \bm{u},
    \label{compliance-definition}
\end{equation}
in which $\bm{u} \in \mathbb{R}^{n_\mathrm{dof}}$ solves the elastic equilibrium
\begin{equation}
    \bm{K} (\bm{x}) \bm{u} = \bm{f}(\x).
    \label{elasticity-equation}
\end{equation}
In \eqref{elasticity-equation}, $\bm{K}(\bm{x}) \in \mathbb{S}^{n_{\text{dof}}}$ is the corresponding symmetric positive semidefinite stiffness matrix assembled as
\begin{equation}
    \bm{K} (\bm{x})=\bm{K}_{0}+\sum_{e=1}^{n_e} \left[\bm{K}_{e}^{(1)}x_e+\bm{K}_{e}^{(2)}x_e^2+\bm{K}_{e}^{(3)}x_e^3 \right].
     \label{stifness_matrix}
 \end{equation}
Here, the matrix $ \bm{K}_{0} \succeq 0$ constitutes a design-independent structural stiffness, $\bm{K}_e^{(1)} $ represents the element $e$ membrane stiffness, and the matrices $\bm{K}_e^{(2)}$ and $\bm{K}_e^{(3)}$ encompass the bending effects. By construction, we have $ \bm{K}_{e}^{(i)} \succeq 0$ for all $e$ and $ i \in \lbrace 1, 2, 3 \rbrace $. Moreover, we assume that $\forall \bm{x} >\bm{0}: \bm{K}(\bm{x}) \succ 0$ as, otherwise, the structure would behave as a rigid body mechanism.

\subsection{Compliance minimization problem}
Given a maximum weight $\overline{w} \in \mathbb{R}_{> 0}$, finding the minimum-compliant structure is formulated as
\begin{minie}|s|%
{\strut \bm{x}, \bm{u}}%
 { \bm{f}(\bm{x})^T\bm{u}\label{compliance_obj}}%
{\label{compliance_problem}}%
{\gamma^*=}
\addConstraint{\bm{K}(\bm{x})\bm{u}-\bm{f}(\bm{x})}{= \bm{0}\label{compliance_const1}}
\addConstraint{ \overline{w}-\sum_{e=1}^{n_\mathrm{e}}\ell_e \rho_e x_e}{\geq  0 \label{compliance_const2}}
\addConstraint{\bm{x}}{\geq \bm{0}, \label{compliance_const3}}
\end{minie}
where the weight $w(\bm{x}) = \sum_{e=1}^{n_\mathrm{e}} \ell_e \rho_e x_e$ is computed from the element lengths $\bm{\ell} \in \mathbb{R}^{n_\mathrm{e}}_{>0}$ and their densities $\bm{\rho} \in \mathbb{R}^{n_\mathrm{e}}_{>0}$.

The state variables $\bm{u}$ can be eliminated from the formulations by introducing a slack variable $\gamma$ and by using Equations \eqref{elasticity-equation} and \eqref{compliance-definition}, leading to the following problem 

\begin{minie}|s|%
{\strut \bm{x}, \gamma}%
 { \gamma\label{compliance_obj_eliminating_displacements}}%
{\label{compliance_problem_eliminating_displacements}}%
{}
\addConstraint{\gamma-\bm{f}(\bm{x})^T[\bm{K}(\bm{x})]^\dagger\bm{f}(\bm{x})}{\geq \bm{0}\label{compliance_const_eliminating_displacements_1}}
\addConstraint{ \overline{w}-\sum_{e=1}^{n_\mathrm{e}}\ell_e \rho_e x_e}{\geq  0 \label{compliance_const_eliminating_displacements_2}}
\addConstraint{\bm{x}}{\geq \bm{0} \label{compliance_const_eliminating_displacements_3}} 
\addConstraint{\bm{f}(\bm{x})}{\in \text{Im}(\bm{K}(\bm{x})), \label{compliance_const_eliminating_displacements_4}}
\end{minie}
where the constraint \eqref{compliance_const_eliminating_displacements_4} is a necessary and sufficient condition to have $\bm{u}=\bm{K}(\bm{x})^\dagger\bm{f}(\bm{x})$ (see \cite[Lemma 1]{tyburec2021global}). Moreover, it is shown in \cite[Proposition 1]{tyburec2021global} that the condition \eqref{compliance_const_eliminating_displacements_4} is equivalent to the condition 
\begin{equation}
    \bm{K}(\bm{x})[\bm{K}(\bm{x})]^\dagger \bm{f}(\bm{x})=\bm{f}(\bm{x}).
    \label{equivalence-image}
\end{equation}
Using Lemma \ref{shur-complement}, we can derive the equivalent polynomial semidefinite problem
\begin{mini!}|s|%
{\strut \bm{x}, \gamma}%
 {  \gamma \label{compliance_obj_sdp}}%
{\label{compliance_problem_sdp}}%
{\gamma^*=}
\addConstraint{\begin{pmatrix}
    \bm{K}(\bm{x}) & -\bm{f}(\bm{x})^T \\ -\bm{f}(\bm{x}) & \gamma
\end{pmatrix}}{\succeq 0 \label{compliance_const1_sdp}}
\addConstraint{\overline{w}-\sum_{e=1}^{n_\mathrm{e}} \ell_e \rho_e x_e}{\ge  0 \label{compliance_const2_sdp}}
\addConstraint{\bm{x}}{\geq \bm{0},\label{compliance_const3_sdp}}
\end{mini!}
in which the polynomial matrix inequality \eqref{compliance_const1_sdp} enforces the equilibrium.

\subsubsection{Weight minimization problem}
Similarly to the compliance minimization problem, we can write a polynomial SDP formulation of the weight minimization problem only in terms of $\bm{x}$ as
\begin{mini!}|s|%
{\strut \bm{x}}%
 { \sum_{e=1}^{n_\mathrm{e}} \ell_e \rho_e x_e \label{weight_obj_sdp}}%
{\label{weight_problem_sdp}}%
{}
\addConstraint{\begin{pmatrix}
    \bm{K}(\bm{x})& -\bm{f}(\bm{x}) \\ -\bm{f}(\bm{x})^T &  \overline{\gamma}
\end{pmatrix}}{\succeq 0\label{weight_const1_sdp}}
\addConstraint{\bm{x}}{\geq \bm{0},\label{weight_const2_sdp}}
\end{mini!}
where $\overline{\gamma}$ here represents a fixed upper bound for the compliance.
Moreover, for both problems \eqref{compliance_problem_sdp} and \eqref{weight_problem_sdp}, it is shown in \cite{tyburec2022global,tyburec2021global} that the problem \eqref{compliance_problem_sdp} can be compactified and individual variables scaled to the $[-1,1]$ domain. On the one hand, this improves numerical performance, and, on the other hand, ensures that these problems satisfy the Archimedean assumption \ref{assumption-archimedean}. Thus, assumptions for the convergence of the moment SOS hierarchy are satisfied (Section \ref{sec:mSOS}). 

In addition to lower bounds obtained by solving relaxations of the Lasserre hierarchy, in \cite{tyburec2022global,tyburec2021global}, we also derived a simple procedure for evaluating feasible upper bounds. When the gap between these bounds is small, the mSOS provides approximately globally optimal solution.

\subsection{Decomposition of the compliance minimization problem}
In this section, we explicitly constuct the arrow decomposed formulation for the case of compliance minimization \eqref{compliance_problem_sdp}. Because the result for the weight optimization settings \eqref{weight_problem_sdp} is analogous, we do not detail it here.


Let us denote by $\mathcal{D}:= [n_\mathrm{e}]$ the domain of potential elements and partition $\mathcal{D}$ into $p$ subdomains $\mathcal{D}_k$, $k \in [p]$, where each $\mathcal{D}_k$ can contain one or more elements $e$. For each subdomain $\mathcal{D}_k$, we thus receive the stiffness matrix $\bm{K}_{k}$ that is built as $\bm{K(x)}=\sum_{k \in D_k}\bm{K}_{k}(\bm{x}).$ Recall that by construction, we have $\bm{K}_k(\bm{x}) \succeq 0$, so that the assumption of Theorem \ref{Kocvara-AD-PSD} holds.

Next, thanks to the link between AD and domain decomposition methods in \cite{kovcvara2021decomposition}, we can define the partitioning of the matrix $\bm{K}_k(\bm{x})$ according to the subdomains $\mathcal{D}_k$. Let $\mathcal{L}_k$ denote the set of indices of degrees of freedom \footnote{Here, the total degrees of freedom of the problem play the role of the index set of the matrix $\bm{K}(\x)$.} located in the interior of $\mathcal{D}_k$, i.e., $\mathcal{L}_k = \mathcal{I}_k \setminus \bigcup_{\ell: \ell>k} \mathcal{I}_{k,\ell}$ 
%
%
and the set $\mathcal{J}_k$ of the indices of all remaining degrees of freedom in one of the intersection $\mathcal{I}_{k,\ell}$ or $\mathcal{I}_{\ell,k}$, i.e.
$\mathcal{J}_k = \bigcup_{\ell: \ell>k} \mathcal{I}_{k,\ell}$. Then, we can permute and partition each matrix $\bm{K}_{k}(\bm{x})$, $\bm{f}(\x)$ as follows:
\begin{equation}
    \bm{K}_{k}(\bm{x})= \begin{bmatrix} \bm{K}_{\mathcal{L}_k\mathcal{L}_k}(\bm{x}) & \bm{K}_{\mathcal{L}_k\mathcal{J}_k}(\bm{x})  \\ \bm{K}_{\mathcal{J}_k\mathcal{L}_k}(\bm{x}) & \bm{K}_{\mathcal{J}_k\mathcal{J}_k}(\bm{x}) 
\end{bmatrix}, ~\bm{f}_k(\x)=\begin{bmatrix}
    \bm{f}_{\mathcal{L}_k} \\ \bm{f}_{\mathcal{J}_k}
\end{bmatrix}.
    \label{stifness-matrices-AD}
\end{equation}
Let $r$ be a relaxation order, and $\y$ a pseudo-moment vector related to the variables $\x$ and $\gamma$. The arrow decomposed mSOS relaxations for the compliance optimization problem \eqref{compliance_const3_sdp} reads as \begin{equation*}
\begin{aligned}
  \gamma_{\text{AD,r}}^*=&\underset{\substack{\y, \bm{d}, \bm{c} }}{\min} \mathscr{L}_{\y}(\gamma),  \\
 \text{s.t. }& y_{\bm{0}}=1, \\
 &\M_r(\y) \succeq 0, \\
 &\M_{r-r_{K}}\left( \bm{Z}_k \y\right)+\begin{bmatrix}
    \bm{0} & (\bm{\Pi}_k\otimes \bm{I}_{L_{r-r_{K}}})\bm{d} \\ \bm{d}^T(\bm{\Pi}_k\otimes \bm{I}_{L_{r-r_{K}}})^T & c_k
\end{bmatrix}\succeq 0, k \in [p-1],\\
& \M_{r-r_{K}}\left( \bm{Z}_p \y\right)+\begin{bmatrix}
     \bm{K}_{p}(\bm{x}) & (\bm{\Pi}_p\otimes \bm{I}_{L_{r-r_{K}}})\bm{d} \\
   \bm{d}^T(\bm{\Pi}_p\otimes \bm{I}_{L_{r-r_{K}}})^T & \gamma- \sum_{k=1}^{p-1} c_k
\end{bmatrix} \succeq 0,\\
 & \M_{r-1}\left((\overline{V}-\sum_{i=1}^{n_e}x_i\ell_i )\y\right) \succeq 0, \\
 & \M_{r-1}(x_i\y) \succeq 0, i \in [n_e],
\end{aligned}
\end{equation*}
with $\bm{Z}_k(\x)=\begin{bmatrix}
   \bm{K}_{k}(\bm{x}) & \bm{f}_{k}(\bm{x}) \\
   \bm{f}_{k}(\bm{x})^T & \bm{0}
\end{bmatrix}$ for all $k \in [p]$, and $r_K=\lceil \frac{\text{deg}(\bm{K})}{2}\rceil$.
The convergence of the sequence $\gamma_{\text{AD,r}}^*$ to $\gamma^*$ as $r \rightarrow \infty$ is ensured by Theorem \ref{Arrow-decomposition-PMI-posterior-mSOS}.

An analysis of the additional variables $\bm{d}_k$ and their physical interpretation in the context of frame structures is provided in Appendix \ref{sec:physical}.

\section{Numerical examples}
\label{sec:numerics}

In this section, we present two numerical examples that demonstrate the computational advantages of arrow decomposition in frame structure optimization problems. We first recall some technical features when working with frame structures and mSOS hierarchies. Specifically, we highlight that a reduced monomial basis can be used to construct the mSOS hierarchies instead of the standard one. Based on the numerical results developed in \cite{handaterm}, we additionally adopt the non-mixed term (NMT) basis, which is defined by removing all mixed products from the standard monomial basis, i.e., $
\bm{b}_{r}^{\mathrm{NMT}}(\bm{x}):= \left\lbrace 1,x_1,\ldots, x_n,x_1^2, \ldots, x_n^2, \ldots, x_1^r, \ldots, x_n^r \right\rbrace.$ Moroever, we define the relative optimality gap $\varepsilon_\mathrm{r} = f_\mathrm{ub}/f_\mathrm{lb} - 1$, where $f_\mathrm{ub}$ is an upper bound. 
For compliance and weight optimization problems, this upper bound is defined in \cite{tyburec2021global} and  \cite{tyburec2022global}, respectively. 
This optimality gap provides a global optimum certificate and serves as a measure of solution quality \cite{tyburec2021global,tyburec2022global,tyburec2024globalweightoptimizationframe}. Furthermore, the optimality gap approach allows us to avoid evaluating the rank of large matrices, which is required for the rank flatness condition in \eqref{flatness}. 

  The first example examines a double-hinged beam under self-weight loading, where we analyze the scalability of AD posterior-to-mSOS approach by varying the number of finite elements from $1$ to $15$. The second example considers a $24$-element modular frame structure, showcasing the method's effectiveness on a more practical engineering problem. For both examples, we compare four solution approaches: standard moment-SOS (mSOS), moment-SOS with arrow decomposition (mSOS+AD), moment-SOS using nonmixed-term basis (mSOS+NMT), and moment-SOS combining both techniques (mSOS+NMT+AD). We evaluate these methods across different relaxation orders, analyzing their performance in terms of computation time, solution quality, and numerical stability. All computations were performed on a workstation equipped with dual Intel\textsuperscript{\textregistered} Xeon\textsuperscript{\textregistered} E5-2630 v3 processors and 128 GB of RAM using the Mosek optimizer with default parameters. Our implementation of the method and numerical examples are publicly available as a MATLAB package at \url{https://gitlab.com/tyburec/pof-dyna}.

\subsection{Double-hinged beam under selfweight}

Consider a symmetric half of a double-hinged beam as our first numerical example, see Fig. \ref{fig:doublehinged}. The beam features a pinned support at its left end restricting both translations, while its right end has a roller support preventing vertical translation and rotation. The beam is made of a linear-elastic material with Young's modulus $E = 210$~GPa and density $\rho = 7850$~kg/m$^3$. With an I-shaped cross-section, as displayed in Fig. \ref{fig:doublehinged}b, the stiffness matrix $\bm{K}(\bm{x})$ contains polynomials of degree at most 2. 
The beam is subjected to two types of loads: a uniform distributed load with intensity $\overline{f} = 1$~kN/m, and self-weight loading that depends on the design variables $\bm{x} \in \mathbb{R}^{n_\mathrm{e}}$ representing the element's cross-sectional areas. The presence of design-dependent self-weight makes this problem numerically more challenging than standard compliance minimization because, typically, a higher relaxation degree is required to solve the problem globally and the relaxation solutions are less accurate.

\begin{figure}[!htbp]
\centering
\begin{subfigure}{0.85\linewidth}
\begin{tikzpicture}
\scaling{2}
\point{a}{0}{0};
\point{b}{1}{0};
\point{c}{2}{0};
\point{d}{3}{0};
\point{e}{4}{0};
\point{f}{5}{0};
\beam{1}{a}{e};
\support{1}{a};
\support{4}{e}[90];
\lineload{2}{a}{e}[0.5][0.5][0.125];
\notation{1}{c}{$f=1$~kN/m}[above=7mm];
\dimensioning{1}{a}{e}{-1.2}[$5$~m];
\end{tikzpicture}
\caption{}
\end{subfigure}%
\hfill\begin{subfigure}{0.15\linewidth}
\begin{tikzpicture}
	\scaling{1.50}
	\point{a}{0.0}{0.0};\point{a0}{0.5}{0.0};
	\point{b}{0.0}{0.1};\point{b0}{0.5}{0.1};
	\point{c}{0.2}{0.1};\point{c0}{0.3}{0.1};
	\point{d}{0.2}{0.9};\point{d0}{0.3}{0.9};
	\point{e}{0.0}{0.9};\point{e0}{0.5}{0.9};
	\point{f}{0.0}{1.0};\point{f0}{0.5}{1.0};
	\draw[black, thick, fill=black!25] (a0) -- (b0) -- (c0) -- (d0) -- (e0) -- (f0) -- (f) -- (e) -- (d) -- (c) -- (b) -- (a) -- cycle;
	\dimensioning{2}{a}{f}{1.25}[$10t_e$];
	\dimensioning{1}{a}{a0}{1.9}[$5t_e$];
	\draw [-stealth](0.15,0.60) -- (0.35,0.85);
	\draw [-stealth](0.15,0.60) -- (0.25,1.45);
	\draw [-stealth](0.15,0.60) -- (0.25,0.05);
	\draw (-0.3,0.6) -- node[above=-0.5mm]{$t_e$}(0.025,0.6) -- (0.15,0.60);
\end{tikzpicture}
\caption{}
\end{subfigure}
\begin{subfigure}{\linewidth}
\centering
\includegraphics[width=0.45\linewidth]{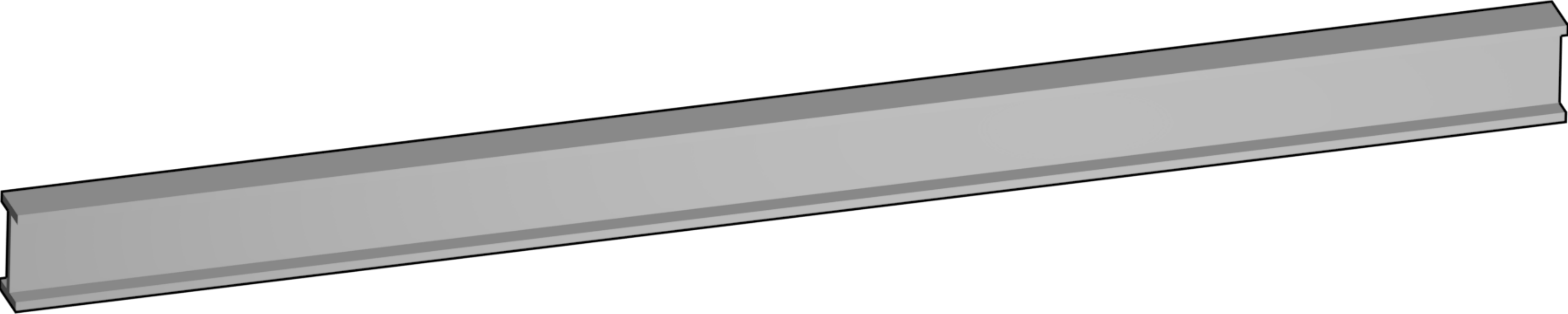}\hfill\includegraphics[width=0.45\linewidth]{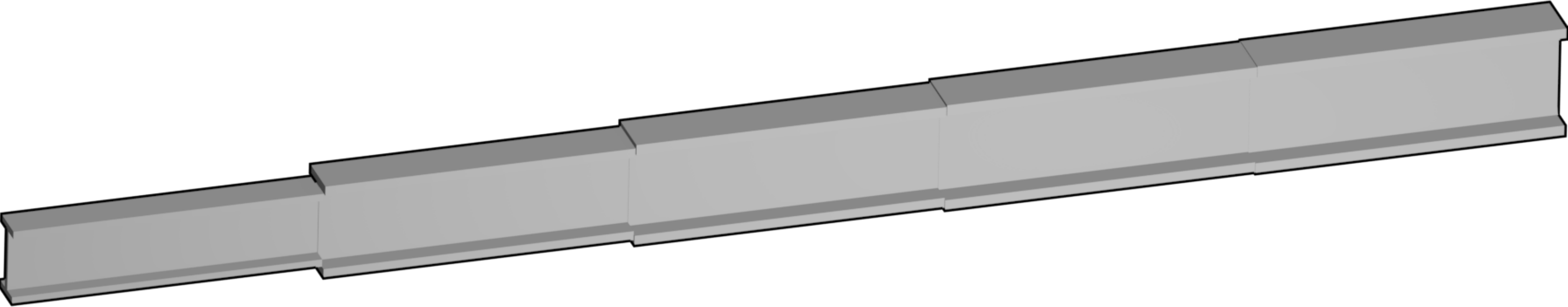}\\
\includegraphics[width=0.45\linewidth]{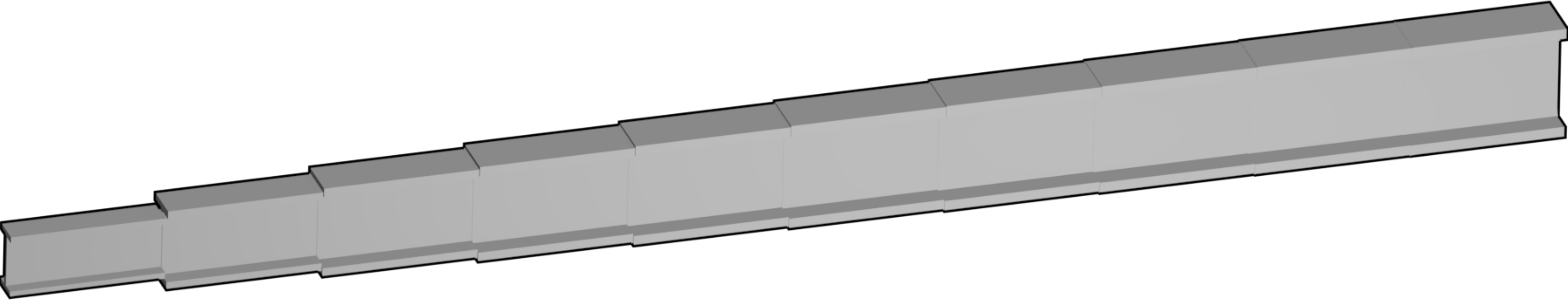}\hfill\includegraphics[width=0.45\linewidth]{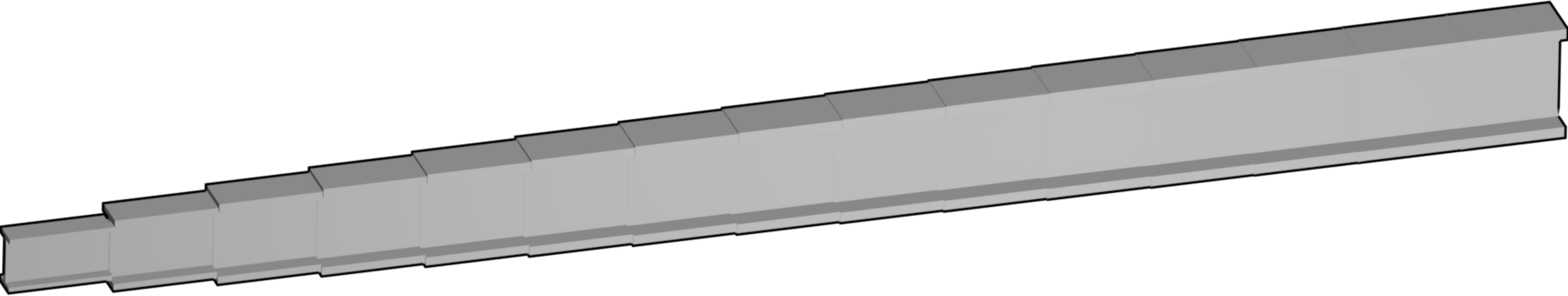}
\caption{}
\end{subfigure}
\caption{Double-hinged beam under selfweight loads: (a) boundary conditions, (b) cross-section parametrization, and (c) approximately globally-optimal designs for $n_\mathrm{e}=\{1,5,10,15\}$.}
\label{fig:doublehinged}
\end{figure}
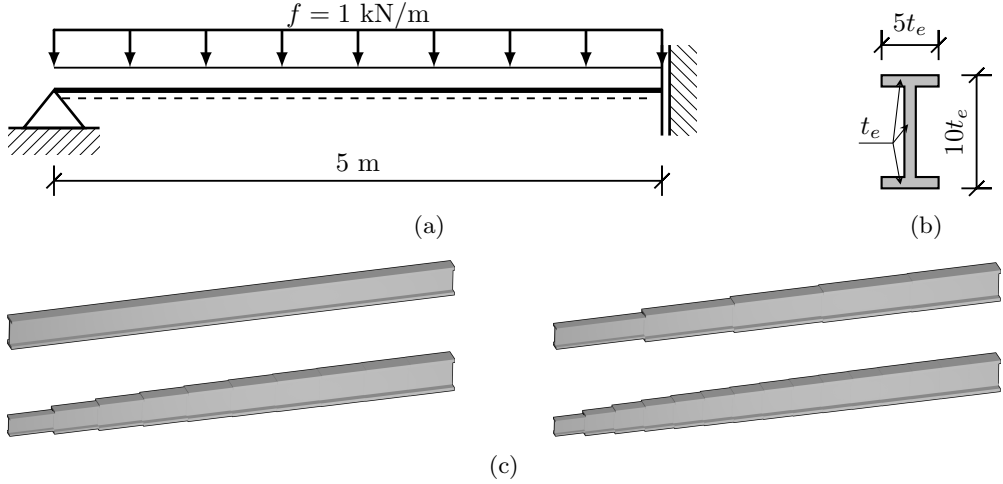

The optimization objective is to minimize the beam's compliance while respecting a weight constraint $\overline{w} = 785$~kg, which corresponds to $0.1$~m$^3$ volume of material. For the numerical study, we discretize the $5$~m long beam half into $n_\mathrm{e}$ equal-length elements, with $n_\mathrm{e}$ ranging from $1$ to $15$. Each element $i$ thus experiences a total distributed load of $f_i(x_i) = 1 + \frac{385.435}{n_\mathrm{e}} x_i$ [kN/m], where the first term represents the external load and the second term accounts for the self-weight, with $x_i$ being the cross-sectional area of element $i$.

For the arrow decomposition, we partition the beam into $n_\mathrm{e}-1$ subdomains, where the first subdomain $\mathcal{D}_1$ contains elements $\{1,2\}$, and each subsequent subdomain $\mathcal{D}_k$ contains a single element $k+1$. This leads to index sets $\mathcal{I}_1 = \{1,2,3,4,5,6,7\}$ containing degrees of freedom associated with the first two elements, $\mathcal{I}_2 = \{5,6,7,8,9,10\}$ for the third element, and so on. The intersections of consecutive subdomains contain three degrees of freedom each, corresponding to the shared nodes between elements. For instance, $\mathcal{I}_{1,2} = \{5,6,7\}$ represents the degrees of freedom at the node shared between the domains $\mathcal{D}_1$ and $\mathcal{D}_2$.

For all discretizations, the variable reduction procedure from Section \ref{sec:po_variable_reduction} leads to $|\bm{D}_\mathrm{red}| = 1$. This corresponds directly to the single degree of statical indeterminacy in the double-hinged beam -- physically, only one redundant force cannot be determined purely from equilibrium equations.

\begin{figure}[htbp]
  \centering
  \subcaptionbox{Solving time in function of the number of elements}[0.5\textwidth]{%
    \includegraphics[width=0.5\textwidth]{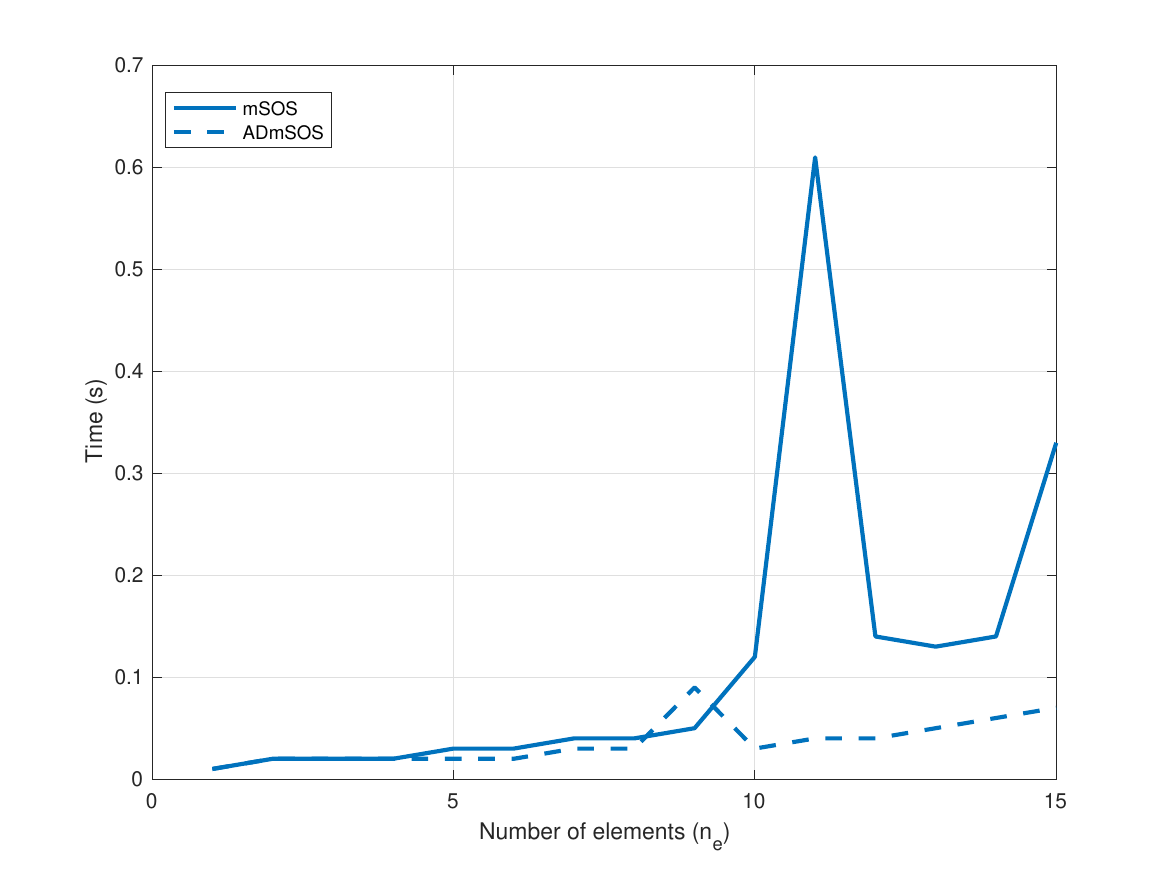}}
 \hspace*{-0.5cm} 
\hfill 
  \subcaptionbox{Total number of relaxation variables and size of the largest SDP matrix}[0.5\textwidth]{%
    \includegraphics[width=0.5\textwidth]{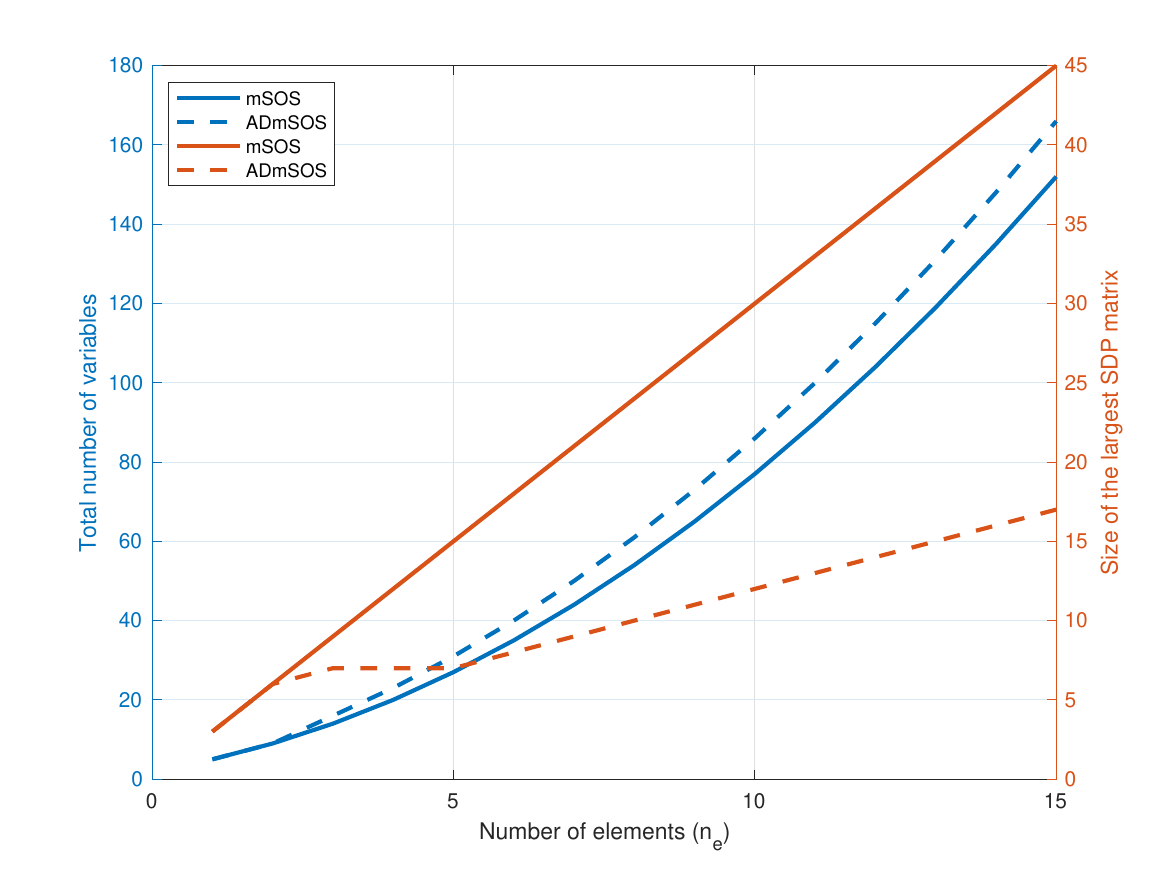}%
\label{fig:sub2}
  }
  \caption{ Comparison of mSOS and ADmSOS for the first-order relaxation of a double-hinged beam problem  }
  \label{fig:mSOS-rel1}
\end{figure}

\begin{figure}[htbp]
  \centering
  \subcaptionbox{Solving time in function of the number of elements}{%
    \includegraphics[width=0.5\textwidth]{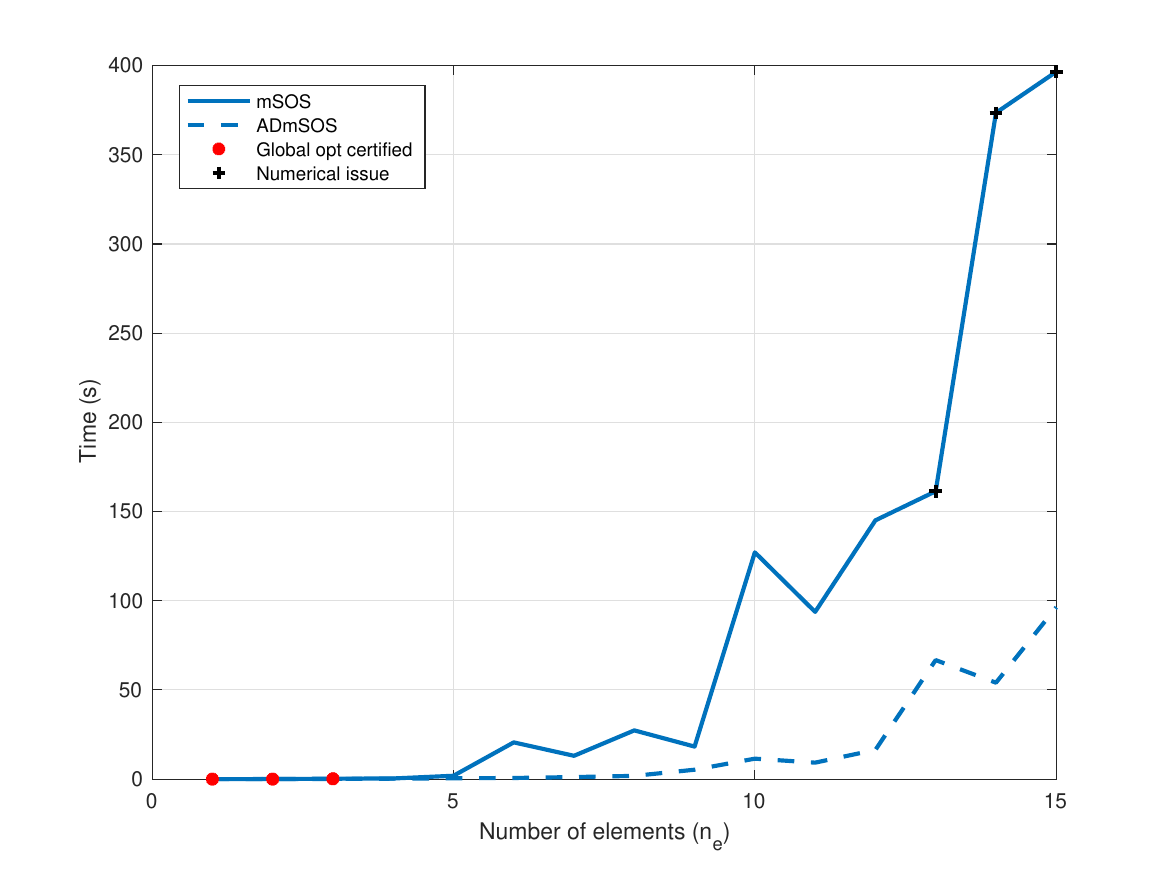}%

  }
 \hspace*{-0.5cm} \hfill
  \subcaptionbox{Total number of relaxation variables and size of the largest SDP matrix}{%
    \includegraphics[width=0.5\textwidth]{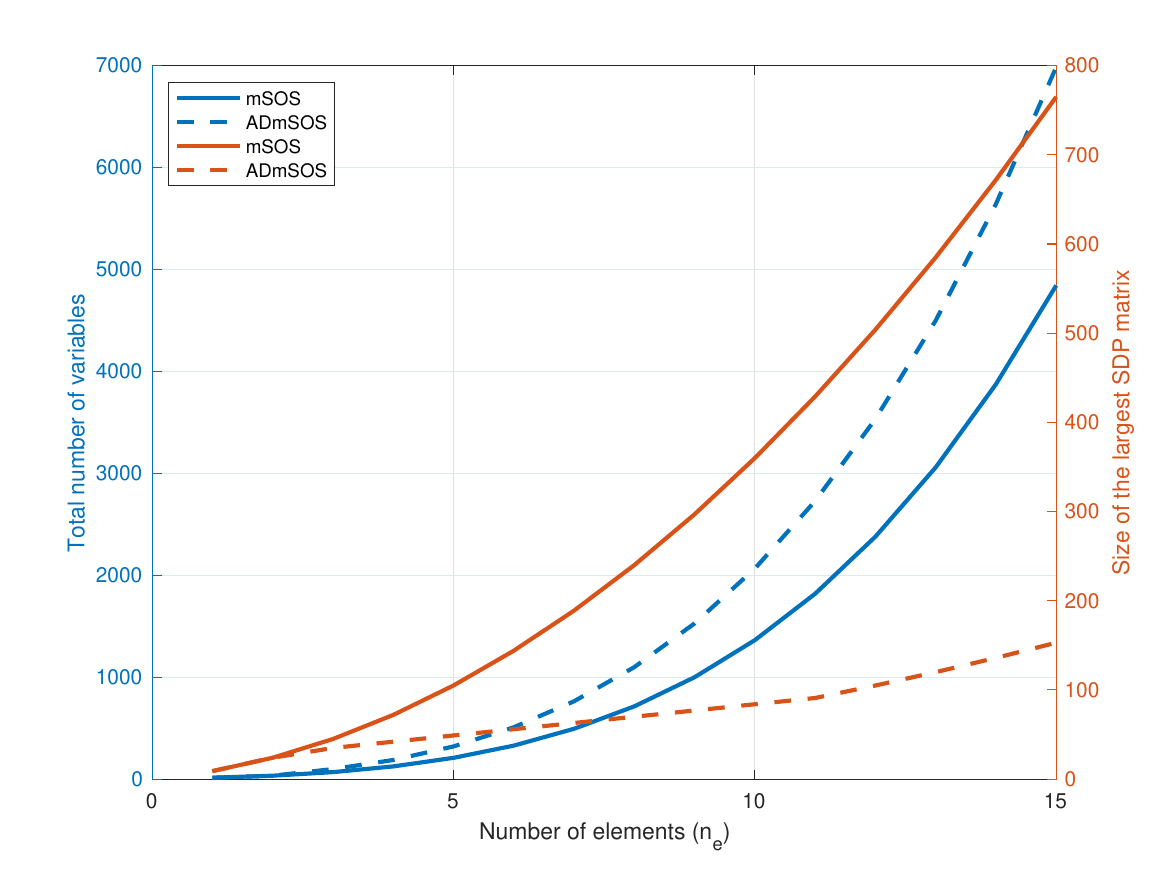}%

  }
  \caption{ Comparison of mSOS and ADmSOS for the second-order relaxation of a double-hinged beam problem  }
  \label{fig:mSOS-rel2}
\end{figure}

\begin{figure}[htbp]
  \centering
  \subcaptionbox{Solving time in function of the number of elements}{%
    \includegraphics[width=0.5\textwidth]{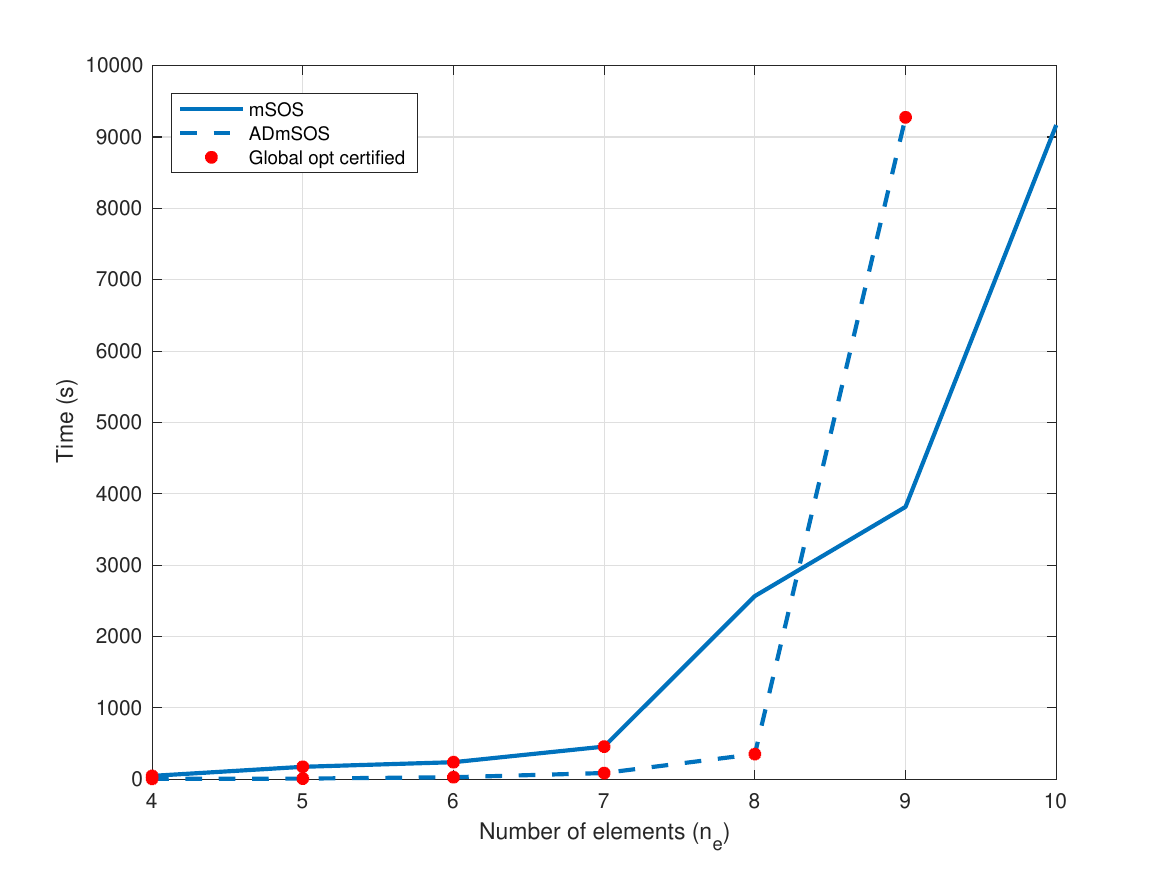}%
  }
 \hspace*{-0.5cm} \hfill
  \subcaptionbox{Total number of relaxation variables and size of the largest SDP matrix}{%
    \includegraphics[width=0.5\textwidth]{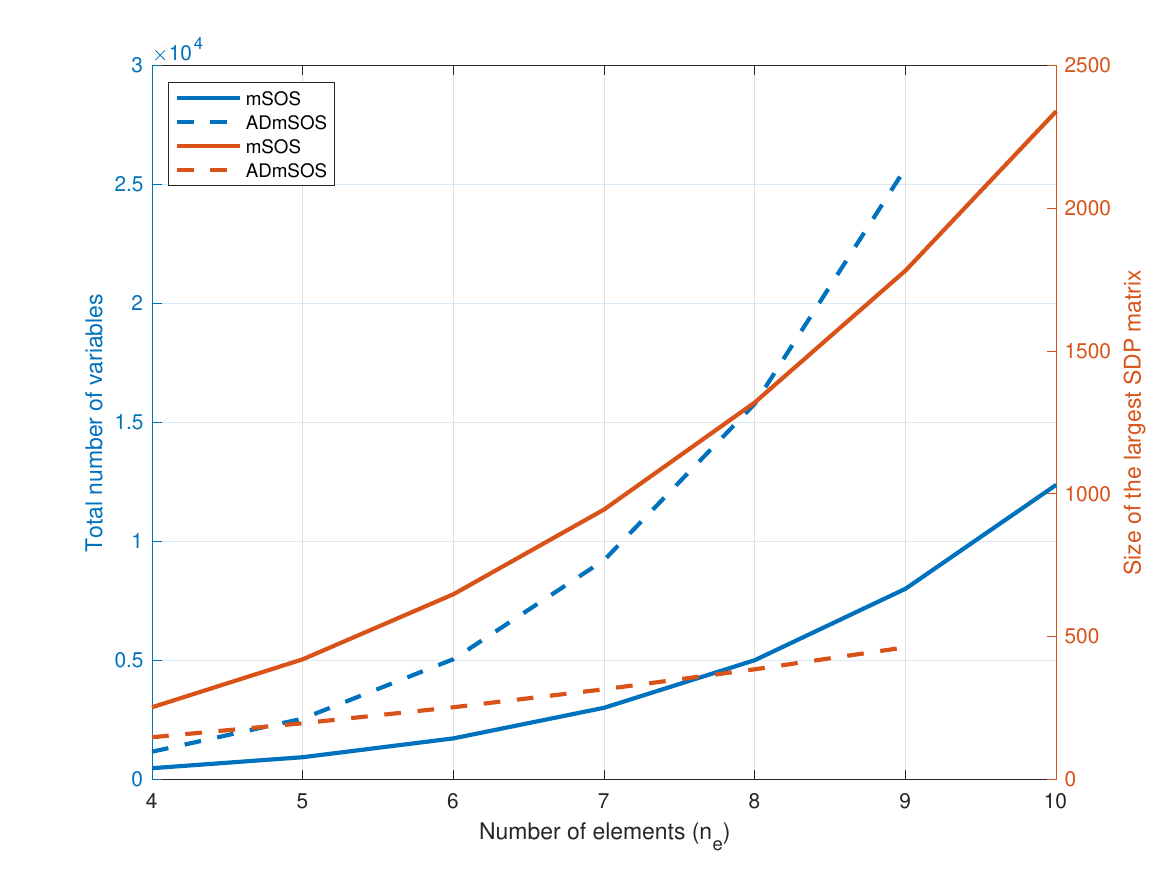}%
  }
  \caption{ Comparison of mSOS and ADmSOS for the third-order relaxation of a double-hinged beam problem  }
  \label{fig:mSOS-rel3}
\end{figure}

\begin{figure}[htbp]
  \centering
  \subcaptionbox{Solving time in function of the number of elements}{%
    \includegraphics[width=0.5\textwidth]{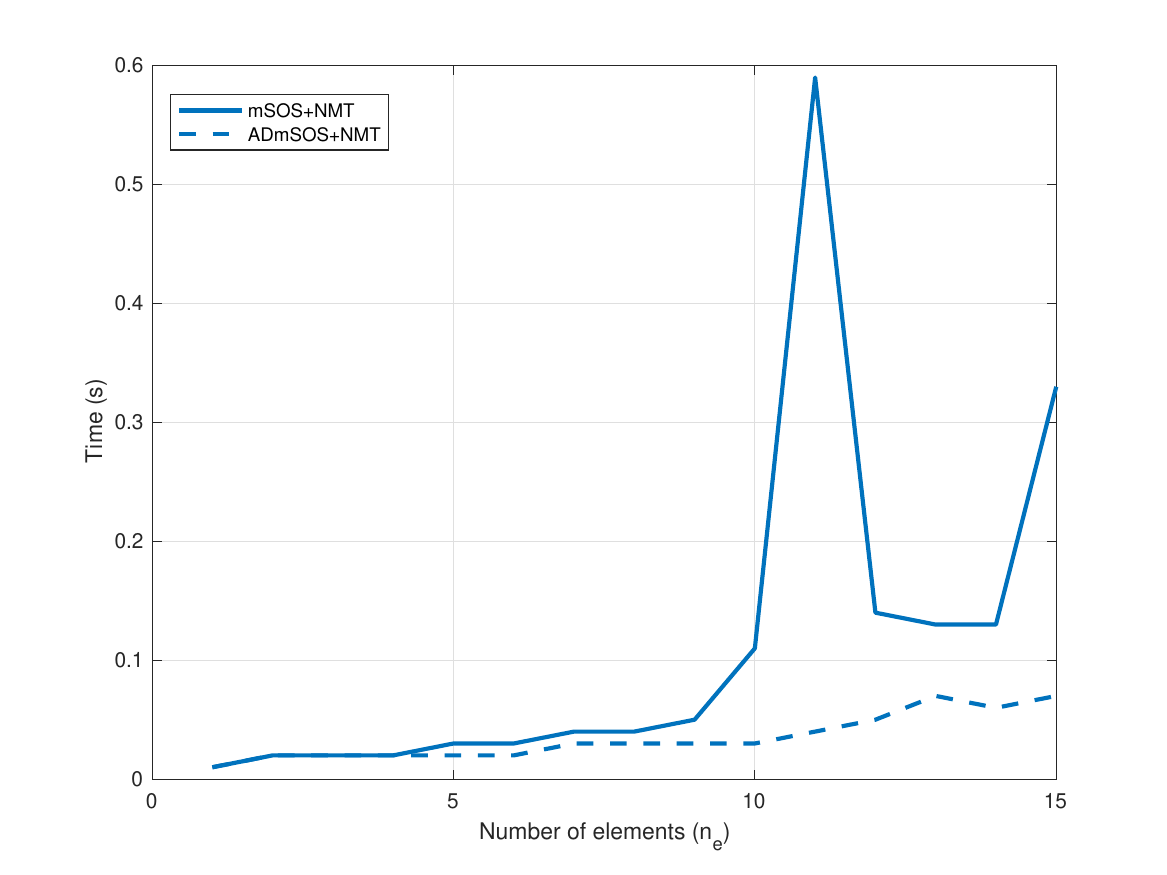}%
  }
 \hspace*{-0.5cm} \hfill
  \subcaptionbox{Total number of relaxation variables and size of the largest SDP matrix}{%
    \includegraphics[width=0.5\textwidth]{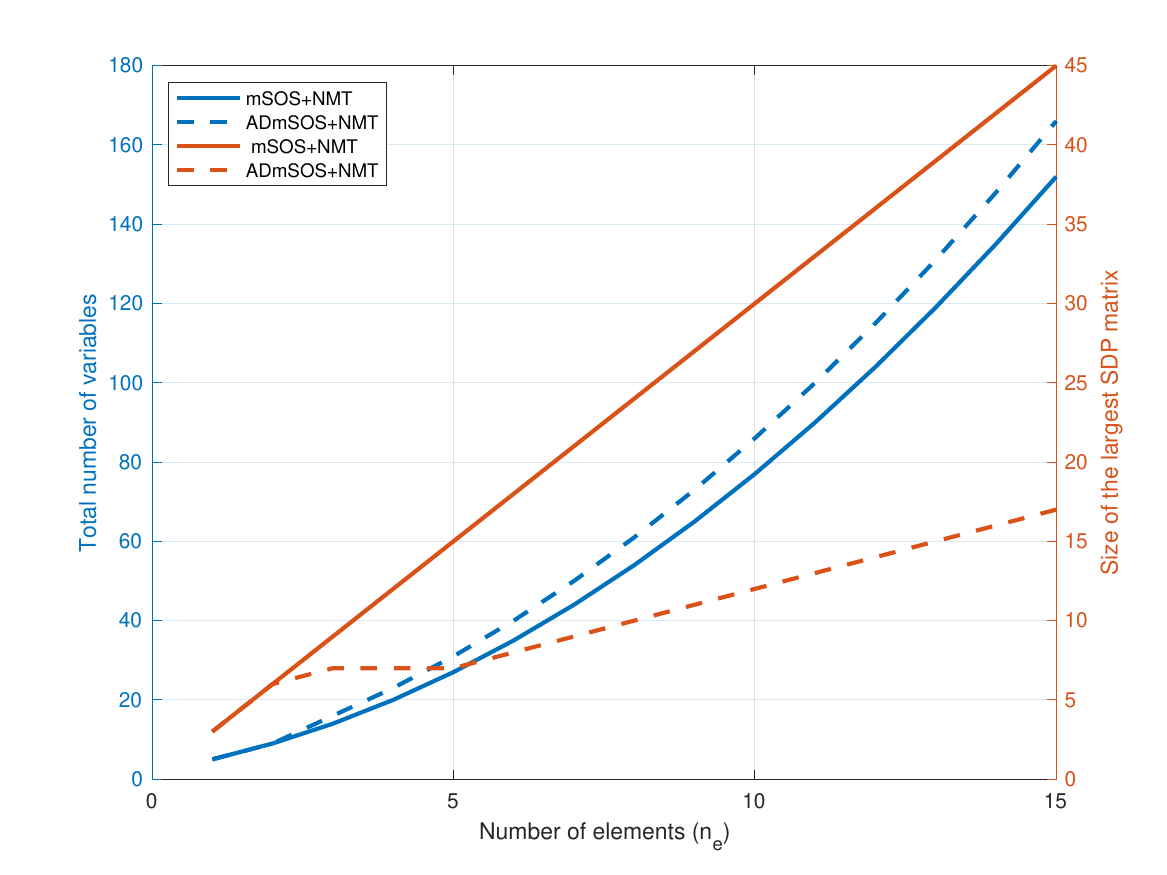}%
  }
  \caption{ Comparison of mSOS+NMT and ADmSOS+NMT for the first-order relaxation of a double-hinged beam problem  }
  \label{fig:NMT-rel1}
\end{figure}

\begin{figure}[htbp]
  \centering
  \subcaptionbox{Solving time in function of the number of elements}{%
    \includegraphics[width=0.5\textwidth]{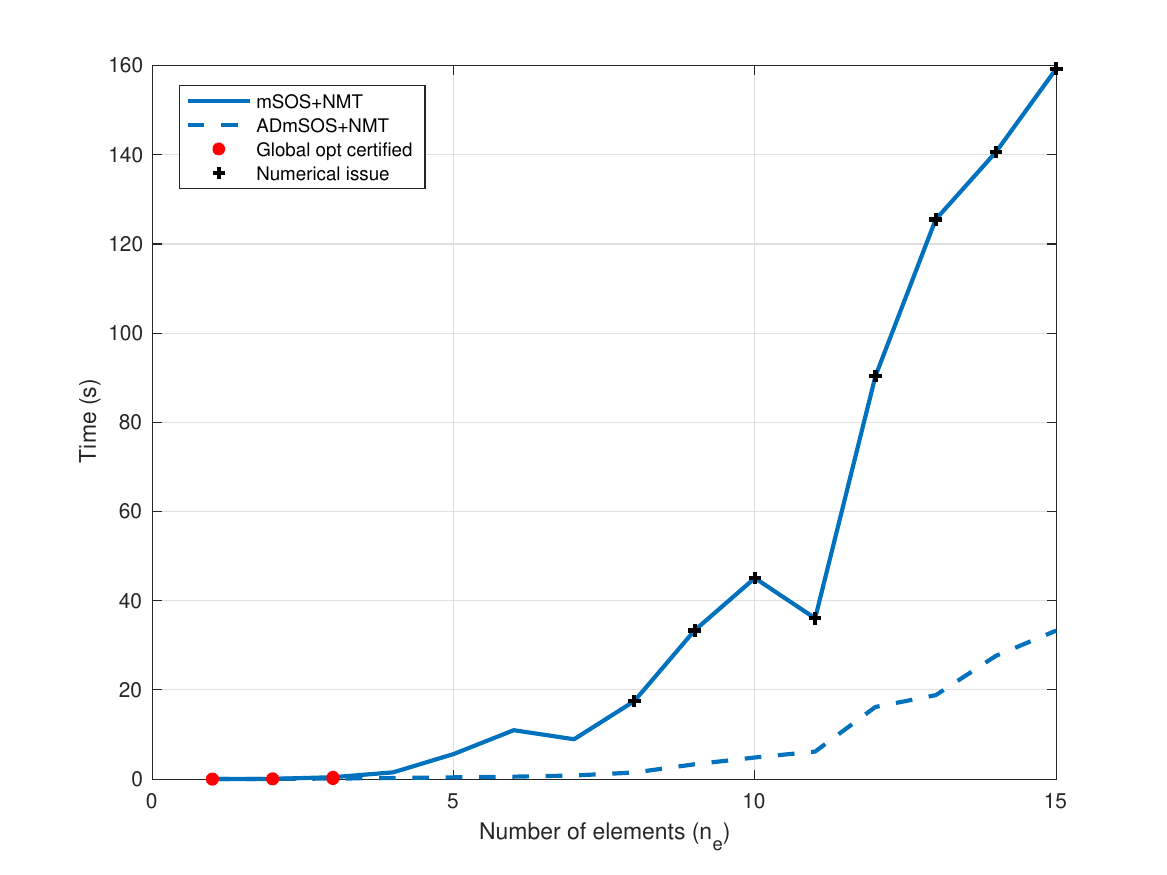}%
   
  }
 \hspace*{-0.5cm} \hfill
  \subcaptionbox{Total number of relaxation variables and size of the largest SDP matrix}{%
    \includegraphics[width=0.5\textwidth]{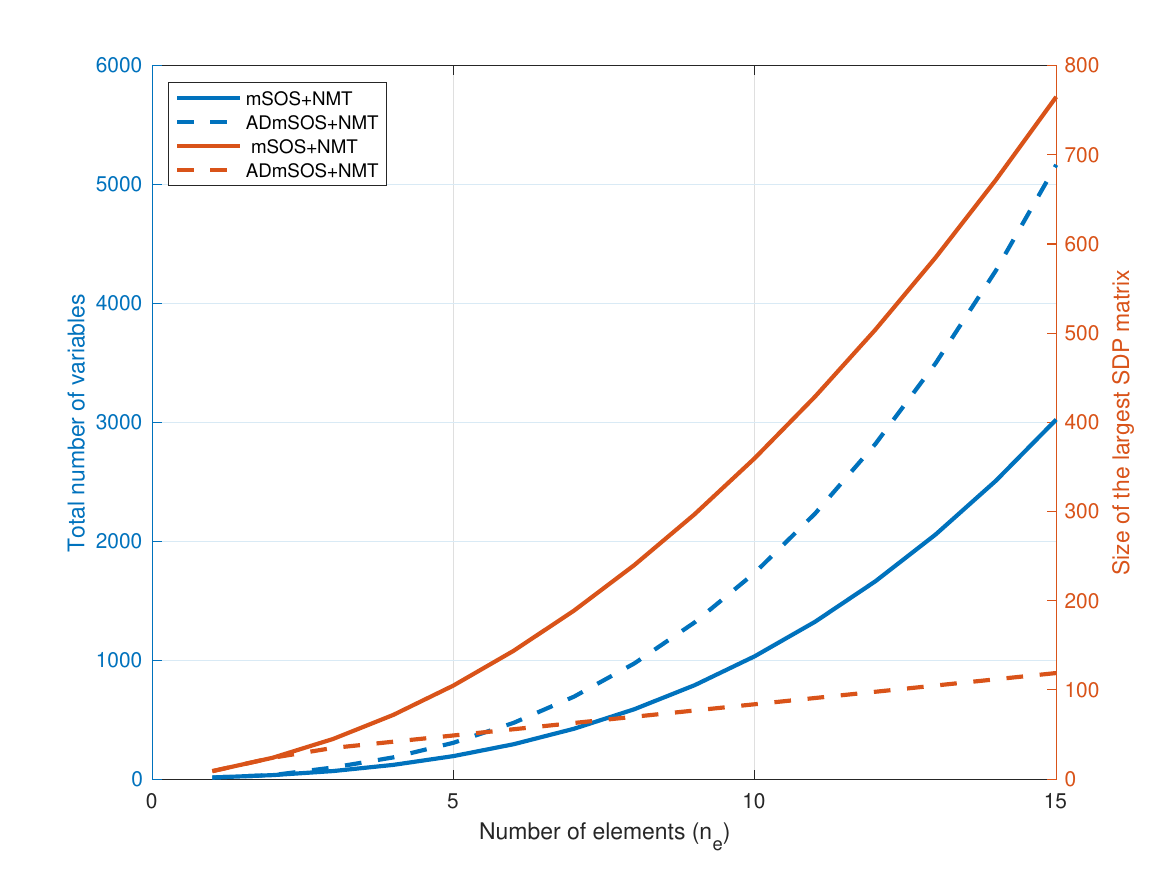}%
    
  }
  \caption{ Comparison of mSOS+NMT and ADmSOS+NMT for the second-order relaxation of a double-hinged beam problem  }
  \label{fig:NMT-rel2}
\end{figure}

\begin{figure}[htbp]
  \centering
  \subcaptionbox{Solving time in function of the number of elements}{%
    \includegraphics[width=0.5\textwidth]{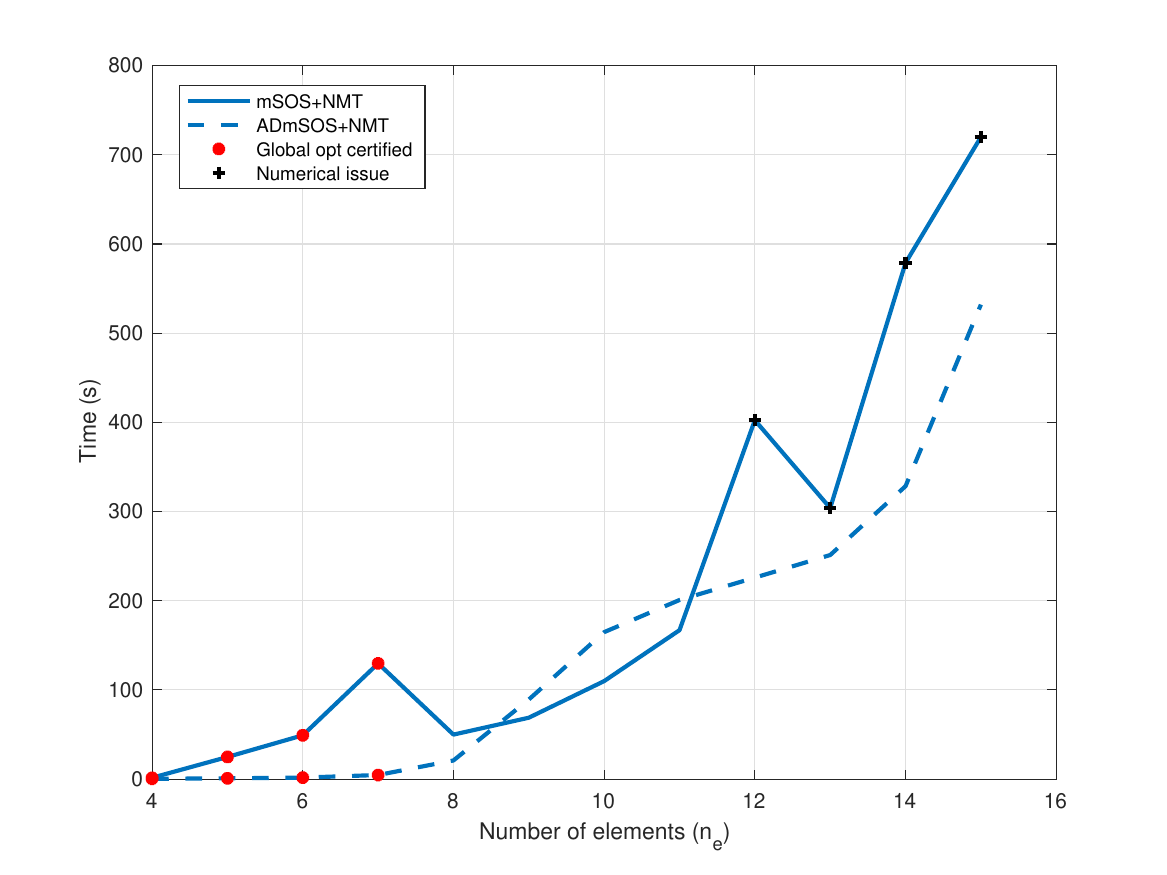}%
 
  }
 \hspace*{-0.5cm} \hfill
  \subcaptionbox{Total number of relaxation variables and size of the largest SDP matrix}{%
    \includegraphics[width=0.5\textwidth]{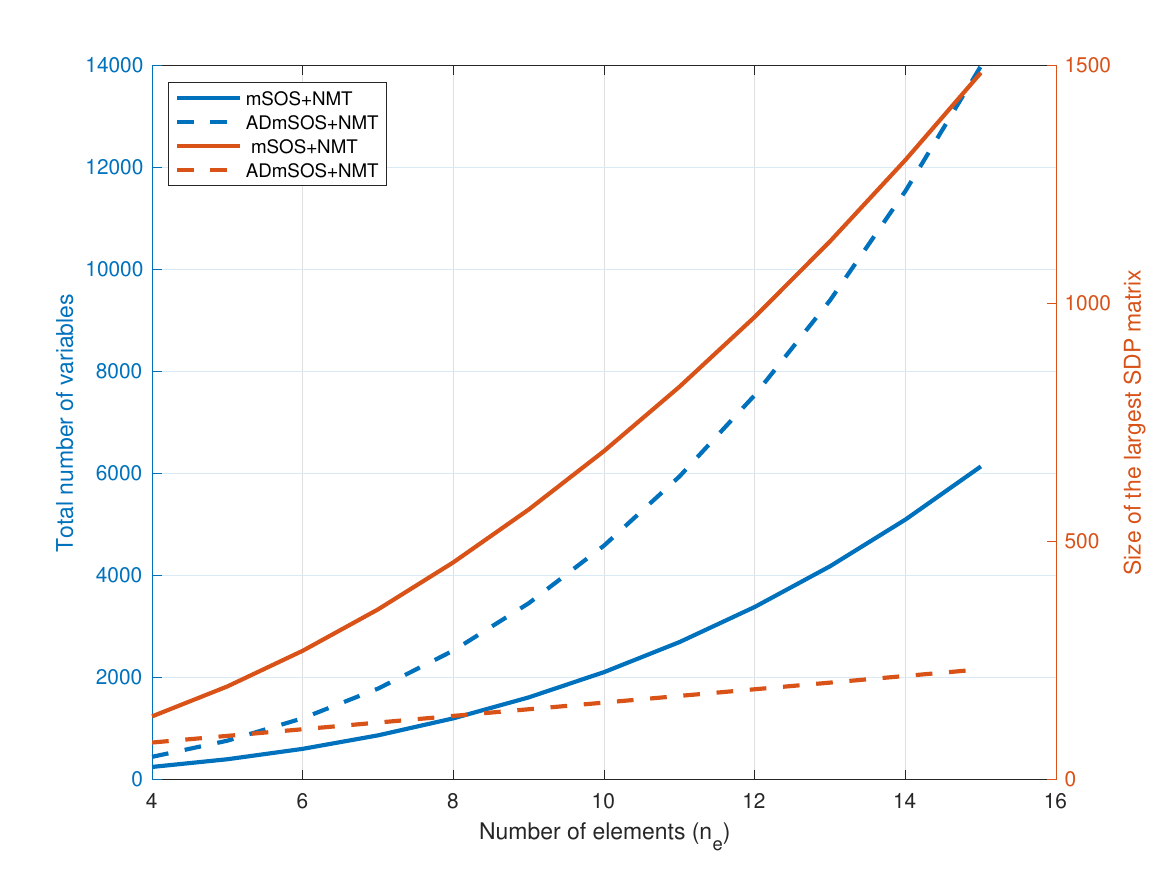}%

  }
  \caption{ Comparison of mSOS+NMT and ADmSOS+NMT for the third-order relaxation of a double-hinged beam problem  }
 \label{fig:NMT-rel3}
\end{figure}

\begin{figure}[htbp]
  \centering
  \subcaptionbox{Solving time in function of the number of elements}{%
    \includegraphics[width=0.5\textwidth]{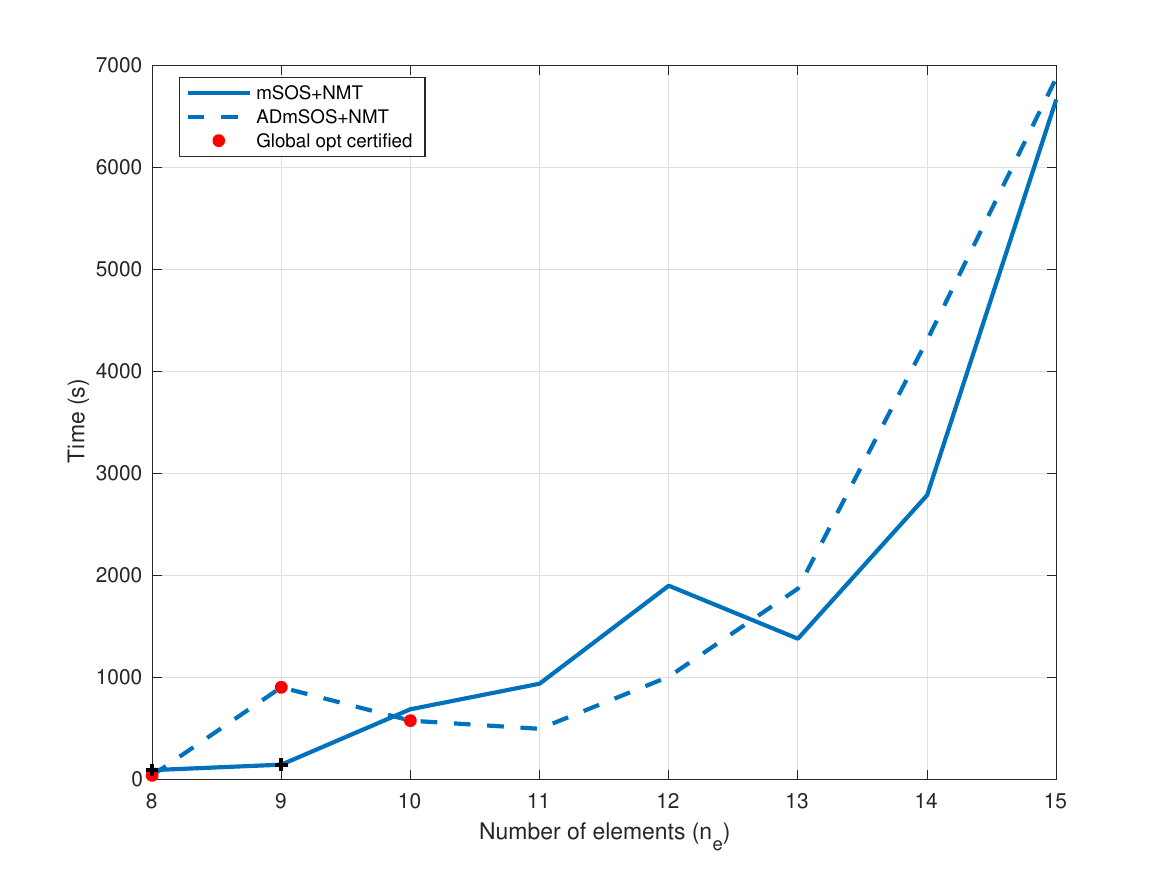}%
  }
 \hspace*{-0.5cm} \hfill
  \subcaptionbox{Total number of relaxation variables and size of the largest SDP matrix}{%
    \includegraphics[width=0.5\textwidth]{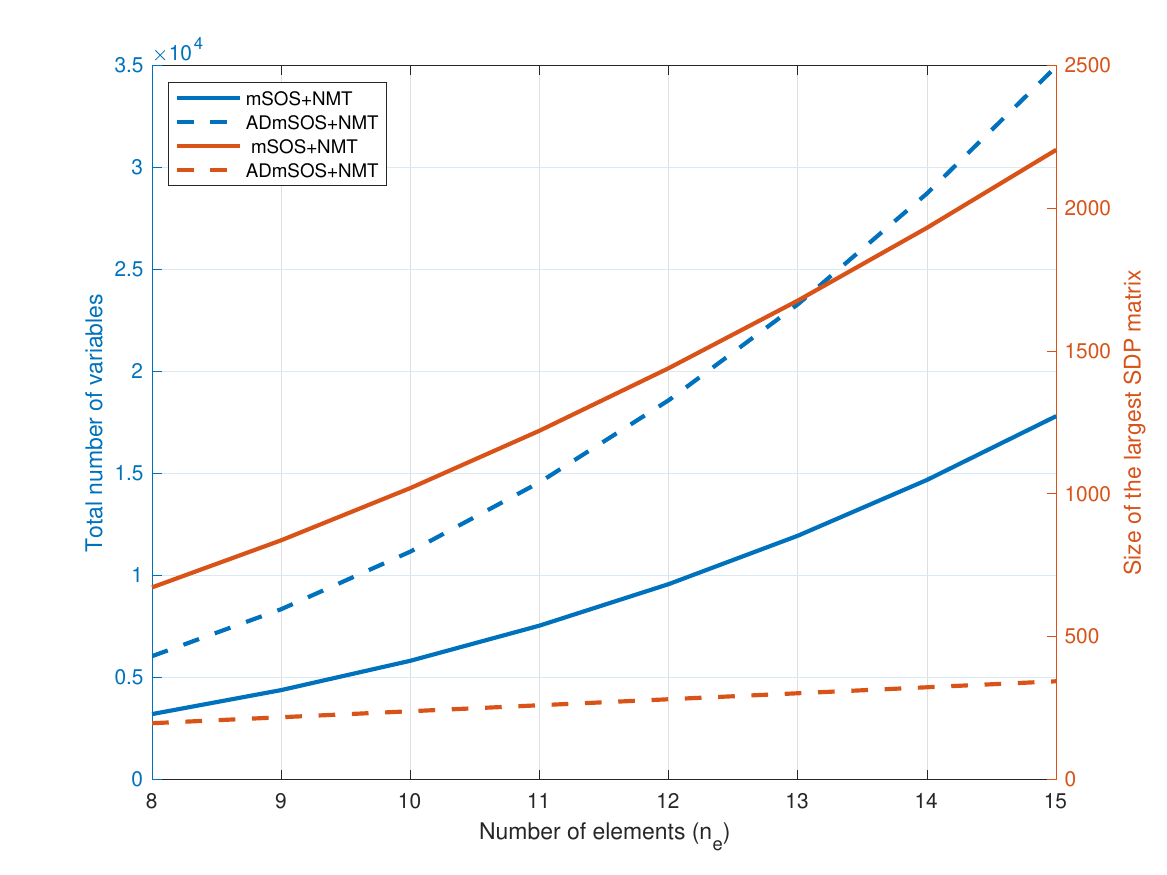}%
   
  }
  \caption{ Comparison of mSOS+NMT and ADmSOS+NMT for the fourth-order relaxation of a double-hinged beam problem  }
  \label{fig:NMT-rel4}
\end{figure}

The numerical results in Figures \ref{fig:mSOS-rel1}--\ref{fig:NMT-rel4}, together with detailed data in Tables \ref{tab:sw1}--\ref{tab:sw4} in Appendix \ref{appendix-tables}, demonstrate the computational advantages offered by arrow decomposition across different problem sizes and relaxation orders. For the first-order relaxation ($r=1$), all methods perform comparably on small problems with $n_\mathrm{e} \leq 9$, requiring less than $0.1$ seconds of computation time. However, even at this lowest relaxation order, both AD variants (mSOS+AD and mSOS+NMT+AD) maintain consistent performance as $n_\mathrm{e}$ increases to $15$ (the running time is less than $0.1$ second), while the standard mSOS approach exhibits occasional performance spikes due to a higher number of iterations (up to $0.6$ seconds running time).

The benefits of arrow decomposition become particularly evident at higher relaxation orders. For the second-order relaxation ($r=2$), the standard mSOS approach requires $397$ seconds for $n_\mathrm{e} = 15$, while mSOS+NMT+AD solves the same problem in just $33$ seconds. Furthermore, both non-AD variants (mSOS and mSOS+NMT) exhibit numerical stability issues for larger problems, as indicated by the $^+$ superscript in Figures \ref{fig:mSOS-rel3} and \ref{fig:NMT-rel3}, and Table~\ref{tab:sw2}. These issues manifest as solver convergence problems in Mosek. In contrast, arrow-decomposition-based methods maintain numerical stability and provide reliable results across all problem sizes, offering up to $12\times$ speedup for $n_\mathrm{e} = 15$.

The computational advantages of arrow decomposition are most pronounced in the third-order relaxation ($r=3$). While the AD combined with the standard mSOS approach becomes computationally intractable beyond $n_\mathrm{e} = 10$ requiring over $2.5$ hours for $n_\mathrm{e} = 8$  (Figure \ref{fig:mSOS-rel3}), the AD combined to mSOS with nonmixed-term basis (mSOS+NMT+AD) solves the same problem in just $21$ seconds, offering a $122\times$ speedup (Figure \ref{fig:NMT-rel3}). The NMT-based approaches continue to solve problems efficiently up to $n_\mathrm{e} = 15$.

However, arrow decomposition alone can be counterproductive at higher relaxation orders. This is evident for $n_\mathrm{e} = 9$ where mSOS+AD requires $9275$ seconds compared to $3816$ seconds for standard mSOS. Such degradation occurs because the number of variables for mSOS+AD is $25695$, whereas for mSOS we have significantly less, $8007$.

At the fourth relaxation order ($r=4$), only the NMT-based approaches remain practically viable for problems with $n_\mathrm{e} \ge 8$. For instance, solving $n_\mathrm{e} = 8$ requires $55.8$ hours with standard mSOS but only $36$ seconds with mSOS+NMT+AD, representing a $5500\times$ speedup. 

 For $r=3$, global optimality for small problems ($n_\mathrm{e} \leq 3$) is verified, while AD-based approaches maintain global optimality certificates up to $n_\mathrm{e} \leq 9$. At $r=4$, AD-based methods achieve tight optimality gaps ($\varepsilon_\mathrm{r} \le 10^{-5}$) up to $n_\mathrm{e} = 10$, while still providing reliable bounds ($\varepsilon_\mathrm{r} \le 10^{-3}$) for larger problems. This improved numerical stability is evidenced by the absence of solver convergence issues (marked by $^+$) in AD-based approaches.

\subsection{$24$-element modular frame}

Having demonstrated the benefits of arrow decomposition on a simple structural component with varying discretization, we now turn to a more realistic engineering application of weight minimization of a modular frame structure, based on an example originally introduced in \cite[Section 4.1]{tyburec2022global}. The structure, shown in Fig.~\ref{fig:frame24}, features $14$ nodes and $24$ elements and is subjected to combined horizontal wind loads (magnitudes $1$, $1$, and $0.5$) and vertical unit loads at beam midSpans. For simplicity, we use normalized material properties (Young's modulus $E = 1$, density $\rho = 1$) and fixed base supports.

To create a practical design that accounts for bidirectional wind loading, we impose symmetry conditions on the structure. This symmetry, combined with structural requirements, leads to several groups of elements that must share the same cross-sectional areas: The columns satisfy $x_1 = x_4$, $x_2 = x_5$, and $x_3 = x_6$, while the horizontal beams must maintain $x_7 = x_8$, $x_9 = x_{10}$, and $x_{11} = x_{12}$. For diagonals within each story, we enforce $x_{13} = x_{14} = x_{15} = x_{16}$, $x_{17} = x_{18} = x_{19} = x_{20}$, and $x_{21} = x_{22} = x_{23} = x_{24}$. These dependencies reduce the number of independent design variables to $9$.

As shown in Fig.~\ref{fig:frame24}b, we employ different cross-sectional shapes optimized for each element's primary function: H-sections for columns to resist axial loads and bending, I-sections for horizontal beams to maximize bending resistance, and thin-walled circular sections for diagonal braces primarily carrying axial forces. This choice of cross-sections results in a quadratic polynomial stiffness matrix $\bm{K}(\bm{x})$ in the cross-sectional areas $\bm{x}$.

For the arrow decomposition, we adopt a natural storey-wise partition of elements that reflects the physical structure: The first subdomain $\mathcal{D}_1 = \{1,4,7,8,13,14,15,16\}$ contains all elements of the bottom story, $\mathcal{D}_2 = \{2,5,9,10,17,18,19,20\}$ encompasses the middle story elements, and $\mathcal{D}_3 = \{3,6,11,12,21,22,23,24\}$ includes all elements of the top story. We set the compliance bound to $\overline{c} = 5,000$.

\begin{figure}[!htbp]
\begin{subfigure}{0.425\linewidth}
    \begin{tikzpicture}
		\scaling{2.2}
		\point{a}{0.00}{0.0};
		\point{b}{1.50}{0.0}; 
		\point{c}{0.00}{1.0}; 
		\point{d}{1.50}{1.0}; 
		\point{e}{0.00}{2.0}; 
		\point{f}{1.50}{2.0}; 
		\point{g}{0.00}{3.0}; 
		\point{h}{1.50}{3.0}; 
		\point{i}{0.75}{1.0};
		\point{j}{0.75}{2.0}; 
		\point{k}{0.75}{3.0}; 
		\point{l}{0.75}{0.5}; 
		\point{m}{0.75}{1.5}; 
		\point{n}{0.75}{2.5}; 
		\point{o}{0.75}{0.0};
		
		\beam{2}{a}{c}; \notation{4}{a}{c}[$1$];
		\beam{2}{b}{d}; \notation{4}{b}{d}[$4$];
		\beam{2}{c}{e}; \notation{4}{c}{e}[$2$];
		\beam{2}{d}{f}; \notation{4}{d}{f}[$5$];
		\beam{2}{e}{g}; \notation{4}{e}{g}[$3$];
		\beam{2}{f}{h}; \notation{4}{f}{h}[$6$];
		\beam{2}{c}{i}; \notation{4}{c}{i}[$7$];
		\beam{2}{i}{d}; \notation{4}{i}{d}[$8$];
		\beam{2}{e}{j}; \notation{4}{e}{j}[$9$];
		\beam{2}{j}{f}; \notation{4}{j}{f}[$10$];
		\beam{2}{g}{k}; \notation{4}{g}{k}[$11$];
		\beam{2}{k}{h}; \notation{4}{k}{h}[$12$];
		\beam{2}{a}{l}; \notation{4}{a}{l}[$13$];
		\beam{2}{l}{d}; \notation{4}{l}{d}[$14$];
		\beam{2}{b}{l}; \notation{4}{b}{l}[$15$];
		\beam{2}{l}{c}; \notation{4}{l}{c}[$16$];
		\beam{2}{c}{m}; \notation{4}{c}{m}[$17$];
		\beam{2}{m}{f}; \notation{4}{m}{f}[$18$];
		\beam{2}{d}{m}; \notation{4}{d}{m}[$19$];
		\beam{2}{m}{e}; \notation{4}{m}{e}[$20$];
		\beam{2}{e}{n}; \notation{4}{e}{n}[$21$];
		\beam{2}{n}{h}; \notation{4}{n}{h}[$22$];
		\beam{2}{f}{n}; \notation{4}{f}{n}[$23$];
		\beam{2}{n}{g}; \notation{4}{n}{g}[$24$];

        \support{3}{a};
		\support{3}{b};
		
		\load{1}{i}[90][0.7][0.0]; \notation{1}{i}{$1$}[above=3mm,xshift=1.5mm];
		\load{1}{j}[90][0.7][0.0]; \notation{1}{j}{$1$}[above=3mm,xshift=1.5mm];
		\load{1}{k}[90][0.7][0.0]; \notation{1}{k}{$1$}[above=3mm,xshift=1.5mm];
		\load{1}{c}[180][0.7][0.0]; \notation{1}{c}{$1$}[below=0mm, xshift=-3mm];
		\load{1}{e}[180][0.7][0.0]; \notation{1}{e}{$1$}[below=0mm, xshift=-3mm];
		\load{1}{g}[180][0.35][0.0]; \notation{1}{g}{$0.5$}[below=0mm, xshift=-3mm];
		
		\dimensioning{2}{o}{l}{4.0}[$0.5$];
		\dimensioning{2}{l}{i}{4.0}[$0.5$];
		\dimensioning{2}{i}{m}{4.0}[$0.5$];
		\dimensioning{2}{m}{j}{4.0}[$0.5$];
		\dimensioning{2}{j}{n}{4.0}[$0.5$];
		\dimensioning{2}{n}{k}{4.0}[$0.5$];
		\dimensioning{1}{a}{o}{-0.75}[$0.75$];
		\dimensioning{1}{o}{b}{-0.75}[$0.75$];
	\end{tikzpicture}
	\caption{}
\end{subfigure}%
\hfill\begin{subfigure}{0.2\linewidth}
    \centering
	\squared{1}--\squared{6}:\\
	\begin{tikzpicture}
		\scaling{1.0}
		\point{a}{0.0}{0.0};\point{a0}{1.0}{0.0};
		\point{b}{0.0}{0.1};\point{b0}{1.0}{0.1};
		\point{c}{0.45}{0.1};\point{c0}{0.55}{0.1};
		\point{d}{0.45}{0.9};\point{d0}{0.55}{0.9};
		\point{e}{0.0}{0.9};\point{e0}{1.0}{0.9};
		\point{f}{0.0}{1.0};\point{f0}{1.0}{1.0};
		\draw[black, thick, fill=black!25] (a0) -- (b0) -- (c0) -- (d0) -- (e0) -- (f0) -- (f) -- (e) -- (d) -- (c) -- (b) -- (a) -- cycle;
		\dimensioning{2}{a}{f}{1.25}[$10t_e$];
		\dimensioning{1}{a}{a0}{1.25}[$10t_e$];
		\draw [-stealth](0.25,0.60) -- (0.5,0.75);
		\draw [-stealth](0.25,0.60) -- (0.35,0.95);
		\draw [-stealth](0.25,0.60) -- (0.35,0.05);
		\draw (-0.2,0.6) -- node[above=-0.5mm]{$t_e$}(0.025,0.6) -- (0.25,0.60);
	\end{tikzpicture}\\
        \vspace{3mm}
	\squared{7}--\squared{12}:\\
	\begin{tikzpicture}
		\scaling{1.0}
		\point{a}{0.0}{0.0};\point{a0}{1.0}{0.0};
		\point{b}{0.0}{0.1};\point{b0}{1.0}{0.1};
		\point{c}{0.45}{0.1};\point{c0}{0.55}{0.1};
		\point{d}{0.45}{1.9};\point{d0}{0.55}{1.9};
		\point{e}{0.0}{1.9};\point{e0}{1.0}{1.9};
		\point{f}{0.0}{2.0};\point{f0}{1.0}{2.0};
		\draw[black, thick, fill=black!25] (a0) -- (b0) -- (c0) -- (d0) -- (e0) -- (f0) -- (f) -- (e) -- (d) -- (c) -- (b) -- (a) -- cycle;
		\dimensioning{2}{a}{f}{1.25}[$20t_e$];
		\dimensioning{1}{a}{a0}{2.25}[$10t_e$];
		\draw [-stealth](0.25,0.60) -- (0.5,0.75);
		\draw [-stealth](0.25,0.60) -- (0.35,1.95);
		\draw [-stealth](0.25,0.60) -- (0.35,0.05);
		\draw (-0.2,0.6) -- node[above=-0.5mm]{$t_e$}(0.025,0.6) -- (0.25,0.60);
	\end{tikzpicture}\\
	\vspace{3mm}
	\squared{13}--\squared{24}:\\
	\begin{tikzpicture}
		\point{a}{-1.0}{0};\point{a0}{0.0}{0.0};
		\draw[thick, fill=black!25] (0.0,0) arc (0:360:0.5);
		\draw[thick, fill=white] (-0.1,0) arc (0:360:0.4);
		\dimensioning{1}{a}{a0}{0.65}[$10t_e$];
		\draw [-stealth](-0.3,-0.1) -- (-0.05,0.05);
		\draw (-0.7,-0.1) -- node[above=-0.5mm]{$t_e$}(-0.4,-0.1) -- (-0.3,-0.1);
	\end{tikzpicture}
	\caption{}
\end{subfigure}%
\hfill\begin{subfigure}{0.35\linewidth}
        \includegraphics[width=0.8\linewidth]{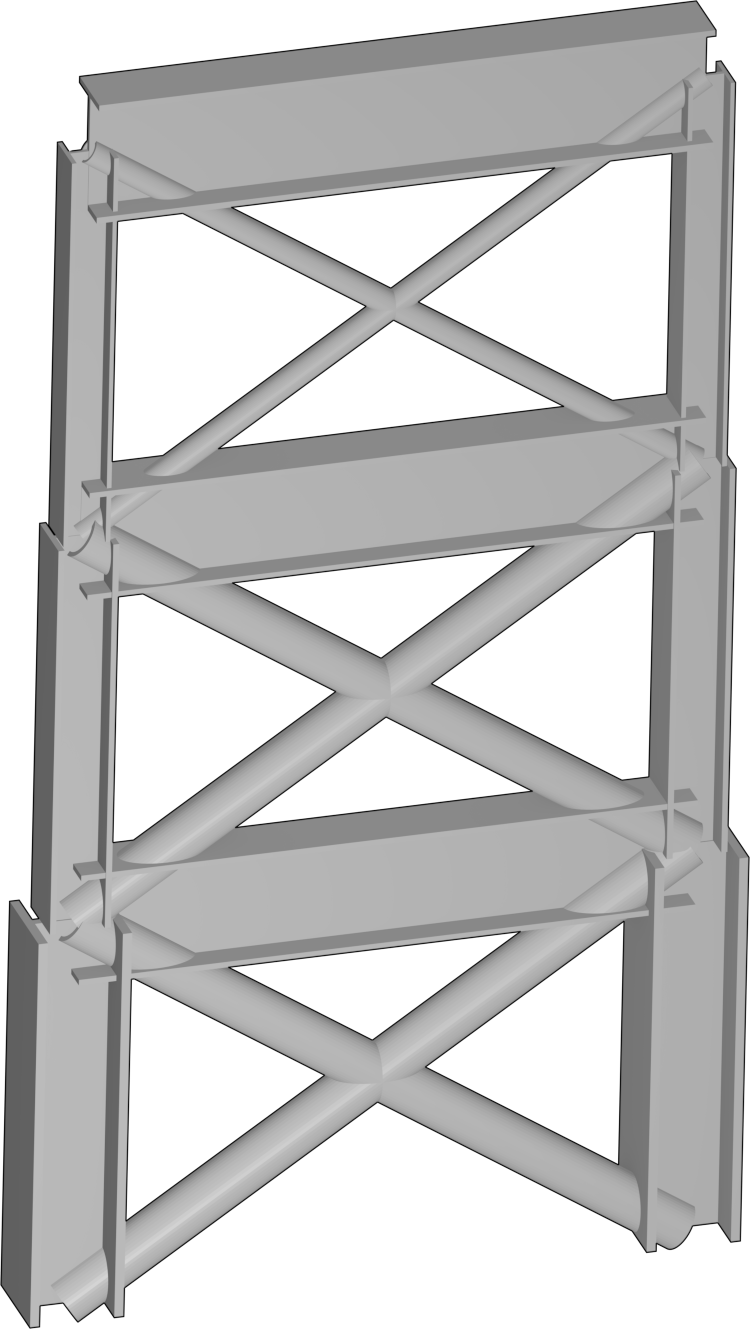}
        \vspace{3.5mm}
        \caption{}
\end{subfigure}
    \caption{$24$-elements modular frame: (a) discretization and boundary conditions, (b) cross-section parametrizations, and (c) optimal design.}
    \label{fig:frame24}
\end{figure}
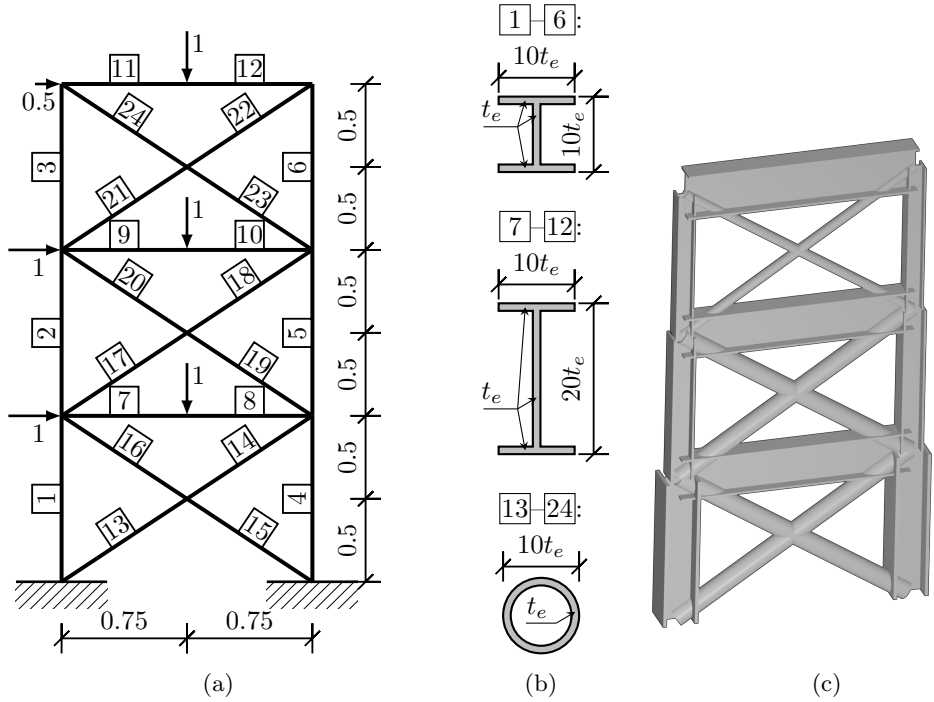

The results in Table \ref{tab:frame24} compare solution times across different relaxation orders $r$ and the four solution approaches. All methods achieve consistent lower bounds, but their computational efficiency varies significantly. At the first relaxation order ($r=1$), all approaches perform similarly, requiring only $0.04-0.07$ seconds. However, the benefits of decomposition become apparent at higher orders. For $r=2$, both standard mSOS and NMT mSOS combined with AD reduce solution time from $21$ seconds to under $9$ seconds. The most dramatic differences appear at $r=3$, where standard mSOS requires nearly $28$ minutes while mSOS+NMT+AD completes in just $32$ seconds--a $53\times$ speedup.

Interestingly, for $r=3$, mSOS+AD requires more time than standard mSOS due to additional variables introduced by the decomposition. However, AD consistently provides better numerical stability and tighter optimality guarantees, as evidenced by the relative optimality gaps $\varepsilon_\mathrm{r}$.

The hierarchy converges at $r=3$, where we can verify global optimality through different means: the rank flatness condition for mSOS and mSOS+AD, and the relative optimality gap $\varepsilon_\mathrm{r}$ for NMT-based approaches. The small optimality gaps ($\varepsilon_\mathrm{r} \le 1.8\times10^{-6}$) for NMT methods with AD confirm the advantages of combining these techniques. Fig. \ref{fig:frame24}c then shows the resulting optimal design.

\begin{table}[htbp]
    \centering
    \scriptsize
    \setlength{\tabcolsep}{1pt}
    \renewcommand{\arraystretch}{1.0}
    \resizebox{\textwidth}{!}{%
    \begin{tabular}{c|c|c|c|c|c}
     \toprule
         $r$ & & mSOS & mSOS+AD & mSOS+NMT & mSOS+NMT+AD\\
        \hline
         $1$ & $t$ [s]& $0.07$ & $0.04$ & $0.07$ & $0.04$\\
             & l.b. & $5.15\times10^{-2}$ & $5.15\times10^{-2}$ & $5.15\times10^{-2}$ & $5.15\times10^{-2}$\\
             & $\varepsilon_\mathrm{r}$ & $1.3$ & $1.3$ & $1.3$ & $1.3$\\
             & $n$ & $54$ & $62$ & $54$ & $62$ \\
             & size & $10\times1, 1\times10, 1\times37$ & $10\times1, 1\times10, 1\times13, 2\times16$ & $10\times1, 1\times10, 1\times37$ & $10\times1, 1\times10, 1\times13, 2\times16$\\
        \hline
         $2$ & $t$ [s] & $21.82$ & $8.83$ & $20.18$ & $8.40$\\
             & l.b. & $1.09\times10^{-2}$ & $1.09\times10^{-2}$ & $1.09\times10^{-2}$ & $1.09\times10^{-2}$\\
             & $\varepsilon_\mathrm{r}$ & $8.2\times10^{-2}$ & $8.2\times10^{-2}$ & $8.3\times10^{-2}$ & $8.3\times10^{-2}$\\
             & $n$ & $714$ & $1424$ & $588$ & $1298$\\
             & size & $10\times10, 1\times55, 1\times370$ & $10\times10, 1\times55, 1\times130, 2\times160$ & $10\times10, 1\times19, 1\times370$ & $10\times10, 1\times19, 1\times130, 2\times160$\\
         \hline
         $3$ & $t$ [s] & $1679.34$ & $1913.92$ & $46.27$ & $31.69$\\
             & l.b. & $1.18\times10^{-2}$ & $1.18\times10^{-2}$ & $1.18\times10^{-2}$ & $1.18\times10^{-2}$\\
             & $\varepsilon_\mathrm{r}$ & $1.5\times10^{-3}$ & $3.3\times10^{-4}$ & $1.8\times10^{-6}$ & $3.0\times10^{-7}$\\
             & $n$ & $5004$ & $26234$ & $1194$ & $3740$\\
             & size & $10\times55, 1\times220, 1\times2035$ & $10\times55, 1\times220, 1\times715, 2\times880$ & $10\times19, 1\times28, 1\times703$ & $10\times19, 1\times28, 1\times247, 2\times304$\\
             \bottomrule  
    \end{tabular}}
    \caption{Comparison of four moment-SOS hierarchies for a $24$-element frame optimization problem at different relaxation orders $r$. The hierarchies are: standard moment-SOS (mSOS), mSOS with arrow decomposition (mSOS+AD), mSOS using nonmixed-term basis (mSOS+NMT), and mSOS combining nonmixed-term basis with arrow decomposition (mSOS+NMT+AD). For each method and relaxation order, we report: solution time ($t$) in seconds, computed lower bound (l.b.), relative global optimality gap ($\varepsilon_\mathrm{r}$), number of variables in the relaxation ($n$), and sizes of matrix constraints (where $a\times b$ indicates $a$ matrix inequalities of the size $\mathbb{S}^b$).}
    \label{tab:frame24}
\end{table}

\section{Conclusion and perspectives}
\label{sec:conclusion}

This paper advanced the Arrow Decomposition (AD) method beyond the framework established in \cite{kovcvara2021decomposition} by developing weaker applicability conditions and extending its theoretical foundations to handle polynomial matrix inequalities. Our main contribution was showing how AD could be effectively combined with moment-SOS relaxations in ways that preserve computational advantages while ensuring theoretical convergence.

To these goals, we established several theoretical improvements to the AD framework. First, we weakened the conditions for AD applicability to handle problems where the LMI/PMI need only be positive semidefinite instead of positive definite. We then showed that AD potentially produces degenerate problems without interior points and addressed this issue by developing a projection-based preprocessing method. This not only resolved the Slater condition violation but also reduced computational complexity by decreasing both the number of additional variables and matrix inequality sizes.

In extending AD to PMIs, we provided two approaches: direct application to POP problems (``prior to mSOS'') and application at the relaxation level (``posterior to mSOS''). Notably, we proved that the posterior approach can achieve the benefits of prior application without the computational overhead of additional variables in the monomial basis.

Our numerical experiments focused on compliance and weight minimization problems in structural engineering. These examples provided physical interpretations for the additional variables introduced by AD. The results showed that mSOS+AD and its variants significantly improve both solution times and numerical accuracy, particularly for problems combining high dimensionality with low relaxation orders.

This work opens several promising directions for future research. One possibility is to extend the results of Section \ref{computation-consideration-msos} to cases where the range and null space of the matrices $\A_k(\x)$ may depend on $\x$. Another key area is the development of automated methods for generating optimal decompositions that balance matrix sizes against additional variable counts. Further opportunities lie in extending applications to other optimization problems with arrow-type LMI/PMIs, particularly in continuum topology optimization across various physical domains. The integration of AD with the correlative sparsity method presents another avenue for investigation, as does the study of potential combinations with the finest block diagonal decomposition technique presented in \cite{Murota2010}. These extensions could further enhance the practical applicability and computational efficiency of the AD method.

\appendix

\section{Appendix}
\label{appendix}

\subsection{Technical lemmas}
\label{appendix-lemma}
Here, we present few technical lemmas that are focused more on the sparsity in the arrow decomposition technique presented in this paper. In order to do so, let $\bm{G}$ be a matrix with an arrow shape (recall Section \ref{sec: ADLMI}) and let $ \mathcal{I} $ be the index set of the matrix $\bm{A}$, and $(\mathcal{I}_k)_{k \in [p]}$ its partition.

The first result shows the existence and the construction of a matrix $\bm{\Pi}_k$ for each $k \in [p]$, such that the equality $ \bm{D}_k=\sum_{\substack{k:k < \ell \\ \mathcal{I}_{k,\ell}\neq \varnothing}} (\bm{D}_{k,\ell})_{ij}-\sum_{\substack{\ell: \ell <k \\ \mathcal{I}_{\ell, k}\neq \varnothing}}(\bm{D}_{\ell,k})_{ij},$ can be written in the compact form $\bm{D}_k=\bm{\Pi}_k\bm{D}$.
\begin{lemma}
Let $ n_{\mathcal{I}}=\sum_{k=1}^p \sum_{\ell >k} 
        \vert \mathcal{I}_{k,\ell}\vert$ and $\bm{D} \in \mathbb{R}^{n_{\mathcal{I} }  \times m}$ be a matrix containing all the variables involved in each matrix $\bm{D}_{k,\ell}$, i.e.,
$\bm{D}=\lbrace (\bm{D}_{k,\ell})_{ij} \neq 0: k<\ell, (i,j) \in \mathcal{I}_{k,\ell} \times [m] \rbrace $. Then there exists a matrix $\bm{\Pi}_k \in \mathbb{R}^{ \vert \mathcal{I}_k \vert  \times n_{\mathcal{I} } }$ such that Equation \eqref{additional-variables-dependencies} can be written in the compact form 
\begin{equation}\label{definition-Sigma}
    \bm{D}_k=\bm{\Pi}_k\bm{D}.
\end{equation}  
 \label{compacte-form-add-variables}
\end{lemma}

\begin{proof}
 Let $\overline{\bm{D}}_{k,\ell}=\lbrace (\bm{D}_{k,\ell})_{ij} \neq 0 : (i,j) \in \mathcal{I}_{k,\ell} \times [m] \rbrace $. By choosing the order $<$ defined by $(a,b) < (k,\ell) \Leftrightarrow a \leq k \text{ and } b < \ell$, we can write the matrix $\bm{D}$ as a block-wise concatenation of the matrices $\overline{\bm{D}}_{k,\ell}$, i.e., $\bm{D}=\begin{bmatrix}
        \dots & \overline{\bm{D}}_{k,\ell} &\dots 
        \end{bmatrix}^T$. Now, we define the matrix $\bm{\Pi}_k \in \mathbb{R}^{ \vert \mathcal{I}_k \vert  \times n_{\mathcal{I} } }$ by 
\begin{equation}
   (\bm{\Pi}_k)_{ij}=\left\lbrace\begin{aligned}
  & 1 \text{ if } \exists \ell >k: \mathcal{I}_{k,\ell} \neq \varnothing, i \in \mathcal{P}_{k,\ell} \text{ and } j \in \mathcal{S}_{k,\ell}  ,\\
  -& 1 \text{ if } \exists \ell <k: \mathcal{I}_{\ell,k} \neq \varnothing, i \in \mathcal{P}_{k,\ell} \text{ and } j \in  \mathcal{S}_{\ell,k}  ,\\
  & 0 \text{ otherwise },
\end{aligned}\right.
\label{matrix-Pi}
\end{equation}
where the set $\mathcal{S}_{k,\ell} $ enumerates the row positions of the block $\overline{\bm{D}}_{k,\ell}$ in $\bm{D}$, and is defined by $\mathcal{S}_{k,\ell}= \lbrace s_{k,\ell}+1,\cdots,s_{k,\ell}+ \vert \mathcal{\overline{I}}_{k,\ell}\vert \rbrace$ with $\mathcal{\overline{I}}_{k,\ell}=\lbrace i \in \mathcal{I}_{k,\ell}:(\bm{D}_{k,\ell})_{i,j} \neq 0\rbrace$ and $s_{k,\ell}=\sum_{(a,b) < (k,\ell)}\vert \overline{\mathcal{I}}_{k,\ell}\vert $; and the set $\mathcal{P}_{k,\ell}$ enumerates the positions of  the elements of $\mathcal{I}_{k,l}$ in $\mathcal{I}_{k}$, i.e. $\mathcal{P}_{k,\ell}=\left\lbrace e \in [\vert \mathcal{I}_k\vert]: \exists j \in \mathcal{I}_{k,\ell } \text{ s.t. } j=j_e\right\rbrace$.
Notice that, for $s \in \mathcal{S}_{k,\ell}$ and $i \in [\vert \mathcal{S}_{k,\ell}\vert]$,  we have $(\bm{D})_{sj}=(\overline{\bm{D}}_{k,\ell})_{i j}$  for all $j \in [m]$.

Now, for all $k \in [p]$ and $j \in [m]$, we have 
\begin{align*}
    (\bm{\Pi}_k\bm{D})_{ij}&= \sum_{s=1}^{n_{\mathcal{I}}}(\bm{\Pi}_k)_{is}(\bm{D})_{sj} \\
    &= \sum_{\substack{k:k < \ell \\ \mathcal{I}_{k,\ell}\neq \varnothing}} \sum_{\substack{s \in \mathcal{S}_{k,\ell} \\ i \in \mathcal{P}_{k,\ell}}}(\bm{D})_{sj}-\sum_{\substack{\ell: \ell <k \\ \mathcal{I}_{\ell, k}\neq \varnothing}} \sum_{\substack{s \in \mathcal{S}_{\ell,k} \\ i \in \mathcal{P}_{k,\ell}} }(\bm{D})_{sj}\\
     &=\sum_{\substack{k:k < \ell \\ \overline{\mathcal{I}}_{k,\ell}\neq \varnothing}} (\overline{\bm{D}}_{k,\ell})_{ij}-\sum_{\substack{\ell: \ell <k \\ \overline{\mathcal{I}}_{\ell, k}\neq \varnothing}}(\overline{\bm{D}}_{\ell,k})_{ij},\\
    &=    \sum_{\substack{k:k < \ell \\ \mathcal{I}_{k,\ell}\neq \varnothing}} (\bm{D}_{k,\ell})_{ij} -\sum_{\substack{ \ell:\ell<k \\ \mathcal{I}_{\ell,k}\neq \varnothing}} (\bm{D}_{\ell,k})_{ij}, \text{ with } (\bm{D}_{k,\ell})_{ij}=\varnothing \text{ whenever } i \notin \mathcal{P}_{k,\ell},\\
    &=(\bm{D}_k)_{ij}.
\end{align*}
\end{proof}

Next, we present a link between the matrices $\bm{\Pi}_k$ defined by \eqref{compacte-form-add-variables} for the sets $\widehat{\mathcal{I}}_k $, and $\widehat{\bm{\Pi}}_k$ defined by \eqref{compacte-form-add-variables} for the sets $\widehat{\mathcal{I}}_k $.

        \begin{lemma}
 Let $d \in \mathbb{N}^*$ and $ n_{\widehat{\mathcal{I}}}=L_d n_{\mathcal{I}}$, where $n_\mathcal{I}= \sum_{k=1}^p \sum_{\ell>k} \lvert \mathcal{I}_{\ell,k} \rvert $. For $k \in [p]$, consider the matrix $\widehat{\bm{\Pi}}_k \in \mathbb{R}^{\vert \mathcal{\widehat{I}}_k \vert \times n_{\mathcal{\widehat{I}}}}$ defined by \eqref{compacte-form-add-variables} for the set $\mathcal{\widehat{I}}_k$. 
 Then, we have 
  $\widehat{\bm{\Pi}}_k =\bm{\Pi}_k \otimes \bm{I}$, where the matrix $\bm{\Pi}_k$ is associated to $\mathcal{I}_k$ as in  \eqref{definition-Sigma}, and $ \bm{I} \in \mathbb{S}^{L_{d}} $.
   \label{lemma-correspondance-sigmas}
\end{lemma}

\begin{proof}
First, recall that the matrix $   \bm{\widehat{\Pi}}_k \in \mathbb{R}^{ \vert \widehat{\mathcal{I}}_k \vert  \times n_{\widehat{\mathcal{I}} } }$ is defined by 
\begin{equation*}
   ( \bm{\widehat{\Pi}}_k )_{\hat{i}\hat{j}}=\left\lbrace\begin{aligned}
  & 1 \text{ if } \exists \ell >k: \widehat{\mathcal{I}}_{k,\ell} \neq \varnothing, \hat{i} \in \widehat{\mathcal{P}}_{k,\ell} \text{ and } \hat{j} \in \widehat{\mathcal{S}}_{k,\ell}  ,\\
  -& 1 \text{ if } \exists \ell <k: \widehat{\mathcal{I}}_{\ell,k} \neq \varnothing, \hat{i} \in \widehat{\mathcal{P}}_{\ell,k} \text{ and } \hat{j} \in \widehat{\mathcal{S}}_{\ell,k} ,\\
  & 0 \text{ otherwise },
\end{aligned}\right.
\end{equation*}
where $\widehat{\mathcal{S}}_{k,\ell}= \lbrace \widehat{s}_{k,\ell}+1,\cdots,\widehat{s}_{k,\ell}+ \vert \mathcal{\widehat{\overline{I}}}_{k,\ell}\vert \rbrace$ with $\widehat{\mathcal{\overline{I}}}_{k,\ell}=\lbrace i \in \mathcal{I}_{k,\ell}:(\widehat{\bm{D}}_{k,\ell})_{ij} \neq 0\rbrace$,  $\widehat{s}_{k,\ell}=\sum_{(a,b) < (k,\ell)}\vert \widehat{\overline{\mathcal{I}}}_{k,\ell}\vert $, and $\widehat{\mathcal{P}}_{k,\ell}=\left\lbrace \hat{j} \in [\vert \widehat{\mathcal{I}}_k\vert]: \exists \hat{e} \in \widehat{\mathcal{I}}_{k,\ell } \text{ such that } \hat{e}=\hat{e}_{\hat{j}}\right\rbrace$.

 Let $(\hat{i},\hat{j}) \in \vert \widehat{\mathcal{I}}_k \vert  \times [n_{\mathcal{\widehat{I}}}]$, $ (i,j) \in \vert \mathcal{I}_k \vert \times [n_{\mathcal{I}}]   $ and $i_d,j_d \in [d]$ such that $\hat{i}=(i-1)L_d+i_d$ and $\hat{j}=(j-1)L_d+j_d$. Because the identity matrix $\bm{I}_d$ is nonzero only when $i_d=j_d$, we have 
\begin{equation}
  (\bm{\Pi}_k \otimes \bm{I}_d)_{\hat{i}\hat{j}}=\left\lbrace \begin{aligned}
   & (\bm{\Pi}_k)_{ij} \text{ if } \hat{i}=(i-1)L_d+i_d \text{ and } \hat{j}=(j-1)L_d+i_d\\
    & 0 \text{ otherwise}.
  \end{aligned}\right.
    \label{proof-lemma6-Main-equality}
\end{equation}
Moreover, we will show that $\hat{j} \in \widehat{\mathcal{P}}_{k,\ell}$ whenever $j \in \mathcal{P}_{k,\ell}$, and $\hat{i} \in \widehat{\mathcal{S}}_{k,\ell}$ whenever $i \in \mathcal{S}_{k,\ell}$.
Let $j \in \mathcal{P}_{k,\ell}$, then $ j \in [\vert \mathcal{I}_{k} \vert] $ and $\exists e  \in \mathcal{I}_{k,\ell}$ such that $e=e_j $. Therefore, we have $ \hat{j} \in [\vert \widehat{\mathcal{I}}_k \vert]$, and for $\hat{e}=(e-1)L_d+j_d \in \mathcal{\widehat{I}}_{k,\ell}$, we have $\hat{e}=\hat{e}_{\hat{j}} $. It follows that $\hat{j} \in \widehat{\mathcal{P}}_{k,\ell}$.
Now, let $i \in \mathcal{S}_{k,\ell}$. There exists $ i_{k,\ell} \in [\vert \overline{\mathcal{I}}_{k,\ell} \vert]$ such that $ i=s_{k,\ell}+i_{k,\ell}$. Thus we have $\hat{i}=(s_{k,\ell}+i_{k,\ell}-1)L_d+i_d=L_ds_{k,\ell}+(i_{k,\ell}-1)L_d+i_d$. Since $\widehat{s}_{k,\ell}=L_ds_{k,\ell}$ and $ (i_{k,\ell}-1)L_d+i_d \in [\vert \widehat{\overline{\mathcal{I}}}_{k,\ell} \vert] $, we get $\hat{i} \in \widehat{\mathcal{S}}_{k,\ell}$.
 
 Therefore, for $\hat{i}=(i-1)L_d+i_d $ and $ \hat{j}=(j-1)L_d+i_d$, we obtain
\begin{align*}
  (\bm{\Pi}_k \otimes \bm{I}_d)_{\hat{i}\hat{j}}&=\left\lbrace \begin{aligned}
   & 1 \text{ if } \exists \ell >k: \widehat{\mathcal{I}}_{k,\ell} \neq \varnothing, \hat{i} \in \widehat{\mathcal{P}}_{k,\ell} \text{ and } \hat{j} \in \widehat{\mathcal{S}}_{k,\ell} ,\\
   -& 1 \text{ if } \exists \ell <k: \widehat{\mathcal{I}}_{\ell,k} \neq \varnothing, \hat{i} \in \widehat{\mathcal{P}}_{\ell,k} \text{ and } \hat{j} \in \widehat{\mathcal{S}}_{\ell,k},\\
    & 0 \text{ otherwise},
  \end{aligned}\right.\\
  &=\widehat{\bm{\Pi}}_k.
\end{align*}
\end{proof}

\subsection{Physical interpretation of the additional variables and their elimination}
\label{sec:physical}
In this section, we provide a physical interpretation for the additionally introduced variables $\bm{c}$ and $\bm{d}$. To this goal, we deal with it from a prior-to-mSOS point of view and we exploit the equivalence of \eqref{compliance_problem} with \eqref{compliance_problem_sdp} and write the prior AD decomposition of \eqref{compliance_problem} as


\begin{equation}
\begin{aligned}
  \tilde{\gamma}_{\text{AD}}^*=&\underset{\substack{\bm{x}, \bm{u}, \bm{d}}}{\min} \sum_{k=1}^p c_k,  \\
 \text{s.t. }& \bm{K}_k(\bm{x})\bm{u}_k-(\bm{f}_k(\bm{x})+\bm{\Pi}_k\bm{d})= \bm{0},k \in [p],\\
 &\overline{w}-\sum_{e=1}^{n_\mathrm{e}}\ell_e \rho_e x_e \geq 0, \\
 & \bm{f}_k(\bm{x})^T\bm{u}_k= c_k,\\
 & \bm{x} \geq \bm{0},
\end{aligned}
\label{AD-compliance-phys}
\end{equation}
in which for $k\in [p]$, $\bm{u}_k=\bm{u}_{\mathcal{I}_k} \in \mathbb{R}^{\vert \mathcal{I}_k\vert}$ is a part of the displacement vector $\bm{u}$ corresponding to the generalized nodal displacements related to the subdomain $\mathcal{D}_k$. According to \cite[Theorem 5]{kovcvara2021decomposition}, we have $\tilde{\gamma}_{\text{AD}}^*=\gamma_{\text{AD}}^*=\gamma^*$.


From \eqref{AD-compliance-phys}, it follows that $c_k$ has the meaning of the potential energy of the external forces (compliance) of the elements in the subset $\mathcal{I}_k$. Clearly, $\sum_{k=1}^p c_k$ provides the overall compliance $\gamma$.

Furthermore, the equilibrium equations given in \eqref{AD-compliance-phys} constitute equilibrium of the substructure formed by elements in the partition $\mathcal{D}_k$. However, due to the substructures formed by elements $\mathcal{D}_k$ being connected in reality to other elements, i.e., $\forall k \ \exists \ell: (\mathcal{I}_{k,\ell}\cup \mathcal{I}_{\ell,k}) \neq \varnothing$, it is needed to introduce so-called interface forces $\bm{d}_k$ that represent the mechanical interaction between the isolated elements $\mathcal{D}_k$ and the remainder of the structure $[n_\mathrm{e}]\setminus \mathcal{D}_k$. Following Newton's Third Law, these forces maintain the equilibrium conditions that existed in the original complete structure. For each interface where the structure is sectioned, the forces and moments acting on the isolated portion are equal and opposite to those acting on the remaining structure at that interface. 

\begin{example}\label{ex:frame3}
Consider the structure shown in Fig. \ref{fig:free_body_diagram}a that contains three structural elements $1$--$3$, four nodes $a$--$d$, and we set $\mathcal{D}_1 = {1}$, $\mathcal{D}_2 = {2}$ and $\mathcal{D}_3 = {3}$. Based on the kinematic boundary conditions, there is one degree of freedom (rotation) allowed at the node $a$, which we denote as $\{1\}$. At the nodes $b$ and $c$, both translations and rotations are allowed, resulting in $3$ degrees of freedom per node. We label these as $\{2,3,4\}$ and $\{5,6,7\}$, respectively. Finally, the node $d$ prevents translations and rotation, so that there is no degree of freedom. Since element $1$ connects nodes $a$ and $b$, we have $\mathcal{I}_1 = \{1, 2, 3, 4\}$. Analogously, we get $\mathcal{I}_2 = \{2,3,4,5,6,7\}$, and $\mathcal{I}_3 = \{5,6,7\}$.

Furthermore, we evaluate the intersections as $\mathcal{I}_{1,2} = \{2,3,4\}$ and $\mathcal{I}_{2,3} = \{5,6,7\}$. Because of $ \mathbf{f}(\mathbf{x}) = \begin{bmatrix}0 & 0 & -q_1& 0 & 0 & q_2 & 0\end{bmatrix}^T \in \mathbb{R}^{n_\mathrm{dof}\times 1}$, we need to introduce only one variable for each intersection in $\mathcal{I}_{k,\ell}$. In particular, we introduce $d_{1,2}^{(1)}$, $d_{1,2}^{(2)}$ and $d_{1,2}^{(3)}$ which denote the horizontal, vertical, and moment internal forces that balance the interaction of the domains $\mathcal{D}_1$ and $\mathcal{D}_2$ at the node $b$. Analogously, we define $d_{2,3}^{(1)}$, $d_{2,3}^{(2)}$ and $d_{2,3}^{(3)}$ for the horizontal, vertical and moment internal forces balancing interaction between domains $\mathcal{D}_2$ and $\mathcal{D}_3$. This yields the free-body diagram shown in Fig.~\ref{fig:free_body_diagram}b.
\end{example}

\begin{figure}[!htbp]
    \centering
    \begin{subfigure}{\textwidth}
    \centering
    \begin{tikzpicture}
    \scaling{2.0}
    \point{a}{0}{0};
    \point{b}{1}{0};
    \point{bu}{1}{0.1};
    \point{c}{2}{0};
    \point{cu}{2}{0.1};
    \point{d}{3}{0};
    \beam{1}{a}{b};
    \beam{1}{b}{c};
    \beam{1}{c}{d};
    \notation{2}{b}{};
    \notation{2}{c}{};
    \support{1}{a};
    \support{3}{d}[90];
    \load{1}{b}[270][0.5][-0.8];
    \load{1}{c}[270][-0.5][-0.3];
    \notation{1}{b}{$q_1$}[above=7mm];
    \notation{1}{c}{$q_2$}[above=7mm];
    \notation{4}{a}{b}[$1$];
    \notation{4}{b}{c}[$2$];
    \notation{4}{c}{d}[$3$];
    \dimensioning{1}{a}{b}{-1.2}[$\ell_1$];
    \dimensioning{1}{b}{c}{-1.2}[$\ell_2$];
    \dimensioning{1}{c}{d}{-1.2}[$\ell_3$];
    \notation{1}{a}{\textcircled{$a$}}[above];
    \notation{1}{b}{\textcircled{$b$}}[below=2mm];
    \notation{1}{c}{\textcircled{$c$}}[below=3mm];
    \notation{1}{d}{\textcircled{$d$}}[below left=0.5mm];
    \end{tikzpicture}
    \caption{}
    \end{subfigure}
    
    \vspace{3mm}\begin{subfigure}{\textwidth}
    \centering
    \begin{tikzpicture}
    \scaling{2.0}
    \point{a}{0}{0};
    \notation{1}{a}{\textcircled{$a$}}[above];
    \point{b1}{1}{0};
    \point{b2}{2}{0};
    \notation{1}{b1}{\textcircled{$b$}}[below];
    \notation{1}{b2}{\textcircled{$b$}}[above];
    \point{c1}{3}{0};
    \point{c2}{4}{0};
    \notation{1}{c1}{\textcircled{$c$}}[below=1mm];
    \notation{1}{c2}{\textcircled{$c$}}[above];
    \point{d}{5}{0};
    \notation{1}{d}{\textcircled{$d$}}[below left=0.5mm];
    \point{bu}{2}{0.2};
    \point{cu}{3}{0.2};
    \beam{1}{a}{b1};
    \beam{1}{b2}{c1};
    \beam{1}{c2}{d};
    \support{1}{a};
    \support{3}{d}[90];
    \load{1}{b2}[270][0.5][-1.2];
    \load{1}{c1}[270][-0.5][-0.7];
    \notation{1}{b2}{$q_1$}[above=11mm];
    \notation{1}{c1}{$q_2$}[above=11mm];
    \notation{4}{a}{b1}[$1$];
    \notation{4}{b2}{c1}[$2$];
    \notation{4}{c2}{d}[$3$];
    \load{1}{b1}[180][0.5][-0.6];
    \load{1}{b1}[270][0.5][-0.6];
    \load{2}{b1}[80][-70];
    \notation{1}{b1}{$d_{1,2}^{(1)}$}[below right=0.5mm];
    \notation{1}{b1}{$d_{1,2}^{(2)}$}[above left];
    \notation{1}{b1}{$d_{1,2}^{(3)}$}[above right=2.5mm];
    \load{1}{b2}[180][-0.5][0.6];
    \load{1}{b2}[270][-0.5][0.6];
    \load{2}{b2}[-170][70];
    \notation{1}{b2}{$d_{1,2}^{(1)}$}[above left];
    \notation{1}{b2}{$d_{1,2}^{(2)}$}[below right];
    \notation{1}{b2}{$d_{1,2}^{(3)}$}[below left=2.5mm];
    \load{1}{c1}[180][0.5][-0.6];
    \load{1}{c1}[270][0.5][-0.6];
    \load{2}{c1}[80][-70];
    \notation{1}{c1}{$d_{2,3}^{(1)}$}[below right=0.5mm];
    \notation{1}{c1}{$d_{2,3}^{(2)}$}[above left];
    \notation{1}{c1}{$d_{2,3}^{(3)}$}[above right=2.5mm];
    \load{1}{c2}[180][-0.5][0.6];
    \load{1}{c2}[270][-0.5][0.6];
    \load{2}{c2}[-170][70];
    \notation{1}{c2}{$d_{2,3}^{(1)}$}[above left=0.5mm];
    \notation{1}{c2}{$d_{2,3}^{(2)}$}[below right];
    \notation{1}{c2}{$d_{2,3}^{(3)}$}[below left=2.5mm];
    \dimensioning{1}{a}{b1}{-1.2}[$\ell_1$];
    \dimensioning{1}{b2}{c1}{-1.2}[$\ell_2$];
    \dimensioning{1}{c2}{d}{-1.2}[$\ell_3$];
    \end{tikzpicture}
    \caption{}
    \end{subfigure}
    \caption{Physical interpretation of the variables $\bm{d}_k$: (a) A beam composed of three segments; (b) Free-body diagram showing internal forces and moments at domain interfaces.}
    \label{fig:free_body_diagram}
\end{figure}
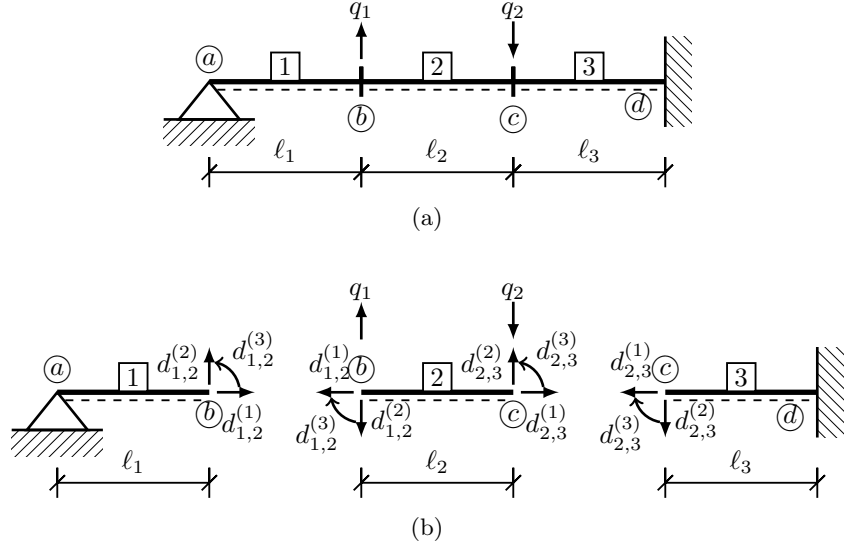

\subsubsection*{Additional variables elimination}\label{sec:po_variable_reduction}

Next, we consider the variable elimination procedure and the related physical interpretation. To this goal, let $\bm{x} \ge \bm{0}$ and $\tilde{\bm{x}}>0$. By construction of the stiffness matrix, the null space of the matrix $\bm{K}_k(\tilde{\bm{x}})$ is constant and $\mathrm{Null}(\bm{K}_k(\tilde{\bm{x}})) \subseteq \mathrm{Null}(\bm{K}_k(\bm{x}))$. Thus, for any $\bm{x} \in \mathbb{R}^{n_x}$, if $\bm{v} \in \text{Null}(\bm{K}_k(\bm{x}))$, then by \eqref{stifness_matrix}, we have $\bm{K}_{0}\bm{v}+\sum_{e\in \mathcal{D}_k} \bm{K}_{e}^{(1)}x_e\bm{v}+\sum_{e\in \mathcal{D}_k}\bm{K}_{e}^{(2)}x_e^2\bm{v}+\sum_{e\in \mathcal{D}_k}\bm{K}_{e}^{(3)}x_e^3\bm{v} =\bm{0}$. Because the monomials $1, x_k, x_k^2,x_k^3 $ are linearly independent and $\bm{K}_0, \bm{K}_{e}^{(i)} \succeq 0$ for $i \in \lbrace 1, 2, 3 \rbrace$, we get $\bm{K}_0 \bm{v}=\bm{0}$ and $\bm{K}_{e}^{(i)}\bm{v}=\bm{0}$ for all $e \in [n_e]$ and $i$. Therefore, the null space of $\bm{K}_k(\bm{x})$ remains in the intersection of the null spaces of $\bm{K}_{0}$, $\bm{K}_{e}^{(1)}$, $\bm{K}_{e}^{(2)}$ and $\bm{K}_{e}^{(3)}$.

Moreover, in the topology optimization, there exists cheaper alternatives, (see \cite{xue2019method,shklarski2009computing}), as the null space has also a physical interpretation of rigid body displacements or rotations. This makes the application of Proposition \ref{proposition-projection-msos} suitable.

\begin{example}
Consider again the setting of Example \ref{ex:frame3}. For $\mathcal{D}_1$, we have the equilibrium system $\bm{K}_1 \bm{u}_1 = \bm{f}_1 + \bm{D}_1$, in which $I_1(x_1)$ is a non-negative polynomial of degree at most three, non-negative constraints $E_1, \ell_1 \in \mathbb{R}_{> 0}$, 
$$\bm{K}_1 (x_1) =\bm{\mathscr{E}}_{\mathcal{I}_1}^T \begin{bmatrix}
        \frac{4E_1 I_1(x_1)}{\ell_1}  & -\frac{6E_1 I_1(x_1)}{\ell_1^2} & \frac{2 E_1 I_1(x_1)}{\ell_1}\\
        0 & \frac{E_1 x_1}{\ell_1} & 0 & 0\\
        -\frac{6 E_1 I_1(x_1)}{\ell_1^2} & 0 & \frac{12 E_1 I_1(x_1)}{\ell_1^3} & -\frac{6 E_1 I_1(x_1)}{\ell_1^2}\\
        \frac{2E_1 I_1(x_1)}{\ell_1} & 0 & -\frac{6 E_1 I_1(x_1)}{\ell_1^2} & \frac{4 E_1 I_1(x_1)}{\ell_1}
        \end{bmatrix}\bm{\mathscr{E}}_{\mathcal{I}_1},
           $$
and $ \bm{f}_1 + \bm{D}_1 = \bm{\mathscr{E}}_{\mathcal{I}_1}^T\begin{bmatrix}
    0 &   d_{1,2}^{(1)} &  d_{1,2}^{(2)} & d_{1,2}^{(3)}  \end{bmatrix}^T.$ For all $x_1\ge 0$, the stiffness matrix $\bm{K}_1(x_1)$ has one singular eigenvalue associated with the eigenvector $\bm{n}_1 = \begin{bmatrix}1 & 0 & \ell_1 & 1\end{bmatrix}^{T}$. Hence, $(\bm{f}_1+\bm{D}_1)^\mathrm{T} \bm{n}_1 = 0$ results in the equality
$$
\circlearrowleft_a:\quad d_{1,2}^{(2)}\ell_1 + d_{1,2}^{(3)} = 0,
$$
which is the moment static equilibrium of the substructure $\mathcal{D}_1$ around the point $a$, which is the only rigid body mode not prevented by the kinematic boundary conditions. Finally, the range space bases evaluate as
$\bm{P}_1=\bm{\mathscr{E}}_{\mathcal{I}_1}^T \begin{bmatrix}
    0 & -\ell_1 & _1
\end{bmatrix}^T$.

Further, let us consider $\mathcal{D}_2$ and the equilibrium system $\bm{K}_2 \bm{u}_2 = \bm{f}_2 + \bm{D}_2$. Then, for $I_2(x_2)$ being a non-negative polynomial of degree at most three and non-negative constraints $E_2, \ell_2 \in \mathbb{R}_{> 0}$, we have
$$
\bm{K}_2(x_2) =
\bm{\mathscr{E}}_{\mathcal{I}_2}^T\begin{bmatrix}
        \frac{E_2 x_2}{\ell_2} & 0 & 0 & -\frac{E_2 x_2}{\ell_2} & 0 & 0\\
        0 & \frac{12 E_2 I_2(x_2)}{\ell_2^3} & \frac{6 E_2 I_2(x_2)}{\ell_2^2} & 0 & -\frac{12 E_2 I_2(x_2)}{\ell_2^3} & \frac{6E_2 I_2(x_2)}{\ell_2^2}\\
        0 & \frac{6 E_2 I_2(x_2)}{\ell_2^2} &\frac{4E_2 I_2(x_2)}{\ell_2} & 0 & -\frac{6E_2 I_2(x_2)}{\ell_2^2} & \frac{2 E_2 I_2(x_2)}{\ell_2}\\
        -\frac{E_2 x_2}{\ell_2}& 0 & 0 & \frac{E_2 x_2}{\ell_2} & 0 & 0\\
        0 & -\frac{12E_2I_2(x_2)}{\ell_2^3} &-\frac{6 E_2 I_2(x_2)}{\ell_2^2} & 0 & \frac{12 E_2 I_2(x_2)}{\ell_2^3} & -\frac{6 E_2 I_2(x_2)}{\ell_2^2}\\
        0 & \frac{6E_2 I_2(x_2)}{\ell_2^2} & \frac{2E_2 I_2(x_2)}{\ell_2} & 0 & -\frac{6 E_2 I_2(x_2)}{\ell_2^2} & \frac{4 E_2 I_2(x_2)}{\ell_2}
    \end{bmatrix}\bm{\mathscr{E}}_{\mathcal{I}_2}
$$
and $ \bm{f}_2+ \bm{D}_2 = \bm{\mathscr{E}}_{2}^T\begin{bmatrix}
    -d_{1,2}^{(1)} & q_1-d_{1,2}^{(2)} & -d_{1,2}^{(3)} & d_{2,3}^{(1)} & d_{2,3}^{(2)}-q_2 & d_{2,3}^{(3)}
\end{bmatrix}^T,
$
with point loads $q_1, q_2 \in \mathbb{R}$ as in Fig. \ref{fig:free_body_diagram}. For $x_2 \ge 0$, the stiffness matrix $\bm{K}_2 (x_2)$ has singularities associated to the three Null space bases $\bm{n}_1 = \begin{bmatrix} 1 & 0 & 0 & 1 & 0 & 0 \end{bmatrix}^T$, $\bm{n}_2 = \begin{bmatrix} 0 & 1 & 0 & 0 & 1 & 0 \end{bmatrix}^T$ and $\bm{n}_3 = \begin{bmatrix} 0 & 0 & 1 & 0 & \ell_2 & 1 \end{bmatrix}^\mathrm{T}$. These then correspond to static equilibria for the three rigid body modes not restricted at $\mathcal{D}_2$:
\begin{align*}
    \rightarrow:&\quad -d_{1,2}^{(1)} + d_{2,3}^{(1)} = 0,\\
    \uparrow:&\quad - d_{1,2}^{(2)} + d_{2,3}^{(2)} = q_2-q_1,\\
    \circlearrowleft_b:&\quad -d_{1,2}^{(3)} + d_{2,3}^{(3)} + \ell_2 d_{2,3}^{(2)} = \ell_2 q_2.
\end{align*}
Hence, the range space bases evaluate as
$
\bm{P}_2 =\bm{\mathscr{E}}_{\mathcal{I}_2}^T \begin{bmatrix}
        -1 & 0 & 0 \\
        0 & -1 & -\ell_2 \\
        0 & 0 & -1 
\end{bmatrix}^T.
$

Finally, for $\mathcal{D}_3$, $E_3,\ell_3 \in \mathbb{R}_{>0}$ and a non-negative polynomial $I_3(x_3)$ of degree at most 3, we have
$$
\bm{K}_3 =\bm{\mathscr{E}}_{\mathcal{I}_3}^T \begin{bmatrix}\frac{E_3 x_3}{\ell_3} & 0 & 0\\ 0 & \frac{12 E_3 I_3(x_3)}{\ell_3^3} & \frac{6 E_3 I_3 (x_3)}{\ell_3^2}\\ 0 & \frac{6 E_3 I_3(x_3)}{\ell_3^2} & \frac{4 E_3 I_3(x_(3)}{\ell_3} \end{bmatrix}\bm{\mathscr{E}}_{\mathcal{I}_3},
\bm{f}_3 + \bm{D}_3 = \bm{\mathscr{E}}_{\mathcal{I}_3}^T\begin{bmatrix}
    -d_{2,3}^{(1)} \\
    -d_{2,3}^{(2)} \\
    -d_{2,3}^{(3)}
\end{bmatrix},
$$
in which $\bm{K}_3(x_3)$ is regular due to prevented translations and rotation by the clamped support at the right end the stucture, recall Fig.~\ref{fig:free_body_diagram}. Hence, $\bm{P}_3 = \bm{I}_3$.

Now, we combine the Null space equations into a single system, and express its solution as
%
%

%
$$
\begin{multlined}\bm{D} = \begin{bmatrix}
     0\\
 (q_1 \ell_2^2 + \ell_1 q_1 \ell_2 + q_1 - q_2)/(2 \ell_1^2 + 2 \ell_1 \ell_2 + \ell_2^2 + 2)\\
- (\ell_1 (q_1 - q_2) - (\ell_2 q_1 (\ell_2^2 + 2))/2)/(2 \ell_1^2 + 2 \ell_1 \ell_2 + \ell_2^2 + 2) - (\ell_2 q_1)/2\\
 0\\
 q_2 - (2 q_1 \ell_1^2 + \ell_2 q_1 \ell_1 + q_1 + q_2)/(2 \ell_1^2 + 2 \ell_1 \ell_2 + \ell_2^2 + 2)\\
 -(\ell_1 (q_1 - q_2) - \ell_2 (q_2 + (q_1 (2 \ell_1^2 + 2))/2))/(2 \ell_1^2 + 2 \ell_1 \ell_2 + \ell_2^2 + 2)
\end{bmatrix} \\+ \begin{bmatrix}
        1 & 0 & 0\\
        0 & -\ell_2/\ell_1 & -1/\ell_1\\
        0 & \ell_2 & 1\\
        1 & 0 & 0\\
        0 & 1 & 0\\
        0 & 0 & 1
\end{bmatrix}
\begin{bmatrix}
    d_R^{(1)}\\
    d_R^{(2)}\\
    d_R^{(3)}
\end{bmatrix},\end{multlined}$$
which allows us to substitute the original variables $\bm{D}$ as well as to reduce the sizes of matrix inequalities for the domains $\mathcal{D}_1$ and $\mathcal{D}_2$ to $4$ and $4$ from the original $5$ and $7$. 


\end{example}

\subsection{Tables}
\label{appendix-tables}

Here, we present a detailed data (Tables \ref{tab:sw1}--\ref{tab:sw4}) corresponding to Figures \ref{fig:mSOS-rel1}--\ref{fig:NMT-rel4} presented in Section \ref{sec:numerics}.

\begin{table}[!htbp]
    \centering
    \scriptsize
    \setlength{\tabcolsep}{1pt}
    \renewcommand{\arraystretch}{1.0}
     \resizebox{\textwidth}{!}{%
    \begin{tabular}{c|c|c|c|c|c}
         \toprule
        $n_e$ & & mSOS & mSOS+AD & mSOS+NMT & mSOS+NMT+AD\\
        \hline
         1 & $t$ [s] & $0.01$ & $0.01$ & $0.01$ & $0.01$ \\
           & l.b. & $3.98\times10^{-2}$ & $3.98\times10^{-2}$ & $3.98\times10^{-2}$ & $3.98\times10^{-2}$\\
           & $n$ & $5$ & $5$ & $5$ & $5$ \\
           & size & $3\times1, 2\times3$ & $3\times1, 2\times3$ & $3\times1, 2\times3$ & $3\times1, 2\times3$ \\
        \hline
         2 & $t$ [s] & $0.02$ & $0.02$ & $0.02$ & $0.02$ \\
           & l.b. & $1.99\times10^{-2}$ & $1.99\times10^{-2}$ & $1.99\times10^{-2}$ & $1.99\times10^{-2}$\\
           & $n$ & $9$ & $9$ & $9$ & $9$ \\
           & size & $4\times1, 1\times4, 1\times6$ & $4\times1, 1\times4, 1\times6$ & $4\times1, 1\times4, 1\times6$ & $4\times1, 1\times4, 1\times6$ \\
        \hline
         3 & $t$ [s] & $0.02$ & $0.02$ & $0.02$ & $0.02$ \\
           & l.b. & $1.32\times10^{-2}$ & $1.32\times10^{-2}$ & $1.32\times10^{-2}$ & $1.32\times10^{-2}$\\
           & $n$ & $14$ & $16$ & $14$ & $16$ \\
           & size & $5\times1, 1\times5, 1\times9$ & $5\times1, 1\times4, 1\times5, 1\times7$ & $5\times1, 1\times5, 1\times9$ & $5\times1, 1\times4, 1\times5, 1\times7$ \\
        \hline
         4 & $t$ [s] & $0.02$ & $0.02$ & $0.02$ & $0.02$ \\
           & l.b. & $9.89\times10^{-2}$ & $9.89\times10^{-3}$ & $9.89\times10^{-3}$ & $9.89\times10^{-3}$ \\
           & $n$ & $20$ & $23$ & $20$ & $23$ \\
           & size & $6\times1, 1\times6, 1\times12$ & $6\times1, 2\times4, 1\times6, 1\times7$ & $6\times1, 1\times6, 1\times12$ & $6\times1, 2\times4, 1\times6, 1\times7$ \\
        \hline
         5 & $t$ [s] & $0.03$ & $0.02$ & $0.03$ & $0.02$\\
           & l.b. & $7.91\times10^{-3}$ & $7.91\times10^{-3}$ & $7.91\times10^{-3}$ & $7.91\times10^{-3}$\\
           & $n$ & $27$ & $31$ & $27$ & $31$ \\
           & size & $7\times1, 1\times7, 1\times15$ & $7\times1, 3\times4, 2\times7$ & $7\times1, 1\times7, 1\times15$ & $7\times1, 3\times4, 2\times7$ \\
        \hline
         6 & $t$ [s] & $0.03$ & $0.02$ & $0.03$ & $0.02$\\
           & l.b. & $6.58\times10^{-3}$ & $6.58\times10^{-3}$ & $6.58\times10^{-3}$ & $6.58\times10^{-3}$\\
           & $n$ & $35$ & $40$ & $35$ & $40$ \\
           & size & $8\times1, 1\times8, 1\times18$ & $8\times1, 4\times4, 1\times7, 1\times8$ & $8\times1, 1\times8, 1\times18$ & $8\times1, 4\times4, 1\times7, 1\times8$ \\
        \hline
         7 & $t$ [s] & $0.04$ & $0.03$ & $0.04$ & $0.03$\\
           & l.b. & $5.64\times10^{-3}$ & $5.64\times10^{-3}$ & $5.64\times10^{-3}$ & $5.64\times10^{-3}$\\
           & $n$ & $44$ & $50$ & $44$ & $50$ \\
           & size & $9\times1, 1\times9, 1\times21$ & $9\times1, 5\times4, 1\times7, 1\times9$ & $9\times1, 1\times9, 1\times21$ & $9\times1, 5\times4, 1\times7, 1\times9$ \\
        \hline
         8 & $t$ [s] & $0.04$ & $0.03$ & $0.04$ & $0.03$\\
           & l.b. & $4.94\times10^{-3}$ & $4.94\times10^{-3}$ & $4.94\times10^{-3}$ & $4.94\times10^{-3}$\\
           & $n$ & $54$ & $61$ & $54$ & $61$ \\
           & size & $10\times1, 1\times10, 1\times24$ & $10\times1, 6\times4, 1\times7, 1\times10$ & $10\times1, 1\times10, 1\times24$ & $10\times1, 6\times4, 1\times7, 1\times10$ \\
        \hline
         9 & $t$ [s] & $0.05$ & $0.09$ & $0.05$ & $0.03$\\
           & l.b. & $4.39\times10^{-3}$ & $4.39\times10^{-3}$ & $4.39\times10^{-3}$ & $4.39\times10^{-3}$\\
           & $n$ & $65$ & $73$ & $65$ & $73$ \\
           & size & $11\times1, 1\times11, 1\times27$ & $11\times1, 7\times4, 1\times7, 1\times11$ & $11\times1, 1\times11, 1\times27$ & $11\times1, 7\times4, 1\times7, 1\times11$ \\
        \hline
        10 & $t$ [s] & $0.12$ & $0.03$ & $0.11$ & $0.03$ \\
           & l.b. & $3.95\times10^{-3}$ & $3.95\times10^{-3}$ & $3.95\times10^{-3}$ & $3.95\times10^{-3}$\\
           & $n$ & $77$ & $86$ & $77$ & $86$ \\
           & size & $12\times1, 1\times12, 1\times30$ & $12\times1, 8\times4, 1\times7, 1\times12$ & $12\times1, 1\times12, 1\times30$ & $12\times1, 8\times4, 1\times7, 1\times12$ \\
        \hline
        11 & $t$ [s] & $0.61$ & $0.04$ & $0.59$ & $0.04$\\
           & l.b. & $3.59\times10^{-3}$ & $3.59\times10^{-3}$ & $3.58\times10^{-3}$ & $3.59\times10^{-3}$\\
           & $n$ & $90$ & $100$ & $90$ & $100$ \\
           & size & $13\times1, 1\times13, 1\times33$ & $13\times1, 9\times4, 1\times7, 1\times13$ & $13\times1, 1\times13, 1\times33$ & $13\times1, 9\times4, 1\times7, 1\times13$ \\
        \hline
        12 & $t$ [s] & $0.14$ & $0.04$ & $0.14$ & $0.05$\\
           & l.b. & $3.29\times10^{-3}$ & $3.29\times10^{-3}$ & $3.29\times10^{-3}$ & $3.29\times10^{-3}$\\
           & $n$ & $104$ & $115$ & $104$ & $115$ \\
           & size & $14\times1, 1\times14, 1\times36$ & $14\times1, 10\times4, 1\times7, 1\times14$ & $14\times1, 1\times14, 1\times36$ & $14\times1, 10\times4, 1\times7, 1\times14$ \\
        \hline
        13 & $t$ [s] & $0.13$ & $0.05$ & $0.13$ & $0.07$\\
           & l.b. & $3.04\times10^{-3}$ & $3.04\times10^{-3}$ & $3.04\times10^{-3}$ & $3.04\times10^{-3}$\\
           & $n$ & $119$ & $131$ & $119$ & $131$ \\
           & size & $15\times1, 1\times15, 1\times39$ & $15\times1, 11\times4, 1\times7, 1\times15$ & $15\times1, 1\times15, 1\times39$ & $15\times1, 11\times4, 1\times7, 1\times15$ \\
        \hline
        14 & $t$ [s] & $0.14$ & $0.06$ & $0.13$ & $0.06$\\
           & l.b. & $2.82\times10^{-3}$ & $2.82\times10^{-3}$ & $2.82\times10^{-3}$ & $2.82\times10^{-3}$\\
           & $n$ & $135$ & $148$ & $135$ & $148$ \\
           & size & $16\times1, 1\times16, 1\times42$ & $16\times1, 12\times4, 1\times7, 1\times16$ & $16\times1, 1\times16, 1\times42$ & $16\times1, 12\times4, 1\times7, 1\times16$ \\
        \hline
        15 & $t$ [s] & $0.33$ & $0.07$ & $0.33$ & $0.07$\\
           & l.b. & $2.63\times10^{-3}$ & $2.63\times10^{-3}$ & $2.63\times10^{-3}$ & $2.63\times10^{-3}$\\
           & $n$ & $152$ & $166$ & $152$ & $166$ \\
           & size & $17\times1, 1\times17, 1\times45$ & $17\times1, 13\times4, 1\times7, 1\times17$ & $17\times1, 1\times17, 1\times45$ & $17\times1, 13\times4, 1\times7, 1\times17$ \\
              \bottomrule  
    \end{tabular}}
    \caption{Comparison of four moment-SOS hierarchies for the first-order relaxation of a double-hinged beam problem with varying number of elements ($n_\mathrm{e}$). The hierarchies are: standard moment-SOS (mSOS), mSOS with arrow decomposition (mSOS+AD), mSOS using nonmixed-term basis (mSOS+NMT), and mSOS combining nonmixed-term basis with arrow decomposition (mSOS+NMT+AD). For each method and problem size, we report: solution time ($t$) in seconds, computed lower bound (l.b.), number of variables in the relaxation ($n$), and sizes of matrix constraints (where $a\times b$ indicates $a$ matrix constraints of size $\mathbb{S}^b$).}
    \label{tab:sw1}
\end{table}

\begin{table}[!htbp]
    \centering
    \scriptsize
    \setlength{\tabcolsep}{1pt}
    \renewcommand{\arraystretch}{1.0}
     \resizebox{\textwidth}{!}{%
    \begin{tabular}{c|c|c|c|c|c}
         \toprule
        $n_e$ & & mSOS & mSOS+AD & mSOS+NMT & mSOS+NMT+AD\\
        \hline
         1 & $t$ [s] & $0.03$ & $0.02$ & $0.02$ & $0.02$\\
           & l.b. & $^*4.17\times10^{-2}$ & $^*4.17\times10^{-2}$ & $^*4.17\times10^{-2}$ & $^*4.17\times10^{-2}$\\
           & $n$ & $14$ & $14$ & $14$ & $14$ \\
           & size & $3\times3, 1\times6, 1\times9$ & $3\times3, 1\times6, 1\times9$ & $3\times3, 1\times5, 1\times9$ & $3\times3, 1\times5, 1\times9$ \\
        \hline
         2 & $t$ [s] & $0.05$ & $0.06$ & $0.06$ & $0.07$\\
           & l.b. & $^*4.00\times10^{-2}$ & $^*4.00\times10^{-2}$ & $^*4.00\times10^{-2}$ & $^*4.00\times10^{-2}$\\
           & $n$ & $34$ & $34$ & $34$ & $34$ \\
           & size & $4\times4, 1\times10, 1\times24$ & $4\times4, 1\times10, 1\times24$ & $4\times4, 1\times7, 1\times24$ & $4\times4, 1\times7, 1\times24$ \\
        \hline
         3 & $t$ [s] & $0.26$ & $0.14$ & $0.44$ & $0.13$ \\
           & l.b. & $^*3.91\times10^{-2}$ & $^*3.91\times10^{-2}$ & $^*3.91\times10^{-2}$ & $^*3.91\times10^{-2}$ \\
           & $n$ & $69$ & $99$ & $68$ & $98$ \\
           & size & $5\times5, 1\times15, 1\times45$ & $5\times5, 1\times15, 1\times20, 1\times35$ & $5\times5, 1\times9, 1\times45$ & $5\times5, 1\times9, 1\times20, 1\times35$ \\
        \hline
         4 & $t$ [s] & $0.39$ & $0.29$ & $1.53$ & $0.25$\\
           & l.b. & $3.79\times10^{-2}$ & $3.79\times10^{-2}$ & $3.78\times10^{-2}$ & $3.79\times10^{-2}$\\
           & $n$ & $125$ & $188$ & $120$ & $183$ \\
           & size & $6\times6, 1\times21, 1\times72$ & $6\times6, 1\times21, 2\times24, 1\times42$ & $6\times6, 1\times11, 1\times72$ & $6\times6, 1\times11, 2\times24, 1\times42$ \\
        \hline
         5 & $t$ [s] & $1.91$ & $0.49$ & $5.62$ & $0.41$\\
           & l.b. & $3.63\times10^{-2}$ & $3.63\times10^{-2}$ & $3.62\times10^{-2}$ & $3.62\times10^{-2}$\\
           & $n$ & $209$ & $321$ & $194$ & $306$ \\
           & size & $7\times7, 1\times28, 1\times105$ & $7\times7, 4\times28, 1\times49$ & $7\times7, 1\times13, 1\times105$ & $7\times7, 1\times13, 3\times28, 1\times49$ \\
        \hline
         6 & $t$ [s] & $20.62$ & $0.66$ & $10.98$ & $0.56$\\
           & l.b. & $3.43\times10^{-2}$ & $3.43\times10^{-2}$ & $3.42\times10^{-2}$ & $3.42\times10^{-2}$\\
           & $n$ & $329$ & $509$ & $294$ & $474$ \\
           & size & $8\times8, 1\times36, 1\times144$ & $8\times8, 4\times32, 1\times36, 1\times56$ & $8\times8, 1\times15, 1\times144$ & $8\times8, 1\times15, 4\times32, 1\times56$ \\
        \hline
         7 & $t$ [s] & $13.10$ & $1.17$ & $8.95$ & $0.82$\\
           & l.b. & $3.22\times10^{-2}$ & $3.22\times10^{-2}$ & $3.21\times10^{-2}$ & $3.21\times10^{-2}$\\
           & $n$ & $494$ & $764$ & $424$ & $694$ \\
           & size & $9\times9, 1\times45, 1\times189$ & $9\times9, 5\times36, 1\times45, 1\times63$ & $9\times9, 1\times17, 1\times189$ & $9\times9, 1\times17, 5\times36, 1\times63$ \\
        \hline
         8 & $t$ [s] & $27.34$ & $1.85$ & $17.50$ & $1.49$\\
           & l.b. & $3.01\times10^{-2}$ & $3.01\times10^{-2}$ & $^+3.00\times10^{-2}$ & $3.01\times10^{-2}$\\
           & $n$ & $714$ & $1099$ & $588$ & $973$ \\
           & size & $10\times10, 1\times55, 1\times240$ & $10\times10, 6\times40, 1\times55, 1\times70$ & $10\times10, 1\times19, 1\times240$ & $10\times10, 1\times19, 6\times40, 1\times70$ \\
        \hline
         9 & $t$ [s] & $18.27$ & $5.30$ & $33.35$ & $3.34$\\
           & l.b. & $2.81\times10^{-2}$ & $2.81\times10^{-2}$ & $^+2.80\times10^{-2}$ & $2.81\times10^{-2}$\\
           & $n$ & $1000$ & $1528$ & $790$ & $1318$ \\
           & size & $11\times11, 1\times66, 1\times297$ & $11\times11, 7\times44, 1\times66, 1\times77$ & $11\times11, 1\times21, 1\times297$ & $11\times11, 1\times21, 7\times44, 1\times77$ \\
        \hline
        10 & $t$ [s] & $127.15$ & $11.49$ & $45.10$ & $4.86$\\
           & l.b. & $2.63\times10^{-2}$ & $2.63\times10^{-2}$ & $^+2.62\times10^{-2}$ & $2.63\times10^{-2}$\\
           & $n$ & $1364$ & $2066$ & $1034$ & $1736$ \\
           & size & $12\times12, 1\times78, 1\times360$ & $12\times12, 8\times48, 1\times78, 1\times84$ & $12\times12, 1\times23, 1\times360$ & $12\times12, 1\times23, 8\times48, 1\times84$ \\
        \hline
        11 & $t$ [s] & $93.78$ & $9.26$ & $36.06$ & $6.19$\\
           & l.b. & $2.46\times10^{-2}$ & $2.46\times10^{-2}$ & $^+2.44\times10^{-2}$ & $2.46\times10^{-2}$\\
           & $n$ & $1819$ & $2729$ & $1324$ & $2234$ \\
           & size & $13\times13, 1\times91, 1\times429$ & $13\times13, 9\times52, 2\times91$ & $13\times13, 1\times25, 1\times429$ & $13\times13, 1\times25, 9\times52, 1\times91$ \\
        \hline
        12 & $t$ [s] & $145.09$ & $16.42$ & $90.36$ & $16.19$\\
           & l.b. & $2.31\times10^{-2}$ & $2.32\times10^{-2}$ & $^+2.30\times10^{-2}$ & $2.32\times10^{-2}$\\
           & $n$ & $2379$ & $3534$ & $1664$ & $2819$ \\
           & size & $14\times14, 1\times105, 1\times504$ & $14\times14, 10\times56, 1\times98, 1\times105$ & $14\times14, 1\times27, 1\times504$ & $14\times14, 1\times27, 10\times56, 1\times98$ \\
        \hline
        13 & $t$ [s] & $161.39$ & $66.75$ & $125.50$ & $18.83$\\
           & l.b. & $^+2.18\times10^{-2}$ & $2.18\times10^{-2}$ & $^+2.16\times10^{-2}$ & $2.19\times10^{-2}$\\
           & $n$ & $3059$ & $4499$ & $2058$ & $3498$ \\
           & size & $15\times15, 1\times120, 1\times585$ & $15\times15, 11\times60, 1\times105, 1\times120$ & $15\times15, 1\times29, 1\times585$ & $15\times15, 1\times29, 11\times60, 1\times105$ \\
        \hline
        14 & $t$ [s] & $373.55$ & $54.08$ & $140.63$ & $27.66$\\
           & l.b. & $^+2.06\times10^{-2}$ & $2.06\times10^{-2}$ & $^+2.03\times10^{-2}$ & $2.06\times10^{-2}$\\
           & $n$ & $3875$ & $5643$ & $2510$ & $4278$ \\
           & size & $16\times16, 1\times136, 1\times672$ & $16\times16, 12\times64, 1\times112, 1\times136$ & $16\times16, 1\times31, 1\times672$ & $16\times16, 1\times31, 12\times64, 1\times112$ \\
        \hline
        15 & $t$ [s] & $396.57$ & $96.38$ & $159.30$ & $33.31$\\
           & l.b. & $^+1.95\times10^{-2}$ & $1.95\times10^{-2}$ & $^+1.87\times10^{-2}$ & $1.95\times10^{-2}$\\
           & $n$ & $4844$ & $6986$ & $3024$ & $5166$ \\
           & size & $17\times17, 1\times153, 1\times765$ & $17\times17, 13\times68, 1\times119, 1\times153$ & $17\times17, 1\times33, 1\times765$ & $17\times17, 1\times33, 13\times68, 1\times119$ \\
              \bottomrule  
    \end{tabular}}
    \caption{Comparison of four moment-SOS hierarchies for the second-order relaxation of a double-hinged beam problem with varying number of elements ($n_\mathrm{e}$). The hierarchies are: standard moment-SOS (mSOS), mSOS with arrow decomposition (mSOS+AD), mSOS using nonmixed-term basis (mSOS+NMT), and mSOS combining nonmixed-term basis with arrow decomposition (mSOS+NMT+AD). For each method and problem size, we report: solution time ($t$) in seconds, computed lower bound (l.b.), number of variables in the relaxation ($n$), and sizes of matrix constraints (where $a\times b$ indicates $a$ matrix constraints of size $\mathbb{S}^b$). The superscript $^*$ indicates solutions with verified global optimality through relative optimality gap, while $^+$ denotes numerical issues in the Mosek optimizer.}
    \label{tab:sw2}
\end{table}

\begin{table}[!htbp]
    \centering
    \scriptsize
    \setlength{\tabcolsep}{1pt}
    \renewcommand{\arraystretch}{1.0}
     \resizebox{\textwidth}{!}{%
    \begin{tabular}{c|c|c|c|c|c}
         \toprule
        $n_e$ & & mSOS & mSOS+AD & mSOS+NMT & mSOS+NMT+AD\\
        \hline
         4 & $t$ [s] & $45.76$ & $4.01$ & $1.58$ & $0.41$\\
           & l.b. & $^*3.86\times10^{-2}$ & $^*3.86\times10^{-2}$ & $^*3.86\times10^{-2}$ & $^*3.86 \times 10^{-2}$\\
           & $n$ & $461$ & $1154$ & $240$ & $438$ \\
           & size & $6\times21, 1\times56, 1\times252$ & $6\times21, 1\times56, 2\times84, 1\times147$ & $6\times11, 1\times16, 1\times132$ & $6\times11, 1\times16, 2\times44, 1\times77$ \\
        \hline
         5 & $t$ [s] & $173.21$ & $7.85$ & $24.86$ & $0.94$\\
           & l.b. & $^*3.84\times10^{-2}$ & $^*3.84\times10^{-2}$ & $^*3.84\times10^{-2}$ & $^*3.84\times 10^{-2}$\\
           & $n$ & $923$ & $2547$ & $391$ & $755$ \\
           & size & $7\times28, 1\times84, 1\times420$ & $7\times28, 1\times84, 3\times112, 1\times196$ & $7\times13, 1\times19, 1\times195$ & $7\times13, 1\times19, 3\times52, 1\times91$ \\
        \hline
         6 & $t$ [s] & $237.84$ & $27.08$ & $49.18$ & $1.73$\\
           & l.b. & $^*3.82\times10^{-2}$ & $^*3.82\times10^{-2}$ & $^*3.82\times10^{-2}$ & $^*3.82\times10^{-2}$\\
           & $n$ & $1715$ & $5045$ & $595$ & $1195$ \\
           & size & $8\times36, 1\times120, 1\times648$ & $8\times36, 1\times120, 4\times144, 1\times252$ & $8\times15, 1\times22, 1\times270$ & $8\times15, 1\times22, 4\times60, 1\times105$ \\
        \hline
         7 & $t$ [s] & $456.53$ & $85.93$ & $129.79$ & $4.67$\\
           & l.b. & $^*3.81\times10^{-2}$ & $^*3.81\times10^{-2}$ & $^*3.81\times10^{-2}$ & $^*3.81\times10^{-2}$\\
           & $n$ & $3002$ & $9212$ & $860$ & $1778$ \\
           & size & $9\times45, 1\times165, 1\times945$ & $9\times45, 1\times165, 5\times180, 1\times315$ & $9\times17, 1\times25, 1\times357$ & $9\times17, 1\times25, 5\times68, 1\times119$ \\
        \hline
         8 & $t$ [s] & $2566.01$ & $350.88$ & $49.97$ & $20.99$\\
           & l.b. & $3.80 \times10^{-2}$ & $^*3.80\times10^{-2}$ & $3.79\times19^{-2}$ & $3.80\times10^{-2}$\\
           & $n$ & $5004$ & $15784$ & $1194$ & $2524$ \\
           & size & $10\times55, 1\times220, 1\times1320$ & $10\times55, 7\times220, 1\times385$ & $10\times19, 1\times28, 1\times456$ & $10\times19, 1\times28, 6\times76, 1\times133$ \\
        \hline
         9 & $t$ [s] & $3816.29$ & $9275.24$ & $68.91$ & $89.30$\\
           & l.b. & $3.79\times10^{-2}$ & $^*3.79\times10^{-2}$ & $3.77\times10^{-2}$ & $3.78\times10^{-2}$\\
           & $n$ & $8007$ & $25695$ & $1605$ & $3453$ \\
           & size & $11\times66, 1\times286, 1\times1782$ & $11\times66, 7\times264, 1\times286, 1\times462$ & $11\times21, 1\times31, 1\times567$ & $11\times21, 1\times31, 7\times84, 1\times147$ \\
        \hline
        10 & $t$ [s] & $9166.91$ & & $109.95$ & $164.84$\\
           & l.b. & $3.78\times10^{-2}$ & $-$& $3.74\times10^{-2}$ & $3.76\times10^{-2}$\\
           & $n$ & $12375$ & & $2101$ & $4585$ \\
           & size & $12\times78, 1\times364, 1\times2340$ & & $12\times23, 1\times34, 1\times690$ & $12\times23, 1\times34, 8\times92, 1\times161$ \\
        \hline
        11 & $t$ [s] & & & $167.05$ & $200.88$\\
           & l.b. & $-$&$-$ & $3.72\times10^{-2}$ & $3.74\times10^{-2}$\\
           & $n$ &  &  & $2690$ & $5940$ \\
           & size &  &  & $13\times25, 1\times37, 1\times825$ & $13\times25, 1\times37, 9\times100, 1\times175$ \\
        \hline
        12 & $t$ [s] & & & $402.48$ & $225.97$\\
           & l.b. & $-$&$-$ & $^+3.68\times10^{-2}$ & $3.72\times10^{-2}$\\
           & $n$ &  &  & $3380$ & $7538$ \\
           & size &  &  & $14\times27, 1\times40, 1\times972$ & $14\times27, 1\times40, 10\times108, 1\times189$ \\
        \hline
        13 & $t$ [s] & & & $303.97$ & $251.30$\\
           & l.b. &$-$ & $-$& $^+3.61\times10^{-2}$ & $3.69\times10^{-2}$\\
           & $n$ &  &  & $4179$ & $9399$ \\
           & size &  &  & $15\times29, 1\times43, 1\times1131$ & $15\times29, 1\times43, 11\times116, 1\times203$ \\
        \hline
        14 & $t$ [s] & & & $578.86$ & $328.50$\\
           & l.b. &$-$ &$-$ & $^+3.56\times10^{-2}$ & $3.66\times10^{-2}$\\
           & $n$ &  &  & $5095$ & $11543$ \\
           & size &  &  & $16\times31, 1\times46, 1\times1302$ & $16\times31, 1\times46, 12\times124, 1\times217$ \\
        \hline
        15 & $t$ [s] & & & $719.89$ & $532.08$\\
           & l.b. & $-$&$-$ & $^+3.46\times10^{-2}$ & $3.61\times10^{-2}$\\
           & $n$ &  &  & $6136$ & $13990$ \\
           & size &  &  & $17\times33, 1\times49, 1\times1485$ & $17\times33, 1\times49, 13\times132, 1\times231$ \\
              \bottomrule  
    \end{tabular}}
    \caption{Comparison of four moment-SOS hierarchies for the third-order relaxation of a double-hinged beam problem with varying number of elements ($n_\mathrm{e}$). The hierarchies are: standard moment-SOS (mSOS), mSOS with arrow decomposition (mSOS+AD), mSOS using nonmixed-term basis (mSOS+NMT), and mSOS combining nonmixed-term basis with arrow decomposition (mSOS+NMT+AD). For each method and problem size, we report: solution time ($t$) in seconds, computed lower bound (l.b.), number of variables in the relaxation ($n$), and sizes of matrix constraints (where $a\times b$ indicates $a$ matrix constraints of size $\mathbb{S}^b$). The superscript $^*$ indicates solutions with verified global optimality through relative optimality gap, while $^+$ denotes numerical issues in the Mosek optimizer, and $-$ indicates solver failure due to time limitations.}
    \label{tab:sw3}
\end{table}

\begin{table}[!htbp]
    \centering
    \scriptsize
    \setlength{\tabcolsep}{1pt}
    \renewcommand{\arraystretch}{1.0}
     \resizebox{\textwidth}{!}{%
    \begin{tabular}{c|c|c|c|c|c}
         \toprule
        $n_e$ & & mSOS & mSOS+AD & mSOS+NMT & mSOS+NMT+AD\\
        \hline
         8 & $t$ [s] & $200931.39$ & & $88.55$ & $36.43$\\
           & l.b. & $3.80\times10^{-2}$ & $-$ & $^+3.80\times10^{-2}$ & $^*3.80\times10^{-2}$\\
           & $n$ & $24309$ &  & $3192$ & $6034$ \\
           & size & $10\times220, 1\times715, 1\times5280$ &  & $10\times28, 1\times37, 1\times672$ & $10\times28, 1\times37, 6\times112, 1\times196$ \\
        \hline
         9 & $t$ [s] & & & $142.16$ & $901.23$\\
           & l.b. & $-$& $-$& $^+3.79\times10^{-2}$ & $^*3.79\times10^{-2}$\\
           & $n$ &  &  & $4370$ & $8338$ \\
           & size &  &  & $11\times31, 1\times41, 1\times837$ & $11\times31, 1\times41, 7\times124, 1\times217$ \\
        \hline
        10 & $t$ [s] & & & $685.56$ & $574.21$\\
           & l.b. & $-$&$-$ & $3.79\times10^{-2}$ & $^*3.79\times10^{-2}$\\
           & $n$ &  &  & $5808$ & $11163$ \\
           & size &  &  & $12\times34, 1\times45, 1\times1020$ & $12\times34, 1\times45, 8\times136, 1\times238$ \\
        \hline
        11 & $t$ [s] & & & $936.62$ & $493.89$\\
           & l.b. &$-$ &$-$ & $3.78\times10^{-2}$ & $3.79\times10^{-2}$\\
           & $n$ &  &  & $7532$ & $14562$ \\
           & size &  &  & $13\times37, 1\times49, 1\times1221$ & $13\times37, 1\times49, 9\times148, 1\times259$ \\
        \hline
        12 & $t$ [s] & & & $1897.92$ & $1004.73$\\
           & l.b. & $-$& $-$& $3.77\times10^{-2}$ & $3.78\times10^{-2}$\\
           & $n$ &  &  & $9568$ & $18588$ \\
           & size &  &  & $14\times40, 1\times53, 1\times1440$ & $14\times40, 1\times53, 10\times160, 1\times280$ \\
        \hline
        13 & $t$ [s] & & & $1377.81$ & $1869.33$ \\
           & l.b. & $-$& $-$ & $3.76\times10^{-2}$ & $3.78\times10^{-2}$\\
           & $n$ &  &  & $11942$ & $23294$ \\
           & size &  &  & $15\times43, 1\times57, 1\times1677$ & $15\times43, 1\times57, 11\times172, 1\times301$ \\
        \hline
        14 & $t$ [s] & & & $2785.38$ & $4305.19$\\
           & l.b. &$-$ & $-$& $3.76\times10^{-2}$ & $3.77\times10^{-2}$\\
           & $n$ &  &  & $14680$ & $28733$ \\
           & size &  &  & $16\times46, 1\times61, 1\times1932$ & $16\times46, 1\times61, 12\times184, 1\times322$ \\
        \hline
        15 & $t$ [s] & & & $6667.30$ & $6879.09$\\
           & l.b. & $-$& $-$& $3.74\times10^{-2}$ & $3.76\times10^{-2}$\\
           & $n$ &  &  & $17808$ & $34958$ \\
           & size &  &  & $17\times49, 1\times65, 1\times2205$ & $17\times49, 1\times65, 13\times196, 1\times343$ \\
              \bottomrule  
    \end{tabular}}
    \caption{Comparison of four moment-SOS hierarchies for the fourth-order relaxation of a double-hinged beam problem with varying number of elements ($n_\mathrm{e}$). The hierarchies are: standard moment-SOS (mSOS), mSOS with arrow decomposition (mSOS+AD), mSOS using nonmixed-term basis (mSOS+NMT), and mSOS combining nonmixed-term basis with arrow decomposition (mSOS+NMT+AD). For each method and problem size, we report: solution time ($t$) in seconds, computed lower bound (l.b.), number of variables in the relaxation ($n$), and sizes of matrix constraints (where $a\times b$ indicates $a$ matrix constraints of size $\mathbb{S}^b$). The superscript $^*$ indicates solutions with verified global optimality through relative optimality gap, while $^+$ denotes numerical issues in the Mosek optimizer, and $-$ indicates solver failure due to time limitations.}
    \label{tab:sw4}
\end{table}

\newpage
\paragraph{Acknowledgements}
Marouan Handa, Marek Tyburec and Michal Ko\v{c}vara acknowledge the financial support of the Czech Science Foundation project GA22-15524S, and Marek Tyburec and Michal Ko\v{c}vara also benefited from the support of the project ROBOPROX (reg. no. CZ.02.01.01/00/22\_008/0004590) that was co-funded by the European Union. In addition, Marek Tyburec and Giovanni Fantuzzi acknowledge the support of the mobility project 8J24DE005 awarded jointly by Ministry of Education, Youth and Sport of the Czech Republic and German Academic Exchange Service. 
Victor Magron benefited from the HORIZON–MSCA-2023-DN-JD of the European Commission under the Grant Agreement No 101120296 (TENORS), the ANITI AI Cluster program under the Grant agreement n${}^\circ$ ANR-23-IACL-0002, as well as the National Research Foundation, Prime Minister's Office, Singapore under its Campus for Research Excellence and Technological Enterprise (CREATE) programme.\\ \\

\noindent\textbf{Availability of data and materials} Source codes and input files to reproduce the computations are available at https://gitlab.com/tyburec/pof-dyna.
\section*{Declarations}
\textbf{Ethics approval and consent to participate} This work does not contain any studies with human participants or animals performed by any of the authors.\\ \\
\textbf{Consent for publication} Not applicable. \\ \\
\textbf{Competing interests} The authors declare that they have no competing interests.

\input{ADmSOS.bbl}

\end{document}